\documentclass[12pt]{article}
\usepackage{amsmath, amssymb,amsfonts,bbm,verbatim,multirow,color,xcolor}
\usepackage{epic,eepic,psfrag,epsfig}
\usepackage[sf,bf,SF,footnotesize]{subfigure}
\usepackage{graphicx}
\usepackage{amsthm,breakcites}
\usepackage{amscd}
\usepackage{epsfig}
\usepackage{fullpage}
\usepackage{tikz}
\tikzstyle{edge} = [fill,opacity=.5,fill opacity=.5,line cap=round, line join=round, line width=50pt]
\pgfdeclarelayer{background}
\pgfsetlayers{background,main}
\usetikzlibrary{calc}
\usetikzlibrary{arrows.meta}

\usepackage[round]{natbib}

\usepackage{setspace}
\usepackage{enumerate}
\usepackage{array}
\usepackage[small]{caption}
\RequirePackage[colorlinks,citecolor=blue,urlcolor=blue,breaklinks]{hyperref}

\newtheorem{thm}{Theorem}

\newtheorem{lemma}[thm]{Lemma}
\newtheorem{example}[thm]{Example}

\newtheorem{prop}[thm]{Proposition}

\def\hat{\widehat}

\allowdisplaybreaks

\begin{document}
\title{Optimal nonparametric testing of Missing Completely At Random, and its connections to compatibility}
\author{Thomas B. Berrett and Richard J. Samworth\\University of Warwick and University of Cambridge}
\date{\today}

\maketitle

\begin{abstract}
    Given a set of incomplete observations, we study the nonparametric problem of testing whether data are Missing Completely At Random (MCAR).  Our first contribution is to characterise precisely the set of alternatives that can be distinguished from the MCAR null hypothesis.  This reveals interesting and novel links to the theory of Fr\'echet classes (in particular, compatible distributions) and linear programming, that allow us to propose MCAR tests that are consistent against all detectable alternatives.  We define an incompatibility index as a natural measure of ease of detectability, establish its key properties, and show how it can be computed exactly in some cases and bounded in others. Moreover, we prove that our tests can attain the minimax separation rate according to this measure, up to logarithmic factors.  Our methodology does not require any complete cases to be effective, and is available in the \texttt{R} package \texttt{MCARtest}. 
    
\end{abstract}

\section{Introduction}

Over the last century, a plethora of algorithms have been proposed to address specific statistical challenges; in many cases these can be justified under modelling assumptions on the underlying data generating mechanism.  When faced with a data set and a question of interest, the practitioner needs to assess the validity of the assumptions underpinning these statistical models, in order to determine whether or not they can trust the output of the method.  Experienced practitioners recognise that mathematical assumptions can rarely be expected to hold exactly, and develop intuition (sometimes backed up by formal tests) about the seriousness of different violations.

One of the most commonly-encountered discrepancies between real data sets and models hypothesised in theoretical work is that of missing data.  In fact, missingness may be even more serious than many other types of departure from a statistical model, in that it may be impossible even to run a particular algorithm without modification when data are missing.  Once it is accepted that methods for dealing with missing data are essential, the primary concern is to understand the relationship between the data generating and missingness mechanisms.  In the ideal situation, these two sources of randomness are independent, a setting known as Missing Completely At Random (MCAR).  When this assumption holds, the analysis becomes much easier, because we can regard our observed data as a representative sample from the wider population. For instance, theoretical guarantees have recently been established in the MCAR setting for a variety of modern statistical problems, including high-dimensional regression \citep{loh2012high}, high-dimensional or sparse principal component analysis \citep{zhu2019high,elsener2019sparse}, classification \citep{cai2019high}, and precision matrix and changepoint estimation \citep{loh2018high,follain2022high}. The failure of this assumption, on the other hand, may introduce significant bias and necessitate further investigation of the nature of the dependence between the data and the missingness \citep{davison2003statistical,little2019statistical}.

Our aim in this work is to study the fundamental problem of testing the null hypothesis that data are MCAR.  It is important to recognise from the outset that in general there will exist alternatives (i.e.~joint distributions of data and missingness that do not satisfy the MCAR hypothesis) for which no test could have power greater than its size. Indeed, to give a toy example, if $X_1,\ldots,X_n \stackrel{\mathrm{iid}}{\sim} N(0,1)$, but we only observe those $X_i$ that are non-negative, then the joint distribution of our data is indistinguishable from the MCAR setting where $X_1,\ldots,X_n$ are a random sample from the folded normal distribution on $[0,\infty)$, and each $X_i$ is observed independently with probability $1/2$.

The first main contribution of this work, then, is to determine precisely the set of alternatives that are distinguishable from our null hypothesis. Surprisingly, this question turns out to be relevant in several different subject areas, namely copula theory \citep{nelsen2007introduction,dall2012advances}, portfolio risk management \citep{embrechts2010bounds,ruschendorf2013mathematical}, coalition games \citep{vorob1962consistent}, quantum contextuality \citep{bell1966problem,clauser1978bell} and relational databases \citep{maier1983theory}.  To describe our results briefly, we introduce the notation that when a random vector $X$ takes values in a measurable space of the form $\mathcal{X} = \prod_{j=1}^d \mathcal{X}_j$ and when $S \subseteq \{1,\ldots,d\}$, we write $X_S:=(X_j : j \in S)$ and $\mathcal{X}_S := \prod_{j \in S} \mathcal{X}_j$.  Following, e.g.,~\citet[][Section~3]{joe1997multivariate}, given a collection $\mathbb{S}$ of subsets of $\{1,\ldots,d\}$, and a collection of distributions $P_{\mathbb{S}} := (P_S:S \in \mathbb{S})$, we define their \emph{Fr\'echet class} as the set of all distributions of $X$ for which $X_S$ has marginal distribution $P_S$ for all $S \in \mathbb{S}$.  We say that a collection of marginal distributions $P_{\mathbb{S}}$ is \emph{compatible} when the corresponding Fr\'echet class is non-empty.  In Section~\ref{Sec:Frechet}, we prove that it is only possible to detect that a joint distribution does not satisfy the MCAR hypothesis if the marginal distributions for which we have simultaneous observations are incompatible.   

Our second contribution, in Section~\ref{Sec:TestingCompatibility}, is to introduce a  general new test of the null hypothesis of compatibility, and consequently (by our result in Section~\ref{Sec:Frechet}) the MCAR hypothesis, in the discrete case.  We prove that it has finite-sample Type~I error control, and is consistent against all incompatible alternatives.  These results therefore describe precisely what can be learnt about the plausibility of the MCAR hypothesis from data.  Our methodology is based on a duality theorem due to \citet{kellerer1984duality} that gives a characterisation of compatibility, and allows us to define a notion of an \emph{incompatibility index}, denoted $R(P_\mathbb{S})$.  Although the result itself is rather abstract, we show how it motivates a test statistic that can be computed straightforwardly using linear programming.  We further argue that a more specific and involved analysis can lead to improved tests in certain cases.  For instance, when $d=3$, with $\mathbb{S} = \bigl\{\{1,2\},\{2,3\},\{1,3\}\bigr\}$ and $|\mathcal{X}_1| = r$, $|\mathcal{X}_2|=s$ and $|\mathcal{X}_3|=2$, we show by means of a minimax lower bound (Theorem~\ref{Prop:rs2lowerbound}) that our improved test achieves the optimal separation rate in $R(P_\mathbb{S})$ simultaneously in $r$, $s$ and the sample sizes for each observation pattern, up to logarithmic factors.

The form of the incompatibility index is a supremum of a class of linear functionals, and exact expressions can become complicated as $|\mathbb{S}|$ and the alphabet sizes increase.  In Section~\ref{Sec:Computation}, we describe computational geometry algorithms that yield analytic expressions for $R(P_\mathbb{S})$; code is available in the \texttt{R} package \texttt{MCARtest} \citep{berrett2022MCARtest}, and in principle, these can be applied for arbitrary $\mathbb{S}$.  As illustrations, we provide examples with binary variables, where these expressions are more tractable.  Moreover, as we show in Section~\ref{Sec:Reductions}, in some cases we can exploit the structure of $\mathbb{S}$ to reduce the computation of $R(P_\mathbb{S})$ to the computation of the analogous quantity for lower-dimensional settings, or at least to bound it in terms of these simpler quantities.

In Section~\ref{Sec:Continuous}, we explain how the methodology and theory described above extends to continuous data, or to variables having both continuous and discrete components.  Here we have an additional approximation error in the minimax separation radius that depends on the smoothness of the densities of the continuous components.  Section~\ref{Sec:Numerics} is devoted to a numerical study of a Monte Carlo version of our test, which uses bootstrap samples to generate the critical value.  We find that this version also provides good Type I error control, and outperforms a test due to \citet{fuchs1982maximum} even when this latter test is provided with additional complete cases (which are required for its application).  Proofs of all of our results, as well as auxiliary results, are deferred to Section~\ref{Sec:Proofs}.

Our theory is based on the study of marginal polytopes, which is a topical problem in convex geometry  \citep{vlach1986conditions,wainwright2008graphical,deza2009geometry}.  Indeed, these polytopes are known to be extremely complicated \citep{de2014combinatorics}, but are of considerable interest in hierarchical log-linear models \citep{eriksson2006polyhedral}, variational inference \citep{wainwright2003variational}, classical transportation  \citep{kantorovich1942translocation} (reprinted as \citet{kantorovich2006translocation}) and max flow-min cut problems \citep{gale1957theorem}. In the special case where all variables are binary, marginal polytopes are equivalent to \emph{correlation polytopes} or \emph{cut polytopes}, which have been heavily studied in their own right \citep{deza2009geometry,coons2020generalized}. In statistical contexts, recent work on hypothesis testing over polyhedral parameter spaces has sought to elucidate the link between the difficulty of the problem and the underlying geometry \citep{blanchard2018minimax,wei2019geometry}.

Most prior work on testing the MCAR hypothesis has been developed within the context of parametric models such as multivariate normality \citep{little1988test,kim2002tests,jamshidian2010tests}, Poisson or multinomial contingency tables with at least some complete cases \citep{fuchs1982maximum} or generalised estimating equations \citep{chen1999test,qu2002testing}.  \citet{li2015nonparametric} study the nonparametric problem of testing whether or not a family of marginal distributions $P_{\mathbb{S}}$ is \emph{consistent}, i.e.~whether, for each $S,S' \in \mathbb{S}$ with $S \cap S' \neq \emptyset$, the marginal distributions of $P_S$ and $P_{S'}$ on the coordinates in $S \cap S'$ agree with each other. \cite{michel2021pklm} consider an equivalent problem, using random forest classification methods to test equalities of distributions.  Consistency is a necessary, but not sufficient, condition for compatibility\footnote{However, in the special case where $[d] \in \mathbb{S}$, a necessary and sufficient condition for compatibility is that $P_S$ is the marginal distribution on $\mathcal{X}_S$ of $P_{[d]}$, for each $S \in \mathbb{S} \setminus \{[d]\}$.  In other words, in this case, consistency is sufficient for compatibility.  A test of compatibility may therefore then be constructed by testing each of these hypotheses via $|\mathbb{S}|-1$ two-sample tests and applying, e.g., a Bonferroni correction.  More generally, this strategy may be applied whenever $\mathbb{S}$ is \emph{decomposable} \citep{lauritzen1984decomposable,lauritzen1988local}.}.  To the best of our knowledge, the current paper is the first  both to characterise the set of detectable alternatives to the MCAR hypothesis, and to provide tests that have asymptotic power 1 against all such detectable alternatives while controlling the Type I error.

We conclude this introduction with some notation that is used throughout the paper.  For $d \in \mathbb{N}$, we write $[d] := \{1,\ldots,d\}$, and also define $[\infty] := \mathbb{N}$.  Given a countable set~$\Omega$, we write $2^\Omega$ for its power set, and $1_\Omega$ for the vector of ones indexed by the elements of~$\Omega$.  If $S \subseteq [d]$, we denote $\mathbbm{1}_S := (\mathbbm{1}_{\{j \in S\}})_{j \in [d]} \in \{0,1\}^d$. For $x = (x_1,\ldots,x_d)^T \in \mathbb{R}^d$, write $x_S = (x_j : j \in S)$.  For $x \in \mathbb{R}$, let $x_+ := \max(x,0)$ and $x_- := \max(-x,0)$.  Given $a, b \geq 0$, we write $a \lesssim b$ to mean that there exists a universal constant $C > 0$ such that $a \leq Cb$, and, for a generic quantity~$x$, write $a \lesssim_x b$ to mean that there exists $C$, depending only on $x$, such that $a \leq Cb$.  We also write $a \asymp b$ to mean $a \lesssim b$ and $b \lesssim a$.  For random elements $X, Y$, we write $X \perp \!\!\! \perp Y$ to mean that $X$ and $Y$ are independent.  For probability measures $P,Q$ on a measurable space $(\mathcal{Z},\mathcal{C})$, we denote their total variation distance as $d_{\mathrm{TV}}(P,Q) := \sup_{C \in \mathcal{C}} |P(C) - Q(C)|$.


\section{Fr\'echet classes and non-detectable alternatives}
\label{Sec:Frechet}

We begin with a brief discussion of Fr\'echet classes, for which a good reference is~\citet[][Section~3]{joe1997multivariate}, as this will allow us to characterise the set of detectable alternatives of an MCAR test.  Throughout the paper, for $d \in \mathbb{N}$ and measurable topological spaces $(\mathcal{X}_1,\mathcal{A}_1),\ldots,(\mathcal{X}_d,\mathcal{A}_d)$, we let $\mathcal{X} := \prod_{j=1}^d \mathcal{X}_j$.   Given a collection~$\mathbb{S}$ of subsets of $[d]$ and a set of marginal distributions~$P_{\mathbb{S}} = (P_S:S\in \mathbb{S})$, where $P_S$ is defined on $\mathcal{X}_S$, we write $\mathcal{F}_{\mathbb{S}}(P_{\mathbb{S}})$ for the corresponding Fr\'echet class.  As a simple example, if $\mathbb{S}=\bigl\{\{1\}, \ldots, \{d\}\bigr\}$, then $\mathcal{F}_{\mathbb{S}}(P_{\mathbb{S}})$ is the class of all joint distributions with specified marginals $P_{\{1\}},\ldots P_{\{d\}}$. It is easy to see that this Fr\'echet class in non-empty, because it includes the product distribution $P_{\{1\}} \times \ldots \times P_{\{d\}}$. More generally, if $S_1,\ldots,S_m$ is a partition of $[d]$ and $\mathbb{S} = \{S_1,\ldots,S_m\}$, then $\mathcal{F}_{\mathbb{S}}(P_{\mathbb{S}})$ contains the corresponding product distribution. However, when $\mathbb{S}$ contains subsets that overlap, the Fr\'echet class $\mathcal{F}_{\mathbb{S}}(P_{\mathbb{S}})$ may be empty, or equivalently, $P_\mathbb{S}$ may be incompatible.  One simple way in which this may occur is if $d=2$ and $\mathbb{S} = \bigl\{\{1\},\{1,2\}\bigr\}$, but $P_{\{1\}}$ and $P_{\{1,2\}}$ are not consistent.  More interestingly, when $d \geq 3$ it may be the case that $P_\mathbb{S}$ is consistent but we still have $\mathcal{F}_{\mathbb{S}}(P_{\mathbb{S}}) = \emptyset$.  For instance when $d=3$ and $\mathbb{S}=\bigl\{\{1,2\},\{1,3\},\{2,3\} \bigr\}$, let $\rho_{23}=\rho_{13} = 2^{-1/2}$, let $\rho_{12}=-2^{-1/2}$ and, for $1 \leq i<j \leq 3$, let 
\[
    P_{\{i,j\}} = N \biggl( \begin{pmatrix} 0 \\ 0 \end{pmatrix}, \begin{pmatrix} 1 & \rho_{ij} \\ \rho_{ij} & 1 \end{pmatrix} \biggr).
\]
Then any joint distribution $P_{\{1,2,3\}}$ with these marginals would have `covariance matrix'
\[
    \begin{pmatrix} 1 & -2^{-1/2} & 2^{-1/2} \\ -2^{-1/2} & 1 & 2^{-1/2} \\ 2^{-1/2} & 2^{-1/2} & 1 \end{pmatrix},
\]
which has a negative eigenvalue.

We are now in a position to describe the main statistical question that motivates our work.  Given $x = (x_1,\ldots,x_d) \in \mathcal{X}$ and $\omega = (\omega_1,\ldots,\omega_d) \in \{0,1\}^d$, we write $x \circ \omega$ for the element of $\prod_{j=1}^d (\mathcal{X}_j \cup \{\star\})$ that has $j$th entry $x_j$ if $\omega_j =1$ and $j$th entry $\star$, denoting a missing value, if $\omega_j=0$.  Assume that we are given $n$ independent copies of $X \circ \Omega$, where the pair $(X, \Omega)$ takes values in $\mathcal{X} \times \{0,1\}^d$, and wish to test the hypothesis $H_0: X \perp \!\!\! \perp \Omega$, i.e.~that entries of $X$ are MCAR. This can be thought of as an independence test where we do not have complete observations, though we will see that the missingness leads to very different phenomena.  

Let $\mathbb{S} := \bigl\{S \subseteq [d] : \mathbb{P}( \Omega = \mathbbm{1}_S) > 0\bigr\}$ denote the set of all missingness patterns that could be observed.  Writing $P_S$ for the conditional distribution of $X_S$ given that $\Omega = \mathbbm{1}_S$, note that if our data are MCAR, then $P_{\mathbb{S}} := (P_S:S \in \mathbb{S})$ is compatible, because the Fr\'echet class $\mathcal{F}_{\mathbb{S}}(P_{\mathbb{S}})$ contains the distribution of $X$.   

On the other hand, suppose now that our data are not MCAR, but that $P_{\mathbb{S}}$ is still compatible.  If $\tilde{X}$ denotes a random vector, independent of $\Omega$, whose distribution lies in the Fr\'echet class $\mathcal{F}_{\mathbb{S}}(P_{\mathbb{S}})$, then $\tilde{X} \circ \Omega \stackrel{d}{=} X \circ \Omega$, so no test of $H_0$ can have power at compatible alternatives that is greater than its size.  The conclusion of this discussion is stated in Proposition~\ref{Thm:Compatibility} below, where we let~$\Psi$ denote the set of all (randomised) tests based on our observed data $X_1 \circ \Omega_1, \ldots, X_n \circ \Omega_n$, i.e.~the set of Borel measurable functions $\psi: \bigl(\prod_{j=1}^d (\mathcal{X}_j \cup \{\star\})\bigr)^n \rightarrow [0,1]$.
\begin{prop}
\label{Thm:Compatibility}
Let $\mathcal{P}_0$ denote the set of distributions on $\mathcal{X} \times \{0,1\}^d$ that satisfy $H_0$, and let $\mathcal{P}_0'$ denote the set of distributions on $\mathcal{X} \times \{0,1\}^d$ for which the corresponding sequence of conditional distributions $P_{\mathbb{S}}$ is compatible.  Then $\mathcal{P}_0 \subseteq \mathcal{P}_0'$, but for any $\psi \in \Psi$, we have
\[
\sup_{P \in \mathcal{P}_0'} \mathbb{E}_P\psi(X_1 \circ \Omega_1,\ldots,X_n \circ \Omega_n) = \sup_{P \in \mathcal{P}_0} \mathbb{E}_P\psi(X_1 \circ \Omega_1,\ldots,X_n \circ \Omega_n).
\]
\end{prop}
A consequence of Proposition~\ref{Thm:Compatibility} is that it is only possible to have non-trivial power against incompatible alternatives to $H_0$, and a search for optimal tests of the MCAR property may be reduced to looking for optimal tests of compatibility.  In subsequent sections, we will construct tests of compatibility, noting that if such a test rejects the null hypothesis, then we can also reject the hypothesis of MCAR. 

\section{Testing compatibility}
\label{Sec:TestingCompatibility}

Let $\mathcal{P}_{\mathbb{S}}$ denote the set of sequences of the form $P_{\mathbb{S}} = (P_S:S \in \mathbb{S})$, where $P_S$ is a distribution on~$\mathcal{X}_S$, and let $\mathcal{P}_{\mathbb{S}}^0$ denote the subset of $\mathcal{P}_{\mathbb{S}}$ consisting of those $P_{\mathbb{S}}$ that are compatible.  In testing compatibility, it is convenient to alter our model very slightly, so that we have deterministic sample sizes within each observation pattern.  More precisely, given a collection~$\mathbb{S} \subseteq [d]$ and $P_\mathbb{S} = (P_S:S \in \mathbb{S}) \in \mathcal{P}_\mathbb{S}$, we assume that we are given independent data $(X_{S,i})_{S \in \mathbb{S},i \in [n_S]}$, where $X_{S,1},\ldots,X_{S,n_S} \sim P_S$.   Our goal is to propose a test of $H_0': P_{\mathbb{S}} \in \mathcal{P}_{\mathbb{S}}^0$, or equivalently, $H_0': \mathcal{F}_{\mathbb{S}}(P_{\mathbb{S}}) \neq \emptyset$.  To this end, for $S \in \mathbb{S}$, let $\mathcal{G}_S^*$ denote the set of all bounded, upper semi-continuous functions on $\mathcal{X}_S$.  We will exploit the characterisation of~\citet[][Proposition~3.13]{kellerer1984duality}, which states that $P_{\mathbb{S}} \in \mathcal{P}_{\mathbb{S}}^0$ if and only if
\begin{equation}
\label{Eq:Characterisation}
    \sum_{S \in \mathbb{S}} \int_{\mathcal{X}_S} f_S(x_S) \,dP_S(x_S) \geq 0 \quad \text{for all  } (f_S : S \in \mathbb{S})  \in \prod_{S \in \mathbb{S}} \mathcal{G}_{S}^* \text{ with } \inf_{x \in \mathcal{X}} \sum_{S \in \mathbb{S}} f_S(x_S) \geq 0.
\end{equation}
This duality theorem can be regarded as a potentially infinite-dimensional generalisation of Farkas's lemma \citep{farkas1902theorie}, which underpins the theory of linear programming.

We now show how~\eqref{Eq:Characterisation} can be used to define a quantitative measure of incompatibility.  For $S \in \mathbb{S}$, let $\mathcal{G}_S$ denote the subset of $\mathcal{G}_S^*$ consisting of functions taking values in $[-1,\infty)$, and let $\mathcal{G}_{\mathbb{S}} := \prod_{S \in \mathbb{S}} \mathcal{G}_S$.  Given $f_S \in \mathcal{G}_S$ for each $S \in \mathbb{S}$, we write $f_\mathbb{S} := (f_S:S \in \mathbb{S}) \in \mathcal{G}_{\mathbb{S}}$.  Now let
\[
    \mathcal{G}_{\mathbb{S}}^+ := \biggl\{ f_{\mathbb{S}} \in \mathcal{G}_{\mathbb{S}}: \inf_{x \in \mathcal{X}} \sum_{S \in \mathbb{S}} f_S(x_S) \geq 0 \biggr\}.
\]
Our key \emph{incompatibility index}, then, is
\begin{equation}
\label{Eq:RPS}
R(P_{\mathbb{S}}):= \sup_{f_{\mathbb{S}} \in \mathcal{G}_\mathbb{S}^+} R(P_{\mathbb{S}},f_{\mathbb{S}}),
\end{equation}
where 
\[
    R(P_{\mathbb{S}},f_{\mathbb{S}}) := -\frac{1}{|\mathbb{S}|} \sum_{S \in \mathbb{S}} \int_{\mathcal{X}_S} f_S(x_S) \, dP_S(x_S).
\]
Since the choice $f_S \equiv 0$ for all $S \in \mathbb{S}$ means that the corresponding $f_{\mathbb{S}}$ belongs to $\mathcal{G}_{\mathbb{S}}^+$, we see from~\eqref{Eq:Characterisation} that $R(P_{\mathbb{S}}) = 0$ if $P_{\mathbb{S}} \in \mathcal{P}_{\mathbb{S}}^0$.  Moreover, if~\eqref{Eq:Characterisation} is violated by some $f_{\mathbb{S}} \in \mathcal{G}_S^*$ with $\inf_{x \in \mathcal{X}} \sum_{S \in \mathbb{S}} f_S(x_S) \geq 0$, then by scaling we may assume that $f_\mathbb{S} \in \mathcal{G}_\mathbb{S}^+$, and hence $R(P_{\mathbb{S}}) > 0$ whenever $P_{\mathbb{S}} \notin \mathcal{P}_{\mathbb{S}}^0$. Finally, observe that we also have $R(P_\mathbb{S}) \leq 1$ for all $P_\mathbb{S} \in \mathcal{P}_\mathbb{S}$; the extreme case $R(P_\mathbb{S}) = 1$ corresponds to \emph{strongly contextual} families of distributions, in the terminology of quantum contextuality \citep{abramsky2011sheaf}. When $|\mathcal{X}| < \infty$, we see from Theorem~\ref{Thm:DualRepresentation} below that $R(P_\mathbb{S}) <1$ if and only if there exists $x \in \mathcal{X}$ with $P_S(\{x_S\})>0$ for all $S \in \mathbb{S}$. 


\begin{thm}
\label{Thm:DualRepresentation}
Suppose that $\mathcal{X}_j$ is a locally compact Hausdorff space\footnote{A brief glossary of definitions of topological and measure-theoretic concepts used in this result and its proof is provided in Section~\ref{Sec:Glossary} for the reader's convenience.}, for each $j \in [d]$, and that every open set in $\mathcal{X}$ is $\sigma$-compact.  Then for any $P_\mathbb{S} \in \mathcal{P}_\mathbb{S}$, 
\begin{equation}
    \label{Eq:InfR}
    R(P_\mathbb{S}) = \inf \bigl\{ \epsilon \in [0,1] : P_\mathbb{S} \in (1-\epsilon) \mathcal{P}_\mathbb{S}^0 + \epsilon \mathcal{P}_\mathbb{S}\bigr\} = 1 - \sup \bigl\{ \epsilon \in [0,1] : P_\mathbb{S} \in \epsilon \mathcal{P}_\mathbb{S}^0 + (1-\epsilon) \mathcal{P}_\mathbb{S}\bigr\}.
\end{equation}
\end{thm}
\noindent \textbf{Remark}: If $\mathcal{X}$ is second countable, then every open set in $\mathcal{X}$ is $\sigma$-compact.

Theorem~\ref{Thm:DualRepresentation} can be regarded as providing a dual representation for $R(P_{\mathbb{S}})$.  In the quantum physics literature and for consistent families of distributions on discrete spaces, the second and third expressions in~\eqref{Eq:InfR} are known as the \emph{contextual fraction} \citep{abramsky2017contextual}.  The first step of the proof of Theorem~\ref{Thm:DualRepresentation} is to apply the idea of \emph{Alexandroff (one-point) compactification}  \citep{alexandroff1924metrisation} to reduce the problem to compact Hausdorff spaces.  Strong duality for linear programming \citep[][Theorem~2.3]{isii1964inequalities}, combined with the Riesz representation theorem for compact spaces, then allows us to deduce the result.

 With our incompatibility index now defined, we can now introduce the minimax framework for our hypothesis testing problem.  Writing $n_{\mathbb{S}} := (n_S:S \in \mathbb{S}) \in \mathbb{N}^\mathbb{S}$, a test of $H_0'$ is a measurable function $\psi'_{n_{\mathbb{S}}}:\prod_{S \in \mathbb{S}} \mathcal{X}_S^{n_S} \rightarrow [0,1]$, and we write $\Psi'_{n_{\mathbb{S}}}$ for the set of all such tests.  Given $\rho \geq 0$, it is convenient to write
\[
\mathcal{P}_\mathbb{S}(\rho) := \{P_\mathbb{S} \in \mathcal{P}_\mathbb{S}:R(P_\mathbb{S}) \geq \rho\},
\]
so that $\mathcal{P}_\mathbb{S}(0) = \mathcal{P}_\mathbb{S}$, $\mathcal{P}_\mathbb{S}^0 = \mathcal{P}_\mathbb{S} \setminus \cup_{\epsilon > 0} \mathcal{P}_\mathbb{S}(\epsilon)$ and $\mathcal{P}_\mathbb{S}(\epsilon) = \emptyset$ for $\epsilon > 1$.  
The minimax risk at separation $\rho$ in this problem is defined as
\[
\mathcal{R}(n_\mathbb{S},\rho) := \inf_{\psi'_{n_{\mathbb{S}}} \in \Psi'_{n_{\mathbb{S}}}}\biggl\{ \sup_{P_\mathbb{S} \in \mathcal{P}_\mathbb{S}^0} \mathbb{E}_{P_\mathbb{S}}(\psi'_{n_{\mathbb{S}}}) + \sup_{P_\mathbb{S} \in \mathcal{P}_\mathbb{S}(\rho) }\mathbb{E}_{P_\mathbb{S}}(1-\psi'_{n_{\mathbb{S}}})\biggr\};
\]
thus $\mathcal{R}(n_\mathbb{S},\rho) = 0$ for $\rho > 1$.  Finally, the minimax testing radius is defined as
\[
\rho^*(n_{\mathbb{S}}) := \inf\bigl\{\rho \geq 0: \mathcal{R}(n_\mathbb{S},\rho) \leq 1/2\},
\]
so that $\rho^*(n_{\mathbb{S}}) \leq 1$.

\subsection{A universal test in the discrete case}
\label{Sec:Discrete}

In this subsection, we will assume that $\mathcal{X}_j = [m_j]$ for every $j \in [d]$, where $m_j \in \mathbb{N}$.  Given our data, for each $S \in \mathbb{S}$ and $A_S \in 2^{\mathcal{X}_S}$, define the empirical distribution of $(X_{S,i})_{i \in [n_S]}$ by
\[
    \hat{P}_S(A_S):= \frac{1}{n_S} \sum_{i=1}^{n_S} \mathbbm{1}_{\{X_{S,i} \in A_S\}}
\]
and write $\hat{P}_{\mathbb{S}} := (\hat{P}_S : S \in \mathbb{S})$.  We propose to reject $H_0'$ at the significance level $\alpha \in (0,1)$ if $R(\hat{P}_{\mathbb{S}}) \geq C_\alpha$, where
\begin{align*}
    C_\alpha&:= \frac{1}{2} \sum_{S \in \mathbb{S}} \Bigl( \frac{|\mathcal{X}_S|-1}{n_S} \Bigr)^{1/2} +  \biggl\{ \frac{1}{2} \log(1/\alpha) \sum_{S \in \mathbb{S}} \frac{1}{n_S} \biggr\}^{1/2}.  
\end{align*}
The following proposition provides size and power guarantees for this test. 
\begin{prop}
\label{Prop:DiscreteTest1}
 Fix $\alpha, \beta \in (0,1)$.  Whenever $P_{\mathbb{S}} = (P_S:S \in \mathbb{S}) \in \mathcal{P}_{\mathbb{S}}^0$, we have $\mathbb{P}_{P_{\mathbb{S}}}\bigl(R(\hat{P}_{\mathbb{S}}) \geq C_\alpha\bigr) \leq \alpha$. Moreover, 
 for any $P_{\mathbb{S}} \in \mathcal{P}_\mathbb{S}$ satisfying
\[
    R(P_{\mathbb{S}}) \geq  C_\alpha + C_\beta,
    \]
we have $\mathbb{P}_{P_{\mathbb{S}}}\bigl(R(\hat{P}_{\mathbb{S}}) \geq  C_\alpha\bigr) \geq 1-\beta$.
\end{prop}
Proposition~\ref{Prop:DiscreteTest1} reveals in particular that in addition to having guaranteed finite-sample size control, our test is consistent against any fixed, incompatible alternative; in other words, whenever $R(P_{\mathbb{S}}) > 0$, we have $\mathbb{P}_{P_{\mathbb{S}}}\bigl(R (\hat{P}_{\mathbb{S}}) \geq C_\alpha\bigr) \rightarrow 1$ as $\min_{S \in \mathbb{S}} n_S \rightarrow \infty$.  In combination with Proposition~\ref{Thm:Compatibility}, then, we see that from a testing perspective, compatibility is the right proxy for MCAR, in that distributions of $(X,\Omega)$ that do not satisfy the MCAR hypothesis are detectable if and only if their observed margins are incompatible. Moreover, we have the following upper bound on the minimax separation rate:
\[
    \rho^*(n_\mathbb{S}) \leq \sum_{S \in \mathbb{S}} \Bigl( \frac{|\mathcal{X}_S|-1}{n_S} \Bigr)^{1/2} +  2\biggl( \log 2 \sum_{S \in \mathbb{S}} \frac{1}{n_S} \biggr)^{1/2} \lesssim_{|\mathbb{S}|} \max_{S \in \mathbb{S}} \biggl( \frac{|\mathcal{X}_S|}{n_S} \biggr)^{1/2}.
\]

As far as computation of the test statistic is concerned, observe that, writing $\mathcal{X}_{\mathbb{S}} := \{(S,x_S):S \in \mathbb{S}, x_S \in \mathcal{X}_S\}$, we can identify $\mathcal{G}_{\mathbb{S}}$ with $[-1,\infty)^{\mathcal{X}_\mathbb{S}}$, and moreover, any $P_\mathbb{S} \in \mathcal{P}_\mathbb{S}$ can be identified with an element of $[0,1]^{\mathcal{X}_\mathbb{S}}$.  We will show in Proposition~\ref{Prop:L1Projection} below that the supremum in~\eqref{Eq:RPS} is attained.  In fact,  $R(P_{\mathbb{S}},\cdot)$ is linear and, under the identification above, $\mathcal{G}_{\mathbb{S}}^+$ is a  convex polyhedral set, so we can compute $R(\hat{P}_{\mathbb{S}})$ using efficient linear programming algorithms.  

\subsection{An improved test under additional information}
\label{Sec:ImprovedTest}

In this subsection, we show how in the discrete setting of Section~\ref{Sec:Discrete}, it may be possible to reduce the critical value of our test, while retaining finite-sample Type I error control, when certain information about the facet structure of relevant polytopes is available.  This information does not depend on any quantities that are unknown to the practitioner, though exact computation may be a challenge when $|\mathbb{S}|$ or the alphabet sizes are large. 

Before we can describe our improved test, it is helpful to study the geometric structure of the problem further.  Regarding $\mathcal{G}_\mathbb{S}^+$ as a polyhedral convex subset of $[-1,\infty)^{\mathcal{X}_{\mathbb{S}}}$, it has a finite number of extreme points, so $\sup_{f_\mathbb{S} \in \mathcal{G}_\mathbb{S}^+} R(P_\mathbb{S}, f_\mathbb{S}) = \max_{\ell \in [L]} R(P_\mathbb{S}, f_\mathbb{S}^{(\ell)})$ for some $f_\mathbb{S}^{(1)},\ldots,f_\mathbb{S}^{(L)} \in \mathcal{G}_\mathbb{S}^+$. Thus $P_\mathbb{S} \in \mathcal{P}_\mathbb{S}^{0}$ if and only if $\max_{\ell \in [L]} R(P_\mathbb{S}, f_\mathbb{S}^{(\ell)}) \leq 0$, and  $\mathcal{P}_\mathbb{S}^{0}$ can be identified with a finite intersection of halfspaces, i.e.~it can be identified with a convex polyhedron in $[0,1]^{\mathcal{X}_\mathbb{S}}$. 
Now define the \emph{marginal cone\footnote{Here, given $\lambda > 0$ and a distribution $P$ on a measurable space $(\mathcal{Z},\mathcal{C})$, the measure $\lambda \cdot P$ is defined in the obvious way by $(\lambda \cdot P)(C) := \lambda \cdot P(C)$ for $C \in \mathcal{C}$; likewise, for a family of distributions $\mathcal{P}$, we write $\lambda \cdot \mathcal{P} := \{\lambda \cdot P:P \in \mathcal{P}\}$.}} $\mathcal{P}_\mathbb{S}^{0,*} := \{\lambda \cdot \mathcal{P}_\mathbb{S}^0 : \lambda \geq 0\}$. From the discussion above, $\mathcal{P}_\mathbb{S}^{0,*}$ can be identified with all non-negative multiples of a convex polyhedron, so can itself be identified with a convex polyhedral cone in $[0,\infty)^{\mathcal{X}_\mathbb{S}}$.  

When $\emptyset \neq S_2 \subseteq S_1 \subseteq [d]$ and $P_{S_1}$ is a measure on $\mathcal{X}_{S_1}$, we write $P_{S_1}^{S_2}$ for the marginal measure of $P_{S_1}$ on $\mathcal{X}_{S_2}$.  Recall that a family $P_{\mathbb{S}} = (P_S:S \in \mathbb{S}) \in \mathcal{P}_\mathbb{S}$ is \emph{consistent} if, whenever $S_1,S_2 \in \mathbb{S}$ have $S_1 \cap S_2 \neq \emptyset$, we have $P_{S_1}^{S_1 \cap S_2} = P_{S_2}^{S_1 \cap S_2}$.  We let $\mathcal{P}_{\mathbb{S}}^{\mathrm{cons}} \subseteq \mathcal{P}_{\mathbb{S}}$ denote the set of consistent families of distributions on~$\mathcal{X}_{\mathbb{S}}$, with corresponding \emph{consistent cone} $\mathcal{P}_{\mathbb{S}}^{\mathrm{cons},*} := \{\lambda \cdot \mathcal{P}_{\mathbb{S}}^{\mathrm{cons}}:\lambda \geq 0\}$ and \emph{consistent ball} $\mathcal{P}_\mathbb{S}^{\mathrm{cons},**}:=\{\lambda \cdot \mathcal{P}_{\mathbb{S}}^{\mathrm{cons}}:\lambda \in [0,1]\}$.  Thinking of $\mathcal{P}_\mathbb{S}^{\mathrm{cons}}$ as a convex polytope in $[0,\infty)^{\mathcal{X}_\mathbb{S}}$, the Minkowski sum $\mathcal{P}_\mathbb{S}^{0,*} + \mathcal{P}_\mathbb{S}^{\mathrm{cons},**}$ is also a convex polyhedral set, so has a finite number of facets \citep[][Theorem~19.1]{rockafellar1997convex}.  These facets fall into two categories: those that define the non-negativity conditions (i.e.~$(P_\mathbb{S})_{(S,x_S)} = P_S(\{x_S\}) \geq 0$ for all $S \in \mathbb{S}$ and $x_S \in \mathcal{X}_S$), which are not of primary interest to us here, and the remainder, which we refer to as the set of \emph{essential} facets. We remark that, in decomposable settings where $\mathcal{P}_\mathbb{S}^0 = \mathcal{P}_\mathbb{S}^\mathrm{cons}$, there are no essential facets.  More generally, regardless of whether $\mathbb{S}$ is decomposable, we still have the following:
\begin{prop}
\label{Prop:FullDim}
$\mathcal{P}_\mathbb{S}^0$ is a full-dimensional subset of $\mathcal{P}_\mathbb{S}^\mathrm{cons}$.
\end{prop}

In addition to the geometric insight of Proposition~\ref{Prop:FullDim}, it is also interesting from a statistical perspective when we consider testing compatibility against consistent alternatives (which captures the main essence of the problem in many examples; see the discussion at the end of Section~\ref{Sec:ImprovedTest}).  It  reveals a distinction with standard, fully-observed hypothesis testing problems (e.g.~goodness-of-fit testing, two-sample testing, independence testing), where the null hypothesis parameter space is of lower dimension \citep[e.g.,][]{fienberg1968geometry}.

We are now in a position to present Proposition~\ref{Prop:L1Projection}, whose main (second) part provides a decomposition of the incompatibility index.

\begin{prop}
\label{Prop:L1Projection}
In the discrete setting above, the supremum in~\eqref{Eq:RPS} and the infimum in~\eqref{Eq:InfR} are attained.  Moreover, writing $F$ for the number of essential facets of $\mathcal{P}_\mathbb{S}^{0,*} + \mathcal{P}_\mathbb{S}^{\mathrm{cons},**}$, there exist $f_\mathbb{S}^{(1)},\ldots,f_\mathbb{S}^{(F)} \in \mathcal{G}_\mathbb{S}^+$, depending only on $\mathbb{S}$ and $\mathcal{X}_\mathbb{S}$, such that for any $P_{\mathbb{S}} = (P_S:S \in \mathbb{S}) \in \mathcal{P}_{\mathbb{S}}$, we have
\begin{align}
\label{Eq:RDecompStatement}
    \max\biggl\{ \max_{\ell\in [F]} R(P_\mathbb{S},& f_\mathbb{S}^{(\ell)}), \frac{1}{|\mathbb{S}|} \max_{S_1,S_2 \in \mathbb{S}} d_\mathrm{TV} \bigl( P_{S_1}^{S_1 \cap S_2}, P_{S_2}^{S_1 \cap S_2} \bigr) \biggr\} \leq R(P_\mathbb{S}) \nonumber \\
    & \leq \max_{\ell\in [F]} R(P_\mathbb{S}, f_\mathbb{S}^{(\ell)})_+ + |\mathbb{S}| 2^{|\mathbb{S}|+2} \cdot \max_{S_1,S_2 \in \mathbb{S}} d_\mathrm{TV} \bigl( P_{S_1}^{S_1 \cap S_2}, P_{S_2}^{S_1 \cap S_2} \bigr),
\end{align}
where we interpret $\max_{\ell \in [0]} R(P_\mathbb{S},f_{\mathbb{S}}^{(\ell)})_+=0$.
\end{prop}
Proposition~\ref{Prop:L1Projection} shows in particular that when $P_\mathbb{S} \in \mathcal{P}_\mathbb{S}^{\mathrm{cons}}$, the number of essential facets of $\mathcal{P}_\mathbb{S}^{0,*} + \mathcal{P}_\mathbb{S}^{\mathrm{cons},**}$ governs the complexity of the incompatibility index, and we can write $R(P_\mathbb{S})$ in irreducible form as
\[
    R(P_\mathbb{S}) = \max_{\ell \in [F]} R(P_\mathbb{S}, f_\mathbb{S}^{(\ell)})_+.
\]
For general $P_\mathbb{S} \in \mathcal{P}_\mathbb{S}$, Proposition~\ref{Prop:L1Projection} shows that $R(P_\mathbb{S})$ can be expressed as a  maximum of this irreducible part and (up to a multiplicative factor depending only on $|\mathbb{S}|$) a total variation measure of inconsistency that quantifies the distance of $P_\mathbb{S}$ from $\mathcal{P}_\mathbb{S}^{\mathrm{cons}}$. As we will see below, the ideal situation is where we have knowledge of $F$, and we can then exploit this in the construction of powerful tests.  For instance, when $\mathbb{S} = \bigl\{\{1,2\},\{2,3\},\{1,3\}\bigr\}$ and $\mathcal{X}_1 = [r]$, $\mathcal{X}_2 = [s]$ and $\mathcal{X}_3 = [2]$, we have $F = (2^r-2)(2^s-2)$; cf.~Theorem~\ref{Prop:rs2example} and the subsequent discussion.  In more complicated examples, such knowledge may not be readily available, but we will also see, e.g.~in Proposition~\ref{Prop:CutSet} below, that it is nevertheless often possible to find bounds of the form 
\begin{equation}
\label{Eq:Equivalence}
   \max_{\ell \in [F']} R(P_\mathbb{S}, f_\mathbb{S}^{(\ell),'})_+ \leq  R(P_\mathbb{S}) \leq D_R \max_{\ell \in [F']} R(P_\mathbb{S}, f_\mathbb{S}^{(\ell),'})_+
\end{equation}
for some known $D_R>0$, $F' \in \mathbb{N}_0$, $f_\mathbb{S}^{(1),'},\ldots, f_\mathbb{S}^{(F'),'} \in \mathcal{G}_\mathbb{S}^+ \cap [-1,|\mathbb{S}|-1]^{\mathcal{X}_\mathbb{S}}$ and for all $P_\mathbb{S} \in \mathcal{P}_\mathbb{S}^{\mathrm{cons}}$. It then follows from~\eqref{Eq:RDecomp} in the proof of Proposition~\ref{Prop:L1Projection} that, in the upper bound in~\eqref{Eq:RDecompStatement}, we may replace $\max_{\ell \in [F]} R(P_\mathbb{S}, f_\mathbb{S}^{(\ell)})_+$ by $D_R \max_{\ell \in [F']} R(P_\mathbb{S}, f_\mathbb{S}^{(\ell),'})_+$.

Our alternative test rejects $H_0':P_\mathbb{S} \in \mathcal{P}_\mathbb{S}^0$ at the significance level $\alpha \in (0,1)$ if and only if $R(\hat{P}_{\mathbb{S}}) \geq  C_\alpha' \equiv C_\alpha'\bigl(|\mathcal{X}_1|,\ldots,|\mathcal{X}_d|,\mathbb{S},(n_S:S \in \mathbb{S}), D_R, F'\bigr)$, where
\begin{align*}
    C_\alpha' := |\mathbb{S}| \max \biggl\{ \frac{2D_R^2 \log\bigl( \frac{2F'|\mathbb{S}|}{\alpha} \vee 1\bigr)}{ \min_{S \in \mathbb{S}} n_S},  2^{2|\mathbb{S}|+7} \! \!  \max_{\substack{S_1,S_2 \in \mathbb{S}:\\S_1 \neq S_2, S_1 \cap S_2 \neq \emptyset}} \frac{|\mathcal{X}_{S_1 \cap S_2}|\log 2 + \log\bigl(\frac{2|\mathbb{S}|(|\mathbb{S}|-1)}{\alpha}\bigr)}{n_{S_1} \wedge n_{S_2}} \biggr\}^{1/2}. 
\end{align*}
Here, $D_R,F'$ are such that~\eqref{Eq:Equivalence} holds.  If the number of essential facets $F$ of $\mathcal{P}_\mathbb{S}^{0,*} + \mathcal{P}_\mathbb{S}^{\mathrm{cons},**}$ is known, then we may take $F'=F$ and $D_R=1$. 
The following theorem provides size and power guarantees for this test. 
\begin{thm}
\label{Prop:DiscreteTest}
 Fix $\alpha, \beta \in (0,1)$.  Whenever $P_{\mathbb{S}} = (P_S:S \in \mathbb{S}) \in \mathcal{P}_{\mathbb{S}}^0$, we have $\mathbb{P}_{P_{\mathbb{S}}}\bigl(R(\hat{P}_{\mathbb{S}}) \geq C_\alpha'\bigr) \leq \alpha$. Moreover, there exists $M \equiv M(|\mathbb{S}|,D_R)>0$ such that for any $P_{\mathbb{S}} \in \mathcal{P}_\mathbb{S}$ satisfying
\begin{equation}
\label{Eq:BoundedFromNull}
    R(P_{\mathbb{S}}) \geq M(C_\alpha' + C_\beta'),
\end{equation}
we have $\mathbb{P}_{P_{\mathbb{S}}}\bigl(R(\hat{P}_{\mathbb{S}}) \geq  C_\alpha'\bigr) \geq 1-\beta$.
\end{thm}
Of course, by combining Proposition~\ref{Prop:DiscreteTest1} and Theorem~\ref{Prop:DiscreteTest}, we see that the test that rejects $H_0'$ if $R(\hat{P}_{\mathbb{S}}) \geq  \min( C_\alpha, C_\alpha') =:C_\alpha^\mathrm{min}$ remains of size $\alpha$, so is an improved test that represents the best of both worlds.  By taking $F'=F$ and $D_R=1$, Proposition~\ref{Prop:DiscreteTest1} and Theorem~\ref{Prop:DiscreteTest} now reveal that 
\begin{align*}
    \rho^*(n_\mathbb{S}) &\leq 2\min\bigl( MC_{1/4}' , C_{1/4}\bigr) \\
    &\lesssim_{|\mathbb{S}|} \min \biggl\{ \biggl( \frac{\log (F \vee 1)}{\min_{S \in \mathbb{S}} n_S} + \max_{\substack{S_1,S_2 \in \mathbb{S}:\\S_1 \neq S_2, S_1 \cap S_2 \neq \emptyset}} \frac{|\mathcal{X}_{S_1 \cap S_2}|}{n_{S_1} \wedge n_{S_2}} \biggr)^{1/2} , \max_{S \in \mathbb{S}} \biggl( \frac{|\mathcal{X}_S|}{n_S} \biggr)^{1/2} \biggr\}.
\end{align*}
By McMullen's Upper bound theorem \citep{mcmullen1970maximum}, 
\[
    \log (F \vee 1) \lesssim_{|\mathbb{S}|} \log |\mathcal{X}| \cdot \max_{S \in \mathbb{S}} |\mathcal{X}_S|,
\]
so that, when all sample sizes are of the same order of magnitude, we have $C_\alpha' + C_\beta' \lesssim_{|\mathbb{S}|} (C_\alpha + C_\beta) \cdot \log |\mathcal{X}|$.  When tight bounds on $\log F$ are available, however, we may have that $C_\alpha' + C_\beta'$ is much smaller than $C_\alpha + C_\beta$; see the discussion following Theorem~\ref{Prop:rs2example} below.  

While these quantities are rather abstract, we can simplify them in certain cases. It is known from previous work~\citep[e.g.][]{vlach1986conditions} that when $\mathbb{S}=\bigl\{\{1,2\},\{2,3\},\{1,3\}\bigr\}$ and $\mathcal{X}=[r] \times [s] \times [2]$ for some $r,s \in \mathbb{N}$, the marginal cone induced by the set of compatible measures is given by
\[
    \mathcal{P}_\mathbb{S}^{0,*} = \Bigl\{ P_\mathbb{S} \equiv p_\mathbb{S} \in \mathcal{P}_\mathbb{S}^{\mathrm{cons},*}: \max_{A \subseteq [r], B \subseteq[s]} (-p_{AB \bullet} +p_{A \bullet 1} + p_{\bullet B 1} - p_{\bullet \bullet 1}) \leq 0 \Bigr\},
\]
where, for example, $p_{AB\bullet} := P_{\{1,2\}}(A \times B)$ and $p_{\bullet \bullet 1} := P_{\{1,3\}}([r] \times \{1\}) =P_{\{2,3\}}([s] \times \{1\})$.  However, the extension in the first part of Theorem~\ref{Prop:rs2example} below, which provides an exact expression for the incompatibility index for an arbitrary family of consistent marginal distributions, is new.  The second part provides a representation of $\mathcal{P}_\mathbb{S}^{0,*} + \mathcal{P}_\mathbb{S}^{\mathrm{cons},**}$ as an intersection of $F = (2^r-2)(2^s-2)$ closed halfspaces; thus, $C_\alpha'$ is known exactly, and can be used in our test of compatibility.
\begin{thm}
\label{Prop:rs2example}
Let $\mathbb{S}=\bigl\{\{1,2\},\{2,3\},\{1,3\}\bigr\}$ and $\mathcal{X}=[r] \times [s] \times [2]$ for some $r,s \in \mathbb{N}$.  Then for any $P_\mathbb{S} \in \mathcal{P}_\mathbb{S}^{\mathrm{cons}}$, we have
\begin{equation}
    \label{Eq:RPSd3}
    R(P_\mathbb{S}) = 2 \max_{A \subseteq [r], B \subseteq[s]} (-p_{AB \bullet} +p_{A \bullet 1} + p_{\bullet B 1} - p_{\bullet \bullet 1})_+.
\end{equation}
Moreover,
\[
    \mathcal{P}_\mathbb{S}^{0,*} + \mathcal{P}_\mathbb{S}^{\mathrm{cons},**} = \Bigl\{ P_\mathbb{S} \equiv p_\mathbb{S} \in \mathcal{P}_\mathbb{S}^{\mathrm{cons},*}: \max_{A \subseteq [r], B \subseteq[s]} (-p_{AB \bullet} +p_{A \bullet 1} + p_{\bullet B 1} - p_{\bullet \bullet 1}) \leq 1/2 \Bigr\}. 
\]
\end{thm}
\noindent \textbf{Remark}: In the special case $s=2$, the expression in~\eqref{Eq:RPSd3} simplifies to
\begin{equation}
    \label{Eq:RS3s2}
R(P_{\mathbb{S}}) = 2\max_{j \in [2]} \biggl\{ p_{\bullet j 1} - \sum_{i=1}^r \min(p_{ij \bullet}, p_{i \bullet 1}) \biggr\}_+.
\end{equation}
This can be compared with corresponding expressions in the $d=4$ cases that are given Example~\ref{Prop:4dexamples} and in Proposition~\ref{Prop:4dexampleAnalytic}.


From the expression for $F$ in this case, we see that when $n_{\{1,2\}} = n_{\{2,3\}} = n_{\{1,3\}} = n/3$, we have
\begin{align*}
C_\alpha' +C_\beta' \asymp \biggl\{\frac{r + s + \log\bigl(1/(\alpha \wedge \beta)\bigr)}{n}\biggr\}^{1/2}, \quad C_{\alpha} + C_\beta \asymp \biggl\{\frac{rs + \log\bigl(1/(\alpha \wedge \beta)\bigr)}{n}\biggr\}^{1/2}.
\end{align*}
More generally, as a consequence of  Theorems~\ref{Prop:DiscreteTest} and~\ref{Prop:rs2example}, 
\begin{equation}
\label{Eq:rs2UpperBound}
    \rho^*(n_{\mathbb{S}}) \lesssim \Bigl( \frac{r + s}{n_{\{1,2\}}} \Bigr)^{1/2} + \Bigl( \frac{r}{n_{\{1,3\}}} \Bigr)^{1/2} + \Bigl( \frac{s}{n_{\{2,3\}}} \Bigr)^{1/2}.
\end{equation}

The main challenge in the proof of Theorem~\ref{Prop:rs2example} is to establish~\eqref{Eq:RPSd3}, since the second part then follows using arguments from the proof of Proposition~\ref{Prop:L1Projection}.  Our strategy is to obtain matching lower and upper bounds on $R(P_{\mathbb{S}})$ via the primal and dual formulations~\eqref{Eq:RPS} and~\eqref{Eq:InfR} respectively.  The lower bound requires, for each $A \subseteq [r]$ and $B \subseteq [s]$, a construction of $f_\mathbb{S} \in \mathcal{G}_\mathbb{S}^+$ for which we can compute $R(P_\mathbb{S},f_\mathbb{S})$.  On the other hand, the upper bound relates $R(P_\mathbb{S})$ to the maximum two-commodity flow \citep[][Chapter~17]{ahuja1988network} through a specially-chosen network.  \cite{vlach1986conditions} gives a halfspace representation for $\mathcal{P}_\mathbb{S}^{0,*}$ using the max-flow min-cut theorem for a single-commodity flow through a simpler network; since there is no general max-flow min-cut theorem for two-commodity flows \citep{leighton1999multicommodity}, our proof is more involved.

Theorem~\ref{Prop:rs2lowerbound} below provides a lower bound on the minimax testing radius in the setting of Theorem~\ref{Prop:rs2example}.  
\begin{thm}
\label{Prop:rs2lowerbound}
Let  $\mathbb{S}=\bigl\{\{1,2\},\{2,3\},\{1,3\}\bigr\}$ with  $|\mathcal{X}_1|=r$ for some $r \geq 2$, $|\mathcal{X}_2|=2$ and $|\mathcal{X}_3|=2$.   There exists a universal constant $c > 0$ such that
\[
    \rho^*(n_{\mathbb{S}}) \geq  c \max \biggl\{ \frac{1}{\log r} \wedge \biggl( \frac{r}{(n_{\{1,2\}} \wedge n_{\{1,3\}} ) \log r} \biggr)^{1/2},  \frac{1}{(\min_{S \in \mathbb{S}} n_S)^{1/2}} \biggr\}.
\]
\end{thm}
Theorem~\ref{Prop:rs2lowerbound} may be applied in $r \times s \times 2$ tables by noting that $\rho^*$ cannot decrease when $|\mathcal{X}_S|$ increases, for any $S \in \mathbb{S}$. In the main regime of interest where $n_{\{1,2\}} \geq (r + s)\log(r + s), n_{\{1,3\}} \geq r \log r$ and $n_{\{2,3\}} \geq s \log s$, we can conclude that 
\[
\rho^*(n_{\mathbb{S}}) \gtrsim \Bigl( \frac{r + s}{n_{\{1,2\}} \log (r + s)} \Bigr)^{1/2} + \Bigl( \frac{r}{n_{\{1,3\}} \log r} \Bigr)^{1/2} + \Bigl( \frac{s}{n_{\{2,3\}} \log s} \Bigr)^{1/2}.
\]
When compared with our upper bound in~\eqref{Eq:rs2UpperBound}, we see that our improved test is minimax rate-optimal, up to logarithmic factors. 

The proof of Theorem~\ref{Prop:rs2lowerbound} relies on Lemma~\ref{Lemma:PoissonLemma} in Section~\ref{Sec:Proofs}, which provides a bound on the total variation distance between paired Poisson mixtures, and is an extension of both \citet[][Lemma~3]{wu2016minimax} and \citet[][Lemma~32]{jiao2018minimax}.  We remark that the sequences $P_\mathbb{S}$ constructed in our lower bound belong to $\mathcal{P}_\mathbb{S}^{\mathrm{cons}}$; in other words, the same lower bound on the minimax separation rate holds for testing against consistent alternatives.


\subsection{Reductions}
\label{Sec:Reductions}

In this subsection, we show how, for certain $\mathbb{S} \subseteq \mathrm{Pow}([d])$, the incompatibility index $R(P_\mathbb{S})$ can be expressed in terms of $R(P_{\mathbb{S}'})$ for some collection $\mathbb{S}' \subseteq \mathrm{Pow}(\mathcal{J})$, with $\mathcal{J}$ a proper subset of  $[d]$.  Conceptually, such formulae provide understanding of the facet structure of $\mathcal{P}_\mathbb{S}^{0,*} + \mathcal{P}_{\mathbb{S}}^{\mathrm{cons},**}$, which in turn allows us to obtain tighter bounds on the critical values employed in our improved test (cf.~Section~\ref{Sec:ImprovedTest}).  Computationally, these formulae extend the scope of results such as Theorem~\ref{Prop:rs2example} by allowing us to provide explicit expressions for $R(P_\mathbb{S})$ in a wider range of examples.

Our first reduction considers a setting where there exists a subset of variables that are only observed as part of a single observation pattern within our class of possible patterns.  Given $\mathbb{S} \subseteq \mathrm{Pow}([d])$ and $J \subseteq [d]$, we write $\mathbb{S}^{-J} := \{S \cap J^c:S \in \mathbb{S}\}$.
\begin{prop}
\label{Prop:SinglePattern}
Let $\mathbb{S} \subseteq \mathrm{Pow}([d])$, and suppose that $\emptyset \neq J \subseteq [d]$ and $S_0 \in \mathbb{S}$ are such that $J \subseteq S_0$ but $J \cap S = \emptyset$ for all $S \in \mathbb{S} \setminus \{S_0\}$.  Writing $P_{\mathbb{S}}^{-J} := (P_S:S \in \mathbb{S}\setminus S_0,P_{S_0}^{S_0 \cap J^c})$, we have that if $P_{\mathbb{S}}^{-J} \in \mathcal{P}_{\mathbb{S}^{-J}}^{\mathrm{cons}}$, then $P_\mathbb{S} \in \mathcal{P}_{\mathbb{S}}^{\mathrm{cons}}$.  Moreover, regardless of consistency, 
\[
R(P_\mathbb{S}) = R(P_{\mathbb{S}}^{-J}).
\]
\end{prop}

As an illustration of Proposition~\ref{Prop:SinglePattern} suppose that $\mathcal{X} = [r] \times [s] \times [2] \times [t] \times [u]$ and $\mathbb{S} = \bigl\{\{1,2,4\},\{2,3\},\{1,3,5\}\bigr\}$.  Then $R(P_\mathbb{S}) = R(P_{\mathbb{S}^{-\{4,5\}}})$, and if  $P_\mathbb{S}^{-\{4,5\}} \in \mathcal{P}_{\mathbb{S}^{-\{4,5\}}}^{\mathrm{cons}}$, then
\[
R(P_\mathbb{S}) = 2 \max_{A \subseteq [r],B \subseteq [s]} \bigl(-p_{AB\bullet \bullet \bullet} + p_{A\bullet 1 \bullet \bullet } + p_{\bullet B1 \bullet \bullet} - p_{\bullet \bullet 1 \bullet \bullet}\bigr)_+.
\]
Next, we consider a complementary situation where a subset of variables appears in all of our possible observation patterns.  For the purposes of this result, we will assume that $(\mathcal{X}_j:j \in [d])$ are Polish spaces, so that regular conditional distributions and disintegrations are well-defined (e.g.~\citet[][Chapter~10]{dudley2018real} and \citet[][Lemma~35]{reeve2021optimal}).  Specifically, if $S \subseteq [d]$ and $J \subseteq S$, then there exists a family $(P_{S|x_J}:x_J \in \mathcal{X}_J)$ of probability measures on $\mathcal{X}_{S \cap J^c}$ with the properties that $x_J \mapsto P_{S|x_J}(B)$ is measurable for every measurable $B \subseteq \mathcal{X}_{S \cap J^c}$, and $\int_A P_{S|x_J}(B) \, dP_S^J(x_J) = P_S(A \times B)$ for all $A \in \mathcal{A}_J, B \in \mathcal{A}_{S \cap J^c}$.  We then write $P_{\mathbb{S}|x_J} := (P_{S|x_J}:S \in \mathbb{S})$ for each $x_J \in \mathcal{X}_J$.
\begin{prop}
\label{Prop:EveryPattern}
Let $\mathbb{S} \subseteq \mathrm{Pow}([d])$, and suppose that $J \subseteq [d]$ is such that $J \subseteq S$ for every $S \in \mathbb{S}$.  Suppose further that there exists a distribution $P^J$ on $\mathcal{X}_J$ such that $P_S^J = P^J$ for all $S \in \mathbb{S}$.  Then
\begin{equation}
    \label{Eq:ConditioningInequality}
R(P_\mathbb{S}) \leq \int_{\mathcal{X}_J} R(P_{\mathbb{S}|x_J}) \, dP^J(x_J).
\end{equation}
Moreover, in the discrete case where $\mathcal{X}_j = [m_j]$ for some $m_1,\ldots,m_d \in \mathbb{N} \cup \{\infty\}$, the inequality~\eqref{Eq:ConditioningInequality} is in fact an equality.
\end{prop}

As an application of Proposition~\ref{Prop:EveryPattern}, suppose that $\mathbb{S}=\bigl\{\{1,2,3\},\{1,3,4\},\{1,2,4\}\bigr\}$, where $\mathcal{X}_1=[r],\mathcal{X}_2=[s],\mathcal{X}_3=[t],\mathcal{X}_4=[2]$, and where $P_\mathbb{S} \in \mathcal{P}_\mathbb{S}^{\mathrm{cons}}$. Then Proposition~\ref{Prop:EveryPattern} combined with Theorem~\ref{Prop:rs2example} yields that
\[
    R(P_\mathbb{S}) = 2\sum_{i=1}^r \max_{A \subseteq [s], B \subseteq [t]} ( - p_{i A B \bullet } + p_{i A \bullet 1 } + p_{i \bullet B 1} - p_{i \bullet \bullet 1})_+.
\]
\sloppy This shows that in this setting we can find  $f_\mathbb{S}^{(1)},\ldots,f_\mathbb{S}^{(F)} \in \mathcal{G}_\mathbb{S}^+$ such that $R(P_\mathbb{S}) = \max_{\ell \in [F]} R(P_\mathbb{S},f_\mathbb{S}^{(\ell)})_+$, with $F \leq \{(2^s-2)(2^t-2)\}^r \leq 2^{r(s+t)}$.

Our final reduction result provides good upper and lower bounds on $R(P_\mathbb{S})$ in settings where there exists $J \in \mathbb{S}$ such that $[d]$ can be partitioned into $(I, J,K)$, where every $S \in \mathbb{S}$ is a subset of either $I \cup J$ or $J \cup K$.  As an alternative way of expressing this, if $\mathbb{S}_1,\mathbb{S}_2\subseteq \mathbb{S}$, we say $J \in \mathbb{S}$ is a \emph{cut set} for $\mathbb{S}_1$ and $\mathbb{S}_2$ if $\mathbb{S}_1 \cap \mathbb{S}_2 = \{J\}$ and $(\cup_{S \in \mathbb{S}_1} S) \cap (\cup_{S \in \mathbb{S}_2} S) = J$.

\begin{prop}
\label{Prop:CutSet}
Let $\mathbb{S} \subseteq [d]$, and suppose that $\mathbb{S}_1,\mathbb{S}_2 \subseteq \mathbb{S}$ are such that $J$ is a cut set for $\mathbb{S}_1$ and $\mathbb{S}_2$.  Then for any $P_\mathbb{S} \in \mathcal{P}_\mathbb{S}$, we have
\[
    \max\{ R(P_{\mathbb{S}_1}),  R(P_{\mathbb{S}_2}) \} \leq R(P_\mathbb{S}) \leq  R(P_{\mathbb{S}_1}) + R(P_{\mathbb{S}_2}).
\]
\end{prop}
In Example~\ref{Prop:4dexamples}(ii) below, we give an exact expression for $R(P_\mathbb{S})$ when $P_\mathbb{S} \in \mathcal{P}_\mathbb{S}^\mathrm{cons}$ in the special case where $\mathbb{S}=\bigl\{\{1,2\},\{2,3\},\{1,3\}, \{3,4\}, \{1,4\}\bigr\}$ with $\mathcal{X}_j=[2]$ for all $j \in [4]$.  Here, $\{1,3\}$ is a cut set for $\mathbb{S}_1 = \bigl\{ \{1,2\}, \{2,3\}, \{1,3\}\bigr\}$ and $\mathbb{S}_2 = \bigl\{ \{1,3\}, \{3,4\},\{1,4\} \bigr\}$ (see Figure~\ref{d4}(b)), and our calculations confirm that the conclusion of Proposition~\ref{Prop:CutSet} holds with these choices of $\mathbb{S}_1,\mathbb{S}_2$ and $J$. More generally, when~$\mathbb{S}$ is as above, $\mathcal{X} = [2] \times [r] \times [s] \times [t]$ and $P_\mathbb{S} \in \mathcal{P}_\mathbb{S}^\mathrm{cons}$, we can now see that $\tilde{R}(P_{\mathbb{S}}) \leq R(P_\mathbb{S}) \leq 2\tilde{R}(P_{\mathbb{S}})$, where
\[
    \tilde{R}(P_{\mathbb{S}}) := 2\max\biggl\{\max_{\substack{A \subseteq [r]\\B \subseteq [s]}}(-p_{\bullet AB \bullet } + p_{1A \bullet \bullet } + p_{1 \bullet B \bullet } - p_{1 \bullet \bullet \bullet }) , \max_{\substack{A \subseteq [t]\\B \subseteq [s]}}(-p_{\bullet \bullet B A } + p_{1 \bullet \bullet A } + p_{1 \bullet B \bullet } - p_{1 \bullet \bullet \bullet }) \biggr\}_+.
\]
Thus~\eqref{Eq:Equivalence} holds with $D_R=2$ and $F' =(2^s-2)\{ (2^r-2) + (2^t-2)\} \leq 2^{s + \max(r,t)+1}$, so we can apply our test using the critical value $C_\alpha'$ with these choices. 

\subsection{Computation}
\label{Sec:Computation}

While $C_\alpha$ can be easily calculated for any test of compatibility and allows for a test with power against all incompatible alternatives, we have seen that $C_\alpha'$ can be smaller and lead to more powerful tests. Practical use of $C_\alpha'$ requires knowledge of the number $F$ of essential facets of the polyhedral set $\mathcal{P}_\mathbb{S}^{0,*} + \mathcal{P}_\mathbb{S}^{\mathrm{cons},**}$, or $D_R$ and $F'$ such that~\eqref{Eq:Equivalence} holds. These are fully determined by $\mathbb{S}$ and $\mathcal{X}$, so in principle are known, but these polyhedral sets can be highly complex and explicit expressions for their numbers of essential facets are not generally available. Nevertheless, given particular $\mathbb{S}$ and $\mathcal{X}$, it is possible to compute explicit halfspace representations of $\mathcal{P}_\mathbb{S}^{0,*} + \mathcal{P}_\mathbb{S}^{\mathrm{cons},**}$ using well-developed packages for linear programming. In this section we describe some of the basic geometric concepts involved and how existing algorithms can be used in our setting. As our concern is to describe computational methods, we restrict attention to discrete settings where $|\mathcal{X}| < \infty$, so $\mathcal{P}_\mathbb{S}$ is finite-dimensional.

Existing work mentioned in the introduction has focused on the simpler problem of the computation of the facet structure of $\mathcal{P}_\mathbb{S}^0$, and we begin by describing the approach taken there. 
Starting from the definition of $\mathcal{P}_\mathbb{S}^0$, and writing $p_S(x_S)=P_S(\{x_S\})$ for $(S,x_S) \in \mathcal{X}_\mathbb{S}$, we have
\begin{align*}
    \mathcal{P}_\mathbb{S}^0 &= \{ P_\mathbb{S} \in \mathcal{P}_\mathbb{S} : \mathcal{F}_\mathbb{S}(P_\mathbb{S}) \neq \emptyset \} \\
    &= \biggl\{ P_\mathbb{S} \in \mathcal{P}_\mathbb{S} : \exists p \in [0,1]^\mathcal{X} \text{ s.t. } p_S(x_S) = \sum_{x_{S^c} \in \mathcal{X}_{S^c}} p(x_S,x_{S^c}) \, \forall S \in \mathbb{S}, x_S \in \mathcal{X}_S \biggr\} \\
    & = \{ P_\mathbb{S} \in\mathcal{P}_\mathbb{S} : \exists p \in [0,1]^\mathcal{X} \text{ s.t. } \mathbb{A} p = p_\mathbb{S} \},
\end{align*}
where the matrix $\mathbb{A} = (\mathbb{A}_{(S,y_S),x})_{(S,y_S) \in \mathcal{X}_\mathbb{S},x \in \mathcal{X}} \in \{0,1\}^{\mathcal{X}_{\mathbb{S}} \times \mathcal{X}}$ has entries
\begin{equation}
\label{Eq:A}
    \mathbb{A}_{(S,y_S),x} := \mathbbm{1}_{\{x_S = y_S\}}.
\end{equation}  
Since each column of $\mathbb{A}$ has exactly $|\mathbb{S}|$ entries equal to 1 (one for each $S \in \mathbb{S}$), it follows that any $p \in [0,1]^\mathcal{X}$ with $\mathbb{A} p = p_\mathbb{S}$ satisfies $ 1_\mathcal{X}^T p =|\mathbb{S}|^{-1} 1_{\mathcal{X}_\mathbb{S}}^T \mathbb{A} p = |\mathbb{S}|^{-1} 1_{\mathcal{X}_\mathbb{S}}^T p_\mathbb{S}=1$. 
We can therefore write $\mathcal{P}_{\mathbb{S}}^0$ as the convex hull of the columns of $\mathbb{A}$, with coefficients in the convex combination given by $p$.  In the rest of this section,  we adopt for compactness the convention that if $i \in [2]$, then $\bar{i} := 3 - i$, so that $\{\bar{i}\} = \{1,2\} \setminus \{i\}$.
\begin{example}
Consider the case $d=3$, where $\mathbb{S}=\bigl\{\{1,2\},\{2,3\},\{1,3\}\bigr\}$ and $\mathcal{X}=[2]^3$.  Here we have $|\mathcal{X}|=8,  |\mathcal{X}_\mathbb{S}|=12$ and, if we order the 12 rows according to $(1,1,\bullet),(1,2,\bullet),(2,1,\bullet),(2,2,\bullet)$ for $S=\{1,2\}$, then $(\bullet,1,1),(\bullet,1,2),(\bullet,2,1),(\bullet,2,2)$ for $S=\{2,3\}$, then $(1,\bullet,1),(2,\bullet,1),(1,\bullet,2),(2,\bullet,2)$ for $S=\{1,3\}$, we have
\begin{align*}
\setcounter{MaxMatrixCols}{20}
    \mathbb{A}^T = \begin{pmatrix} 
    1 & 0 & 0 & 0 & 1 & 0 & 0 & 0 & 1 & 0 & 0 & 0 \\
    1 & 0 & 0 & 0 & 0 & 1 & 0 & 0 & 0 & 0 & 1 & 0 \\
    0 & 1 & 0 & 0 & 0 & 0 & 1 & 0 & 1 & 0 & 0 & 0 \\
    0 & 1 & 0 & 0 & 0 & 0 & 0 & 1 & 0 & 0 & 1 & 0 \\
    0 & 0 & 1 & 0 & 1 & 0 & 0 & 0 & 0 & 1 & 0 & 0 \\
    0 & 0 & 1 & 0 & 0 & 1 & 0 & 0 & 0 & 0 & 0 & 1 \\
    0 & 0 & 0 & 1 & 0 & 0 & 1 & 0 & 0 & 1 & 0 & 0 \\
    0 & 0 & 0 & 1 & 0 & 0 & 0 & 1 & 0 & 0 & 0 & 0
    \end{pmatrix}.
\end{align*}
In this case, the polytope $\mathcal{P}_\mathbb{S}^0$ has 16 facets; of these, 12 correspond to the simple non-negativity conditions $p_\mathbb{S} \geq 0$, while the remaining four essential facets are given by $p_{i,\bar{j},\bullet} + p_{\bullet j 2} + p_{\bar{i}, \bullet, 1} \leq 1$ for $i,j \in [2]$ \citep{vlach1986conditions,eriksson2006polyhedral}. More generally, when $\mathcal{X}=[r] \times [s] \times [2]$ for some $r,s \in \mathbb{N}$, the marginal polytope $\mathcal{P}_\mathbb{S}^0$ has $(2^r-2)(2^s-2) + rs + 2(r+s)$ facets, with $rs +2(r+s)$ of these corresponding to simple non-negativity conditions.
\end{example}
We now turn to the problem of computing the number of essential facets of the polyhedral set $\mathcal{P}_\mathbb{S}^{0,*} + \mathcal{P}_\mathbb{S}^{\mathrm{cons},**}$, which is of more direct relevance in our context. As we see from Theorem~\ref{Prop:rs2example} and the example above, in the special case $\mathbb{S}=\bigl\{\{1,2\},\{2,3\},\{1,3\}\bigr\}$ and $\mathcal{X}=[r] \times [s] \times [2]$, the structure of $\mathcal{P}_\mathbb{S}^{0,*} + \mathcal{P}_\mathbb{S}^{\mathrm{cons},**}$ is similar to that of $\mathcal{P}_\mathbb{S}^0$; indeed, both polyhedral sets have the same numbers of essential and non-essential facets. However, the facet structure of $\mathcal{P}_\mathbb{S}^{0,*} + \mathcal{P}_\mathbb{S}^{\mathrm{cons},**}$ is generally more complicated than that of $\mathcal{P}_\mathbb{S}^0$; Example~\ref{Prop:4dexamples} reveals that when $d=4$ all irreducible choices of $\mathbb{S}$  except the simple chain pairs case exhibit this difference.  The Minkowski sum $\mathcal{P}_\mathbb{S}^{0,*} + \mathcal{P}_\mathbb{S}^{\mathrm{cons},**}$ is the convex hull of a set of directions (the columns of $\mathbb{A}$) and a set of points (the vertices of $\mathcal{P}_\mathbb{S}^{\mathrm{cons}}$, together with the origin).  Moreover, a halfspace representation of $\mathcal{P}_\mathbb{S}^{\mathrm{cons}}$ is given by
\begin{align*}
    \mathcal{P}_\mathbb{S}^{\mathrm{cons}} = \biggl\{ &p_\mathbb{S}  \in [0,\infty)^{\mathcal{X}_\mathbb{S}} : \sum_{x_S \in \mathcal{X}_S} p_S(x_S) = 1 \ \forall S \in \mathbb{S},   \\ &\sum_{x_{S_1 \cap S_2^c}\in \mathcal{X}_{S_1 \cap S_2^c}} p_{S_1}(x_{S_1}) - \sum_{x_{S_1^c \cap S_2}\in\mathcal{X}_{S_1^c \cap S_2}} p_{S_2}(x_{S_2}) = 0 \, \forall x_{S_1 \cap S_2} \in \mathcal{X}_{S_1\cap S_2}, S_1,S_2 \in \mathbb{S}  \biggr\},
\end{align*}
and we can convert this to a vertex representation using software such as the \texttt{rcdd} package in~\texttt{R} \citep{rcdd2021}. In fact, as shown by Proposition~\ref{Prop:FullDim}, the equality constraints of $\mathcal{P}_\mathbb{S}^\mathrm{cons}$ can be extracted from the equality constraints in the halfspace representation of $\mathcal{P}_\mathbb{S}^0$, a fact we use in our computations. The vertex representations of $\mathcal{P}_\mathbb{S}^0$ and $\mathcal{P}_\mathbb{S}^\mathrm{cons}$ lead to a vertex representation of the sum $\mathcal{P}_\mathbb{S}^{0,*} + \mathcal{P}_\mathbb{S}^{\mathrm{cons},**}$ that can then be converted back to a halfspace representation, again using software such as \texttt{rcdd}.  The value of $F$ is then given by the number of halfspaces in this representation, once we subtract the number of halfspaces defining $\mathcal{P}_\mathbb{S}^\mathrm{cons}$. 

To illustrate this computational approach, we find explicit expressions for $R(P_\mathbb{S})$ with $P_\mathbb{S} \in \mathcal{P}_\mathbb{S}^\mathrm{cons}$, for all irreducible four-dimensional examples with binary variables.  If $[d] \in \mathbb{S}$, then $\mathcal{P}_\mathbb{S}^{0} = \mathcal{P}_\mathbb{S}^{\mathrm{cons}}$ so $F=0$ and $R(P_{\mathbb{S}}) =0$ for $P_\mathbb{S} \in \mathcal{P}_\mathbb{S}^\mathrm{cons}$.  We are therefore more interested in situations where $[d] \notin \mathbb{S}$, and where compatibility is not equivalent to consistency.  By a combination of Propositions~\ref{Prop:SinglePattern} and~\ref{Prop:EveryPattern}, the set of possible irreducible observation patterns~$\mathbb{S}$ in the case $d=4$ with $[d] \notin \mathbb{S}$ is the following:
\begin{itemize}
    \item[$\bullet$] Chain pairs: $\mathbb{S} = \bigl\{\{1,2\},\{2,3\},\{3,4\},\{1,4\}\bigr\}$;
    \item[$\bullet$] All pairs except one: $\mathbb{S} = \bigl\{\{1,2\},\{1,3\},\{1,4\},\{2,3\},\{3,4\}\bigr\}$;
    \item[$\bullet$] All pairs: $\mathbb{S} = \bigl\{\{1,2\},\{1,3\},\{1,4\},\{2,3\},\{2,4\},\{3,4\}\bigr\}$;
    \item[$\bullet$] Single triple: $\mathbb{S} = \bigl\{\{1,2,3\},\{1,4\},\{2,4\},\{3,4\}\bigr\}$;
    \item[$\bullet$] All triples: $\mathbb{S} = \bigl\{\{1,2,3\},\{1,2,4\},\{1,3,4\},\{2,3,4\}\bigr\}$.
\end{itemize}
These patterns are illustrated in Figure~\ref{d4}.
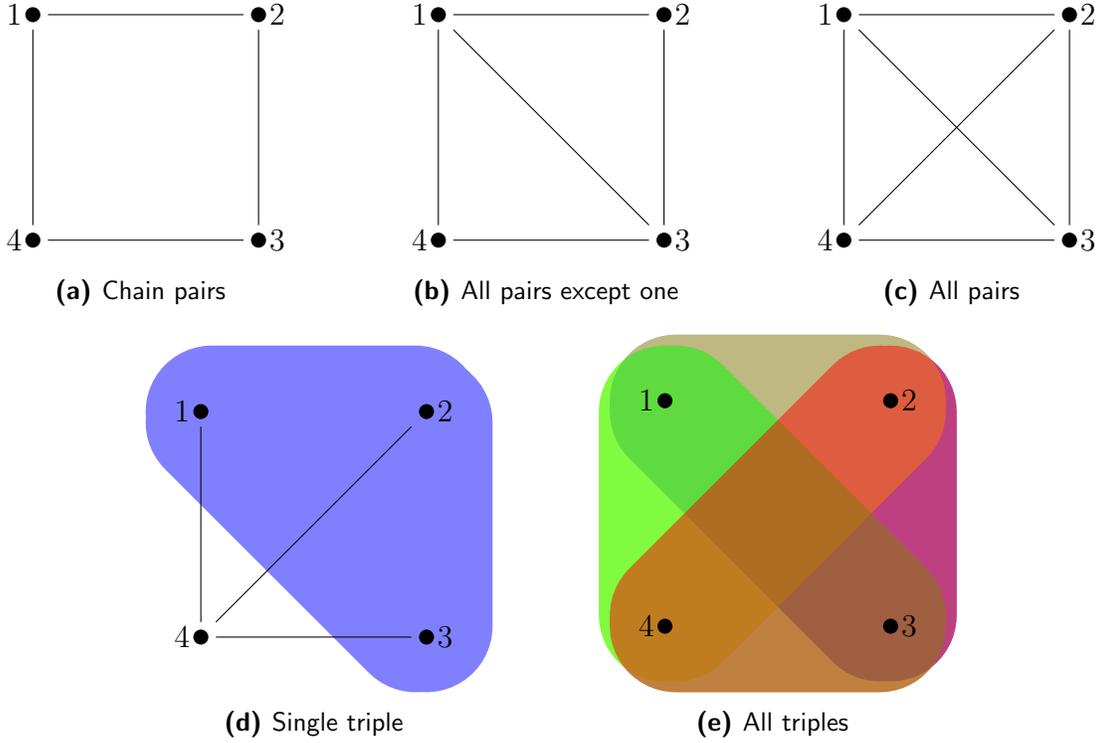
\begin{figure}
    \centering
    \subfigure[Chain pairs]{
    \begin{tikzpicture}
  \node (v1) at (0,0) {};
    \node (v2) at (3,0) {};
    \node (v3) at (3,-3) {};
    \node (v4) at (0,-3) {};
    \fill (v1) circle (0.1) node [left] {$1$};
    \fill (v2) circle (0.1) node [right] {$2$};
    \fill (v3) circle (0.1) node [right] {$3$};
    \fill (v4) circle (0.1) node [left] {$4$};
    \draw (0.2,0) -- (2.8,0);
    \draw (3,-0.2) -- (3,-2.8);
    \draw (2.8,-3) -- (0.2,-3);
    \draw (0,-2.8) -- (0,-0.2);
\end{tikzpicture}}
\hspace{1cm}
\subfigure[All pairs except one]{
    \begin{tikzpicture}
\node (v1) at (0,0) {};
    \node (v2) at (3,0) {};
    \node (v3) at (3,-3) {};
    \node (v4) at (0,-3) {};
    \fill (v1) circle (0.1) node [left] {$1$};
    \fill (v2) circle (0.1) node [right] {$2$};
    \fill (v3) circle (0.1) node [right] {$3$};
    \fill (v4) circle (0.1) node [left] {$4$};
    \draw (0.2,0) -- (2.8,0);
    \draw (3,-0.2) -- (3,-2.8);
    \draw (2.8,-3) -- (0.2,-3);
    \draw (0,-2.8) -- (0,-0.2);
    \draw (0.2,-0.2) -- (2.8,-2.8);
\end{tikzpicture}}
\hspace{1cm}
\subfigure[All pairs]{
    \begin{tikzpicture}
\node (v1) at (0,0) {};
    \node (v2) at (3,0) {};
    \node (v3) at (3,-3) {};
    \node (v4) at (0,-3) {};
    \fill (v1) circle (0.1) node [left] {$1$};
    \fill (v2) circle (0.1) node [right] {$2$};
    \fill (v3) circle (0.1) node [right] {$3$};
    \fill (v4) circle (0.1) node [left] {$4$};
    \draw (0.2,0) -- (2.8,0);
    \draw (3,-0.2) -- (3,-2.8);
    \draw (2.8,-3) -- (0.2,-3);
    \draw (0,-2.8) -- (0,-0.2);
    \draw (0.2,-0.2) -- (2.8,-2.8);
    \draw (0.2,-2.8) -- (2.8,-0.2);
\end{tikzpicture}}
\subfigure[Single triple]{
    \begin{tikzpicture}
\node (v1) at (0,0) {};
    \node (v2) at (3,0) {};
    \node (v3) at (3,-3) {};
    \node (v4) at (0,-3) {};
    \fill (v1) circle (0.1) node [left] {$1$};
    \fill (v2) circle (0.1) node [right] {$2$};
    \fill (v3) circle (0.1) node [right] {$3$};
    \fill (v4) circle (0.1) node [left] {$4$};
    \draw (2.8,-3) -- (0.2,-3);
    \draw (0,-2.8) -- (0,-0.2);
    \draw (0.2,-2.8) -- (2.8,-0.2);
    \begin{pgfonlayer}{background}
\draw[edge,color=blue] (v1) -- (v2) -- (v3) -- (v1);
\end{pgfonlayer}
\end{tikzpicture}}
\hspace{1cm}
\subfigure[All triples]{
    \begin{tikzpicture}
\node (v1) at (0,0) {};
    \node (v2) at (3,0) {};
    \node (v3) at (3,-3) {};
    \node (v4) at (0,-3) {};
    \fill (v1) circle (0.1) node [left] {$1$};
    \fill (v2) circle (0.1) node [right] {$2$};
    \fill (v3) circle (0.1) node [right] {$3$};
    \fill (v4) circle (0.1) node [left] {$4$};
    \begin{pgfonlayer}{background}
\draw[edge,color=blue] (v1) -- (v2) -- (v3) -- (v1);
\draw[edge,color=yellow] (v1) -- (v2) -- (v4) -- (v1);
\draw[edge,color=green] (v1) -- (v3) -- (v4) -- (v1);
\draw[edge,color=red] (v2) -- (v3) -- (v4) -- (v2);
\end{pgfonlayer}
\end{tikzpicture}}
\caption{\label{d4}Irreducible observation patterns $\mathbb{S}$ with $d=4$.}
\end{figure}
  
\begin{example}
\label{Prop:4dexamples}
Let $\mathcal{X}_1=\mathcal{X}_2=\mathcal{X}_3=\mathcal{X}_4=[2]$.  For $P_\mathbb{S} \in \mathcal{P}_\mathbb{S}^\mathrm{cons}$, the following statements hold:
\begin{itemize}
    \item[(i)] When $\mathbb{S}=\bigr\{\{1,2\},\{2,3\},\{3,4\},\{1,4\}\bigr\}$,
    \begin{align}
    \label{Eq:ChainStatement}
        R(P_\mathbb{S}) &= 2 \max_{k,\ell \in [2]} \biggl\{ p_{\bullet \bullet k \ell} - p_{\bullet 2 k \bullet} - \sum_{i=1}^2 \min(p_{i1 \bullet \bullet}, p_{i \bullet \bullet \ell}) \biggr\}_+ \\
        & = 2 \max_{i,j,k \in [2]}(p_{ij \bullet \bullet} - p_{\bullet j k \bullet } - p_{\bullet \bullet \bar{k} 1} - p_{i \bullet \bullet 2})_+. \nonumber
    \end{align}
    From the second representation, we see that we may take $F=8$. In fact, in this example the facet structure of $\mathcal{P}_\mathbb{S}^0$ is again closely related to the facet structure of $\mathcal{P}_\mathbb{S}^{0,*} + \mathcal{P}_\mathbb{S}^{\mathrm{cons},**}$; indeed, by \citet[][Theorem~3.5]{hocsten2002grobner},
    \[
        \mathcal{P}_\mathbb{S}^0 = \Bigl\{ P_\mathbb{S} \in \mathcal{P}_\mathbb{S}^\mathrm{cons} : \max_{i,j,k\in [2]} (p_{ij \bullet \bullet} - p_{\bullet j k \bullet } - p_{\bullet \bullet \bar{k}, 1} - p_{i \bullet \bullet 2}) \leq 0 \Bigr\}
    \]
    is a non-redundant halfspace representation. We give an analytic extension of~\eqref{Eq:ChainStatement} to $\mathcal{X}_1=[r]$ for general $r \in \mathbb{N}$ in Proposition~\ref{Prop:4dexampleAnalytic} of Section~\ref{Sec:Proofs}.
    \item[(ii)] When $\mathbb{S} =\bigr\{\{1,2\},\{2,3\},\{1,3\},\{3,4\},\{1,4\}\bigr\}$,
    \begin{align*}
        R(P_\mathbb{S}) &= 2 \max\biggl[0,\max_{j\in[2]} \biggl\{ p_{\bullet j 1 \bullet } \! - \! \sum_{i=1}^2 \! \min(p_{ij \bullet \bullet }, p_{i \bullet 1 \bullet } ) \biggr\}, \max_{\ell \in[2]} \biggl\{ p_{\bullet \bullet  1 \ell } \!-\! \sum_{i=1}^2 \! \min(p_{i \bullet 1 \bullet }, p_{i \bullet  \bullet \ell} ) \biggr\}, \\
        & \hspace{100pt} \max_{i,j,\ell \in [2]}( p_{\bullet \bullet 1 \ell } - p_{i j \bullet \bullet } - p_{\bar{i} \bullet \bullet \ell } - p_{\bullet \bar{j} 1 \bullet }) \biggr] \\
        &= \max\bigl\{ R(P_\mathbb{S}^{123}), R(P_\mathbb{S}^{134}), R(P_{\mathbb{S} \setminus \{\{1,3\}\}}) \bigr\}.
    \end{align*}
    Here, we write, e.g., $R(P_\mathbb{S}^{123})$ instead of $R(P_\mathbb{S}^{\{1,2,3\}})$ for notational simplicity.  This is a simple example where $\mathcal{P}_\mathbb{S}^{0,*} + \mathcal{P}_\mathbb{S}^{\mathrm{cons},**}$ has a more complex facet structure than that of  $\mathcal{P}_\mathbb{S}^0$.  Indeed, Proposition~\ref{Prop:CutSet} shows that $\max\bigl\{ R(P_\mathbb{S}^{123}), R(P_\mathbb{S}^{134}) \bigr\} \leq R(P_\mathbb{S}) \leq R(P_\mathbb{S}^{123}) +  R(P_\mathbb{S}^{134})$, and hence that $P_\mathbb{S}$ is compatible if and only if $P_\mathbb{S}^{123}$ and $P_\mathbb{S}^{134}$ are compatible. On the other hand, writing
    \begin{equation}
    \label{Eq:4dimReduction}
        p_{\bullet \bullet 1 \ell } - p_{i j \bullet \bullet } - p_{\bar{i} \bullet \bullet \ell } - p_{\bullet \bar{j} 1 \bullet } = (p_{\bullet \bullet 1 \ell } - p_{i \bullet 1 \bullet } - p_{\bar{i} \bullet \bullet \ell }) + (p_{\bullet j 1 \bullet } - p_{i j \bullet \bullet } - p_{\bar{i} \bullet 1 \bullet }),
    \end{equation}
    shows that our expressions for $R(P_\mathbb{S})$ are non-redundant: $\mathcal{P}_\mathbb{S}^{0,*} + \mathcal{P}_\mathbb{S}^{\mathrm{cons},**}$ has $F=16$ essential facets, while $\mathcal{P}_\mathbb{S}^0$ only has $8$.
    \item[(iii)] When $\mathbb{S} = \bigr\{\{1,2\},\{1,3\},\{1,4\},\{2,3\},\{2,4\},\{3,4\}\bigr\}$,
    \begin{align}
    \label{Eq:Chain}
        R(P_\mathbb{S}) = \max \biggl[& R(P_\mathbb{S}^{123}), R(P_\mathbb{S}^{124}),R(P_\mathbb{S}^{134}), R(P_\mathbb{S}^{234}), \nonumber \\
        &2 \max_{i,j,k,\ell \in [2]} ( -p_{ij \bullet \bullet } - p_{\bullet jk \bullet } - p_{i \bullet k \bullet } + p_{i \bullet \bullet \ell } + p_{\bullet j \bullet \ell } + p_{\bullet \bullet k \ell } - p_{\bullet \bullet \bullet \ell} ), \nonumber \\
        &R(P_{\mathbb{S} \setminus \{\{1,3\},\{2,4\}\}}), R(P_{\mathbb{S} \setminus \{ \{1,4\}, \{2,3\} \}}), R(P_{\mathbb{S} \setminus \{\{1,2\},\{3,4\}}) \biggr].
    \end{align}
    Here, we see from the \texttt{R} code output that $P_\mathbb{S}$ is compatible if and only if the first two lines of~\eqref{Eq:Chain} are non-positive. The final line can be bounded above by twice the first line: as in~\eqref{Eq:4dimReduction}, we have, for example, that
    \[
        R(P_{\mathbb{S} \setminus \{\{1,3\},\{2,4\}\}}) =\max_{i,j,\ell \in [2]}( p_{\bullet \bullet 1 \ell } - p_{i j \bullet \bullet } - p_{\bar{i} \bullet \bullet \ell } - p_{\bullet \bar{j} 1 \bullet })_+ \leq R(P_\mathbb{S}^{134}) + R(P_\mathbb{S}^{123}).
    \]
    We see from~\eqref{Eq:Chain} that we may take $F=4\times 4 + 16 + 3 \times 8 = 56$. In the $\texttt{R}$ output, the halfspace representation of $\mathcal{P}_\mathbb{S}^{0,*} + \mathcal{P}_\mathbb{S}^{\mathrm{cons},**}$ has 93 rows, 13 of which are equality constraints coming from the consistency conditions, 24 of which are non-negativity constraints, and the remaining 56 correspond to essential facets reflected in our expression $R(P_\mathbb{S})$ above. Here $\mathcal{P}_\mathbb{S}^0$ has 32 essential facets.
    \item[(iv)] When $\mathbb{S} = \bigl\{ \{1,2,3\}, \{1,4\}, \{2,4\}, \{3,4\} \bigr\}$, we have compatibility if and only if $P_\mathbb{S}^{124},P_\mathbb{S}^{134},P_\mathbb{S}^{234}$ are compatible and
    \[
        \tilde{p}_{ijk \ell} := p_{ijk \bullet } + p_{\bar{i} \bullet \bullet \ell } + p_{\bullet \bar{j} \bullet \ell } + p_{\bullet \bullet \bar{k} \ell } - p_{\bullet \bullet \bullet \ell } \geq 0
    \]
    for all $i,j,k,\ell \in [2]$. These conditions are non-redundant, so that $\mathcal{P}_\mathbb{S}^0$ has $3 \times 4 + 16 = 28$ essential facets. Further,
    \begin{align*}
        &R(P_\mathbb{S})= \max\biggl[ R(P_\mathbb{S}^{124}), R(P_\mathbb{S}^{134}), R(P_\mathbb{S}^{234}), -\frac{3}{2} \min_{i,j,k,\ell \in [2]} \tilde{p}_{ijk \ell}, -\min_{i,j,k,\ell \in [2]}( \tilde{p}_{ijk \ell} + \tilde{p}_{\bar{i}\bar{j}\bar{k}\bar{\ell}})  \\
        & - \min_{i,j,k,\ell \in [2]} \bigl\{\tilde{p}_{\bar{i}\bar{j}\bar{k}\bar{\ell}}  + \min(p_{i j \bullet \bullet } - p_{i \bullet \bullet \ell} + p_{\bullet \bar{j} \bullet \ell }, p_{\bullet jk \bullet } - p_{\bullet \bullet k \ell } +p_{\bullet \bar{j} \bullet \ell } , p_{i \bullet k \bullet } - p_{\bullet \bullet k  \ell } + p_{\bar{i} \bullet \bullet \ell } ) \bigr\} \biggr],
    \end{align*}
    so that we may take $F=3 \times 4 + 2 \times 16 + 3 \times 16 = 92$. From the above expression it also follows that 
    \[
        R(P_\mathbb{S}) \leq \max\bigl\{ R(P_\mathbb{S}^{124}), R(P_\mathbb{S}^{134}), R(P_\mathbb{S}^{234}) \bigr\} + 2 \max_{i,j,k,\ell \in [2]}(-\tilde{p}_{ijk\ell})_+.
    \]
    \item[(v)] When $\mathbb{S} = \bigl\{ \{1,2,3\}, \{1,2,4\}, \{1,3,4\}, \{2,3,4\} \bigr\}$ the marginal polytope $\mathcal{P}_\mathbb{S}^0$ has 32 essential facets. Indeed, writing $\mathbb{S}_J := \mathbb{S} \setminus ([4] \setminus \{J\})$, we have that $P_\mathbb{S}$ is compatible if and only if $P_{\mathbb{S}_J|x_J}$ is compatible for all $J \in [4]$ and $x_J \in [2]$.  Moreover,
    \begin{align*}
        R(P_\mathbb{S}) &= \frac{1}{2}\max_{J \in [4]} \max_{x_J \in [2]}\bigl\{ 3p^J(x_J) R(P_{\mathbb{S}_J | x_J}) + p^J(\bar{x}_J) R(P_{\mathbb{S}_J | \bar{x}_J}) \bigr\} \\
        & \leq 2 \max_{J \in [4]} \max_{x_J \in [2]} p^J(x_J) R(P_{\mathbb{S}_J | x_J}),
    \end{align*}
    and $F=4 \times 2 \times 4 \times 4 = 128$. To see where these numbers come from, consider $J=1$ and $x_J=1$, and note that
    \begin{align*}
        &\frac{1}{2} \{3p^1(1) R(P_{\mathbb{S}_1 | 1}) + p^1(2) R(P_{\mathbb{S}_1 | 2})\} \\
        & = -\min \Bigl\{ 0, 3 \min_{j,k \in [2]} ( p_{1jk \bullet } + p_{1, \bar{j}\bullet 1} - p_{1 \bullet k 1 }) \Bigr\} - \min\Bigl\{ 0, \min_{j',k' \in [2]} (p_{2j'k' \bullet } + p_{2, \bar{j}'\bullet 2} - p_{2 \bullet k' 2 }) \Bigr\}.
    \end{align*}
    This is the maximum of $5 \times 5$ linear functionals of $P_{\mathbb{S}_J}$, but all those where $0$ is chosen in the first term are redundant in the final expression for $R(P_\mathbb{S})$, as are all those where $(j',k')=(\bar{j},\bar{k})$. Thus, for each value of $(J,x_J)$, there are $4 \times 4$ non-redundant essential facets.
\end{itemize}
\end{example}

\section{Mixed discrete and continuous variables}
\label{Sec:Continuous}

In this section, we consider a setting of mixed discrete and continuous variables, where there exist positive integers $d_0 \leq d$ such that $\mathcal{X} = [0,1)^{d_0} \times \prod_{j=d_0+1}^d [m_j]$, with $m_1,\ldots,m_d \in \mathbb{N} \cup \{\infty\}$.  We assume that we observe independent random variables $(X_{S,i}:S \in \mathbb{S},i \in [n_S])$, with $X_{S,i} \sim P_S$ taking values in $\mathcal{X}_S := [0,1)^{S \cap [d_0]} \times \prod_{j \in S \cap ([d] \setminus [d_0])} [m_j]$.  Given a vector of bandwidths $h = (h_1,\ldots,h_{d_0}) \in (0,\infty)^{d_0}$, we partition $[0,1)^{d_0} \times \prod_{j \in [d] \setminus [d_0]} [m_j]$ as 
\[
[0,1)^{d_0} \times \prod_{j \in [d] \setminus [d_0]} [m_j] = \bigcup_{(k_1,\ldots,k_d) \in \mathcal{K}_h} \biggl(\prod_{j=1}^{d_0} I_{h_j,k_j} \times \prod_{j \in [d] \setminus [d_0]} \{k_j\}\biggr),
\]
where $\mathcal{K}_h := [\lceil 1/h_1 \rceil] \times \cdots \times [\lceil 1/h_{d_0} \rceil] \times \prod_{j \in [d] \setminus [d_0]} [m_j]$ and
\[
I_{h_j,k_j} := \bigl[(k_j-1)h_j,(k_jh_j) \wedge 1\bigr).
\]
Let $\mathcal{G}_{\mathbb{S},h}^+$ denote the set of sequences of functions $(f_S:S \in \mathbb{S})$ where each $f_S:\mathcal{X}_S \rightarrow [-1,\infty)$ is piecewise constant on all sets of the form $\prod_{j \in S \cap [d_0]} I_{h_j,k_j} \times \prod_{j \in S \cap ([d] \setminus [d_0])} \{k_j\}$ and where the sequence satisfies
\[
\inf_{x \in \mathcal{X}} \sum_{S \in \mathbb{S}} f_S(x_S) \geq 0.
\]
We further define
\[
R_h(P_\mathbb{S}) := \sup_{f_\mathbb{S} \in \mathcal{G}_{\mathbb{S},h}^+} R(P_\mathbb{S},f_\mathbb{S}).
\]
Recalling our definition of $C_\alpha^\mathrm{min}$ from Section~\ref{Sec:ImprovedTest}, in this mixed continuous and discrete setting, we reject the null hypothesis that $P_\mathbb{S} = (P_S:S \in \mathbb{S}) \in \mathcal{P}_\mathbb{S}^0$ at the level $\alpha \in (0,1)$ if
\[
R_h(\hat{P}_\mathbb{S}) \geq C_\alpha^\mathrm{min}\bigl(\lceil 1/h_1\rceil,\ldots,\lceil 1/h_{d_0}\rceil,m_{d_0+1},\ldots,m_d,\mathbb{S},(n_S:S \in \mathbb{S})\bigr) =: C_\alpha^*,
\]
where $\hat{P}_S$ is the empirical distribution of $X_{S,1},\ldots,X_{S,n_S}$ for $S \in \mathbb{S}$, and $\hat{P}_\mathbb{S} = (\hat{P}_S:S \in \mathbb{S})$.

For $d' \in \mathbb{N}$, $r = (r_1,\ldots,r_{d'}) \in (0,1]^{d'}$ and $L > 0$, let $\mathcal{H}_{d'}(r,L)$ denote the class of functions that are $(r,L)$-H\"older on $[0,1)^{d'}$, i.e. the set of functions $p:[0,1)^{d'} \rightarrow \mathbb{R}$ satisfying
\[
|p(z_1,\ldots,z_{d'}) - p(z_1',\ldots,z_{d'}')| \leq L \sum_{j=1}^{d'} |z_j-z_j'|^{r_j}
\]
 for all $(z_1,\ldots,z_{d'}), (z_1',\ldots,z_{d'}') \in [0,1)^{d'}$.  Now let $\mathcal{P}_{\mathbb{S},r,L}$ denote the set of sequences of distributions $(P_S:S \in \mathbb{S})$ where $P_S$ is a distribution on $\mathcal{X}_S$ having density $p_S$ with respect to the Cartesian product of Lebesgue measure on $[0,1)^{S \cap [d_0]}$ and counting measure on $\prod_{j \in S \cap ([d] \setminus [d_0])} [m_j]$ satisfying the condition that the conditional density $x_{S \cap [d_0]} \mapsto p_S(x_{S \cap [d_0]} | x_{S \cap ([d] \setminus [d_0])})$ belongs to $\mathcal{H}_{|S \cap [d_0]|}(r,L)$ for all $x_{S \cap ([d] \setminus [d_0])} \in \prod_{j \in S \cap ([d] \setminus [d_0])} [m_j]$.
\begin{thm}
\label{Thm:ContinuousUpperBound}
In the above setting, let $\alpha, \beta \in (0,1)$ and suppose that $P_\mathbb{S} \in \mathcal{P}_{\mathbb{S},r,L}$.  Then the probability of a Type I error for our test is at most $\alpha$.  Moreover, if
\[
R(P_\mathbb{S}) \geq L (|\mathbb{S}|-1)\sum_{j=1}^{d_0} h_j^{r_j} + C_\alpha^* + C_\beta^*,
\]
then the probability of a Type II error is at most $\beta$.
\end{thm}

We now specialise the upper bound of Theorem~\ref{Thm:ContinuousUpperBound} to our main three-dimensional example. When $\mathbb{S}=\bigl\{\{1,2\},\{2,3\},\{1,3\}\bigr\}$ and $\mathcal{X} = [0,1)^2 \times \{1,2\}$ we have for $h_1,h_2 \in (0,1)$ that
\[
    L(|\mathbb{S}|-1)(h_1^{r_1} + h_2^{r_2}) + C_\alpha^* + C_\beta^* \lesssim_{L,|\mathbb{S}|,\alpha,\beta} h_1^{r_1} + h_2^{r_2} + \Bigl( \frac{1/h_1 + 1/h_2}{\min_{S \in \mathbb{S}} n_S} \Bigr)^{1/2},
\]
and we can choose $h_1,h_2$ to minimise this right-hand side. We can take $h_1=h_2 = n^{- \frac{1}{1+2(r_1 \wedge r_2)}}$ and $\alpha=\beta=1/4$ to deduce the minimax upper bound
\[
    \rho^*(n_\mathbb{S}) \lesssim_{L,|\mathbb{S}|} \Bigl( \min_{S \in \mathbb{S}} n_S \Bigr)^{-\frac{r_1 \wedge r_2}{1+ 2(r_1 \wedge r_2)}}.
\]

\section{Numerical studies}
\label{Sec:Numerics}

The tests introduced in Section~\ref{Sec:TestingCompatibility} provide finite-sample Type I error control over the entire null hypothesis parameter space $\mathcal{P}_\mathbb{S}^0$. However, this may lead to conservative tests in particular examples, so we first present an alternative, Monte Carlo-based approach to constructing the critical value for our test.  The first part of Proposition~\ref{Prop:L1Projection} and the dual formulation~\eqref{Eq:InfR} mean that we can write
\[
    \hat{P}_\mathbb{S} = \{1-R(\hat{P}_\mathbb{S})\} \hat{Q}_\mathbb{S} + R(\hat{P}_\mathbb{S}) \hat{T}_\mathbb{S},
\]
where $\hat{Q}_\mathbb{S} \in \mathcal{P}_\mathbb{S}^0$ and $\hat{T}_\mathbb{S} \in \mathcal{P}_\mathbb{S}$. Here $\hat{Q}_\mathbb{S}$ can be thought of as a closest compatible sequence of marginal distributions to $\hat{P}_\mathbb{S}$ (in particular, if $\hat{P}_\mathbb{S} \in \mathcal{P}_\mathbb{S}^0$, then $\hat{Q}_\mathbb{S} = \hat{P}_\mathbb{S}$). Moreover, $\hat{Q}_\mathbb{S}$ can be computed straightforwardly at the same time as our test statistic $R(\hat{P}_\mathbb{S})$. It is therefore natural to generate a critical value by drawing $B$ bootstrap samples from $\hat{Q}_\mathbb{S}$, computing the corresponding empirical distributions $\hat{Q}_\mathbb{S}^{(1)},\ldots,\hat{Q}_\mathbb{S}^{(B)}$ and test statistics $R(\hat{Q}_\mathbb{S}^{(1)}),\ldots,R(\hat{Q}_\mathbb{S}^{(B)})$, and rejecting $H_0'$ at significance level $\alpha \in (0,1)$ if and only if
\[
    1 + \sum_{b=1}^B \mathbbm{1}_{\{R(\hat{Q}_\mathbb{S}^{(b)}) \leq R(\hat{Q}_\mathbb{S})\} } \leq \alpha(B+1).
\]
In our first experiments, we took $\mathbb{S}=\bigl\{\{1,2\},\{2,3\},\{1,3\}\bigr\}$ with $\mathcal{X}=[r] \times [2]^2$ and $r \in \{2,4,6\}$; we fix $P_\mathbb{S} \in \mathcal{P}_\mathbb{S}^\mathrm{cons}$ by setting, for each $i \in [r]$,
\begin{equation}
    \label{Eq:SimEqualities}
    p_{i \bullet \bullet } = \frac{1}{r}, \quad p_{\bullet 1 \bullet } = p_{\bullet \bullet 1 } = \frac{1}{2}, \quad p_{i \bullet 1}=\frac{1}{2r}, \quad p_{i1 \bullet } = \frac{ 1+(-1)^i}{2r}
\end{equation}
and varying $p_{\bullet 2 1}$ to adjust the incompatibility index. Indeed, with these choices, we have $R(P_\mathbb{S})=2(p_{\bullet 2 1} - 1/4)_+$ by Theorem~\ref{Prop:rs2example}.  Our Monte Carlo test was applied with $n_\mathbb{S}=(200,200,200)$, $B=99$ and $\alpha=0.05$, and we repeated our experiments 5000 times in each setting. 

We are not aware of alternative methods that can be applied directly in this context, but the test of \citet{fuchs1982maximum} can be used if there are also complete cases available.  In order to provide some comparison, then, we gave the Fuchs method access to an additional $n_{\{1,2,3\}} = 200$ observations from the distribution on $\mathcal{X}$ with mass function $p_{ijk}=\{1+(-1)^{i+j}\}/(4r)$ for $i \in [r]$ and $j,k \in [2]$, which ensures that $\mathbb{A}p$ is a closest compatible sequence to $P_\mathbb{S}$, in our terminology above.  In particular, $\mathbb{A}p$ satisfies all equalities in~\eqref{Eq:SimEqualities}, as well as $(\mathbb{A}p)_{\bullet 21} = 1/4$.  We emphasise that these complete cases were not accessed by our method.   


\begin{figure}
         \centering
         \includegraphics[width=\textwidth]{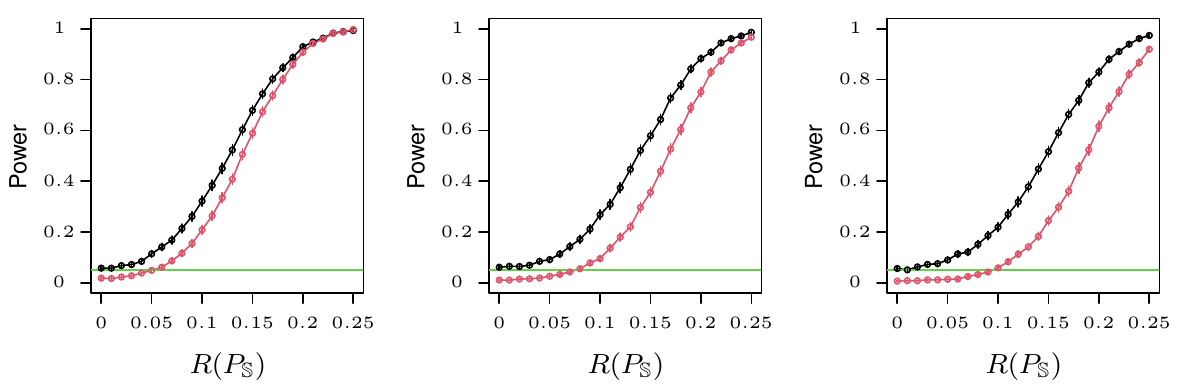}
        \caption{Power curves for our Monte Carlo test (black) and Fuchs's test (red). Error bars show three standard errors.  Here, $\mathbb{S}=\bigl\{\{1,2\},\{2,3\},\{1,3\}\bigr\}$ with $\mathcal{X}=[r] \times [2]^2$ and $r =2$ (left), $r=4$ (middle) and $r=6$ (right).}
        \label{Fig:Simulations}
\end{figure}
Figure~\ref{Fig:Simulations} presents the power curves of the two tests for the three different choices of~$r$.  We see that both tests have good control of the size of the test, and in fact the Fuchs test is slightly conservative.  Despite the extra complete cases that are available to the Fuchs method, though, our test is significantly more powerful, with the difference in power increasing as $r$ increases.

In our second set of experiments, we took $d=5$, $n_\mathbb{S}=(500,500,500,500,500)$, $\mathcal{X} = [2]^5$ and $\mathbb{S} = \bigl\{\{1,2,3,4\},\{1,2,3,5\},\{1,2,4,5\},\{1,3,4,5\},\{2,3,4,5\}\bigr\}$.  For $\epsilon \in [0.2,0.35]$ and $i,j,k,\ell,m \in [2]$, we set
\begin{align*}
p_{ijk\ell\bullet} &= \frac{1 + \epsilon(-1)^{i+j+k+\ell}}{16}, \quad p_{ijk\bullet m} = \frac{1 + \epsilon(-1)^{i+j+k+m}}{16}, \quad p_{ij\bullet\ell m} = \frac{1 + \epsilon(-1)^{i+j+\ell+m}}{16} \\  p_{i\bullet k\ell m} &= \frac{1 + \epsilon(-1)^{i+k+\ell+m}}{16}, \quad p_{\bullet j k\ell m} = \frac{1 - \epsilon(-1)^{j+k+\ell+m}}{16},
\end{align*}
for which $R(P_{\mathbb{S}}) = (5\epsilon-1)_+/4$.  In this case, we applied the Fuchs test for several different choices of the number of complete cases, namely $n_{\{1,2,3,4,5\}} \in \{25,50,100,200\}$. The complete case distribution $p$ was chosen so that $\mathbb{A} p$ was a closest compatible sequence to $P_\mathbb{S}$. Figure~\ref{Fig:Simulations2} shows the corresponding power curves, along with that of our test.  In this example, our test is the only one that controls the Type I error at the nominal level, so none of Fuchs tests are reliable here.  We also see that the additional complete cases are crucial for the power of the Fuchs test, and that the power of our test remains competitive even without these observations.
\begin{figure}
         \centering
         \includegraphics[width=0.6\textwidth]{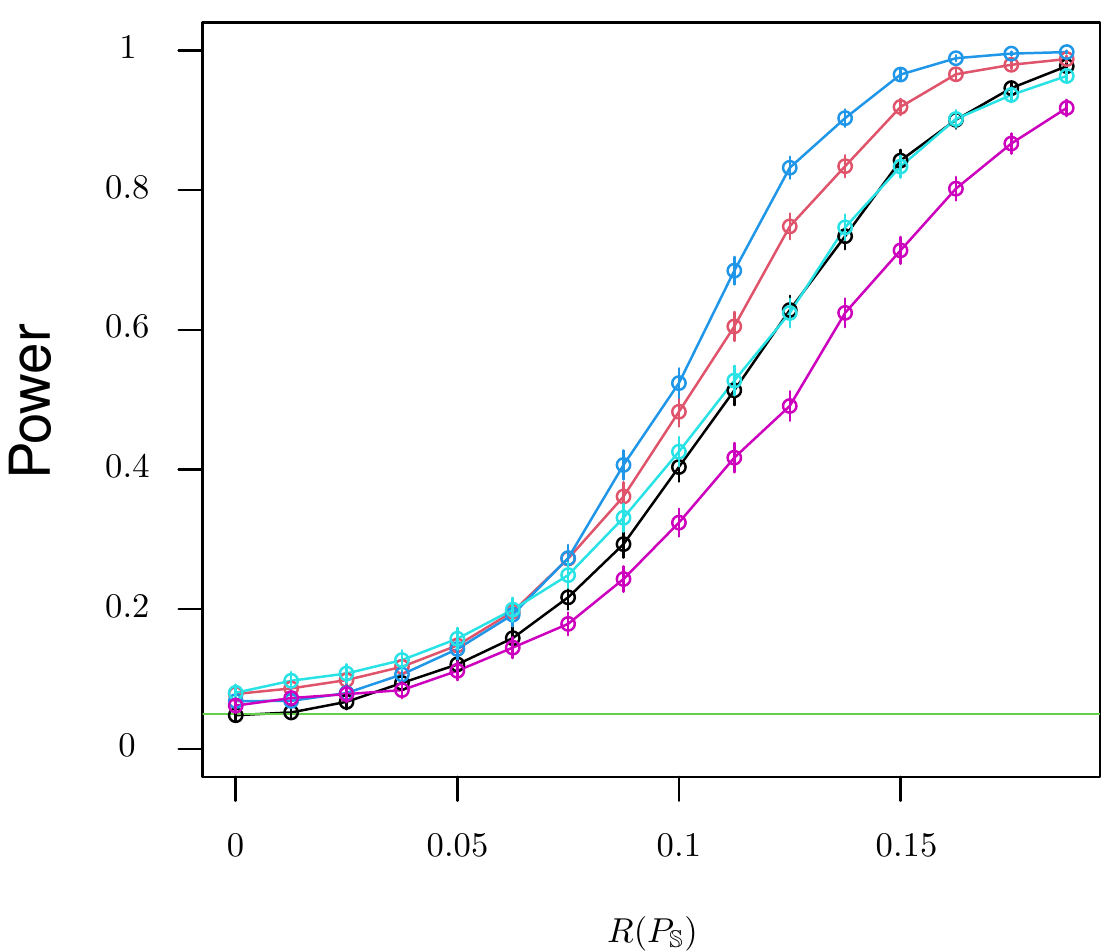}
     \caption{Power curves for our Monte Carlo test (black) and Fuchs's test with the latter test being applied with an additional $n_{\{1,2,3,4,5\}} = 25$ (magenta), $50$ (cyan), $100$ (red) and $200$ (blue) complete cases. Error bars show three standard errors.  Here, $\mathbb{S}=\bigl\{\{1,2,3,4\},\{1,2,3,5\},\{1,2,4,5\},\{1,3,4,5\},\{2,3,4,5\}\bigr\}$, with $\mathcal{X}=[2]^5$.}
        \label{Fig:Simulations2}
\end{figure}

\section{Proofs and auxiliary results}
\label{Sec:Proofs}

\begin{proof}[Proof of Theorem~\ref{Thm:DualRepresentation}]
We apply the idea of \emph{Alexandroff (one-point) compactification}  \citep{alexandroff1924metrisation}.    Specifically, writing $\mathcal{J} := \{j \in [d]: \mathcal{X}_j \text{ is not compact}\}$, for each $j \in \mathcal{J}$, we can construct a one-point enlarged space $\mathcal{X}_j^* := \mathcal{X}_j \cup \{\infty_j\}$ (where $\infty_j \notin \mathcal{X}_j$), and take as a topology on $\mathcal{X}_j^*$ all open subsets of $\mathcal{X}_j$ together with all sets of the form $(\mathcal{X}_j \setminus K) \cup \{\infty_j\}$, where $K$ is compact in~$\mathcal{X}_j$.  With this topology, $\mathcal{X}_j^*$ is a compact, Hausdorff space \citep[][Proposition~4.36]{folland1999real}.  We also set $\mathcal{X}_j^* := \mathcal{X}_j$ for $j \in [d] \setminus \mathcal{J}$.  We can extend each probability measure $P_S$ to a Borel probability measure $P_S^*$ on $\mathcal{X}_S^* := \prod_{j \in S} \mathcal{X}_j^*$ (equipped with the product topology) by setting $P_S^*(B) := P_S(B \cap \mathcal{X}_S)$ for all Borel subsets $B$ of~$\mathcal{X}_S^*$.

Now, suppose that $f_\mathbb{S} \in \mathcal{G}_\mathbb{S}^+(\mathcal{X}_\mathbb{S})$ satisfies $f_S \leq |\mathbb{S}|-1$ for all $S \in \mathbb{S}$.  We extend each $f_S$ to a function $f_S^*$ on $\mathcal{X}_S^*$ by defining
\[
f_S^*(x_S^*) := \left\{ \begin{array}{ll} f_S(x_S^*) & \mbox{if $x_j^* \in \mathcal{X}_j$ \text{for all} $j \in S$} \\ |\mathbb{S}|-1 & \mbox{otherwise.} \end{array} \right.
\]
To see that $f_S^*$ is upper semi-continuous,  first suppose that $x_S^* \in \mathcal{X}_S$ and $y > f_S^*(x_S^*) = f_S(x_S^*)$.  Since $f_S$ is upper semi-continuous and all sets that are open in $\mathcal{X}_S$ are open in $\mathcal{X}_S^*$, there exists a neighbourhood $U \subseteq \mathcal{X}_S^*$ of $x_S^*$ such that $f_S^*(x_S) < y$ for all $x_S \in U$.    On the other hand, if $x_S^* \in \mathcal{X}_S^* \setminus \mathcal{X}_S$ and $y > f_S^*(x_S^*) = |\mathbb{S}|-1$, then we can take the neighbourhood $U = \mathcal{X}_S^*$ to see that $f_S^*(x_S) < y$ for all $x_S \in U$.  This establishes that $f_S^*$ is indeed upper semi-continuous.  Writing $\mathcal{X}^* := \prod_{j \in [d]} \mathcal{X}_j^*$, we also have that 
\[
\inf_{x^* \in \mathcal{X}^*} \sum_{S \in \mathbb{S}} f_S^*(x_S^*) \geq \min \biggl\{ 0,  \inf_{x \in \mathcal{X}} \sum_{S \in \mathbb{S}} f_S(x_S) \biggr\} = 0,
\]
so $f_\mathbb{S}^* \in \mathcal{G}_\mathbb{S}^+(\mathcal{X}_\mathbb{S}^*)$.  Moreover,
\begin{equation}
\label{Eq:XXstar}
R_{\mathcal{X}_\mathbb{S}}(P_\mathbb{S},f_\mathbb{S}) = R_{\mathcal{X}_\mathbb{S}^*}(P_\mathbb{S}^*,f_\mathbb{S}^*).
\end{equation}
In the other direction, given any $f_\mathbb{S}^* \in \mathcal{G}_\mathbb{S}^+(\mathcal{X}_\mathbb{S}^*)$, we can define $f_\mathbb{S} = (f_S:S \in \mathbb{S})$ on $\mathcal{X}_\mathbb{S}$ by defining each $f_S$ to be the restriction of $f_S^*$ to $\mathcal{X}_S$.  Then, for each $t \in \mathbb{R}$, 
\[
(f_S)^{-1}\bigl([t,\infty)\bigr) = (f_S^*)^{-1}\bigl([t,\infty)\bigr) \cap \mathcal{X}_S,
\]
so $(f_S)^{-1}\bigl([t,\infty)\bigr)$ is a closed subset of $\mathcal{X}_S$ and $f_S$ is upper semi-continuous.  Moreover,
\[
\inf_{x \in \mathcal{X}} \sum_{S \in \mathbb{S}} f_S(x_S) \geq \inf_{x^* \in \mathcal{X}^*} \sum_{S \in \mathbb{S}} f_S^*(x_S^*) \geq 0,
\]
so $f_\mathbb{S} \in \mathcal{G}_\mathbb{S}^+(\mathcal{X}_\mathbb{S})$.  Again, the equality~\eqref{Eq:XXstar} holds.  We deduce that
\begin{align}
\label{Eq:RXXstar}
R_{\mathcal{X}_\mathbb{S}}(P_\mathbb{S}) &= \sup\bigl\{R_{\mathcal{X}_\mathbb{S}}(P_\mathbb{S},f_\mathbb{S}):f_{\mathbb{S}} \in \mathcal{G}_\mathbb{S}^+(\mathcal{X}_\mathbb{S})\bigr\} \nonumber \\
&= \sup\bigl\{R_{\mathcal{X}_\mathbb{S}^*}(P_\mathbb{S}^*,f_\mathbb{S}^*):f_{\mathbb{S}}^* \in \mathcal{G}_\mathbb{S}^+(\mathcal{X}_\mathbb{S}^*)\bigr\} = R_{\mathcal{X}_\mathbb{S}^*}(P_\mathbb{S}^*).
\end{align}
Now let $\mathcal{C}_\mathbb{S}^+(\mathcal{X}_\mathbb{S}^*)$ denote the subset of continuous functions in  $\mathcal{G}_\mathbb{S}^+(\mathcal{X}_\mathbb{S}^*)$.  Since compact Hausdorff spaces are completely regular, by \citet[][Proposition~1.33 and an inspection of the proof of  Proposition~3.13]{kellerer1984duality}, we have
\[
    R_{\mathcal{X}_\mathbb{S}}(P_\mathbb{S}) = R_{\mathcal{X}_\mathbb{S}^*}(P_\mathbb{S}^*) = \sup \bigl\{ R_{\mathcal{X}_\mathbb{S}^*}(P_\mathbb{S}^*,f_\mathbb{S}^*) : f_\mathbb{S}^* \in \mathcal{C}_\mathbb{S}^+(\mathcal{X}_\mathbb{S}^*) \bigr\}.
\]
Having established that $R_{\mathcal{X}_\mathbb{S}}(P_\mathbb{S})$ may be computed as a supremum over functions defined on compact spaces, we now consider the implications for the dual representation of the one-point compactification.  Suppose that $\epsilon \in [0,1]$ is such that $P_\mathbb{S} \in (1-\epsilon)\mathcal{P}_\mathbb{S}^0(\mathcal{X}_\mathbb{S}) + \epsilon \mathcal{P}_\mathbb{S}(\mathcal{X}_\mathbb{S})$.  Then $P_\mathbb{S} = (1-\epsilon)Q_\mathbb{S} + \epsilon T_\mathbb{S}$, where $Q_\mathbb{S} \in \mathcal{P}_\mathbb{S}^0(\mathcal{X}_\mathbb{S})$ and $T_\mathbb{S} \in \mathcal{P}_\mathbb{S}(\mathcal{X}_\mathbb{S})$.  For each $S \in \mathbb{S}$, we define probability measures $Q_S^*,T_S^*$ on $\mathcal{X}_S^*$ by $Q_S^*(B) := Q_S(B \cap \mathcal{X}_S)$ and $T_S^*(B) := T_S(B \cap \mathcal{X}_S)$ for all Borel subsets $B$ of~$\mathcal{X}_S^*$.  Then $Q_\mathbb{S}^* \in \mathcal{P}_\mathbb{S}^0(\mathcal{X}_\mathbb{S}^*)$, because $R_{\mathcal{X}_\mathbb{S}^*}(Q_\mathbb{S}^*) = R_{\mathcal{X}_\mathbb{S}}(Q_\mathbb{S}) = 0$ from~\eqref{Eq:RXXstar} and the fact that $Q_\mathbb{S} \in \mathcal{P}_\mathbb{S}^0(\mathcal{X}_\mathbb{S})$.  Hence $P_\mathbb{S}^* = (1-\epsilon)Q_\mathbb{S}^* + \epsilon T_\mathbb{S}^* \in (1-\epsilon)\mathcal{P}_\mathbb{S}^0(\mathcal{X}_\mathbb{S}^*) + \epsilon \mathcal{P}_\mathbb{S}(\mathcal{X}_\mathbb{S}^*)$. 

Conversely, suppose initially that $\epsilon \in (0,1)$ is such that $P_\mathbb{S}^* \in (1-\epsilon)\mathcal{P}_\mathbb{S}^0(\mathcal{X}_\mathbb{S}^*) + \epsilon \mathcal{P}_\mathbb{S}(\mathcal{X}_\mathbb{S}^*)$, so that $P_\mathbb{S}^* = (1-\epsilon)Q_\mathbb{S}^* + \epsilon T_\mathbb{S}^*$, where $Q_\mathbb{S}^* \in \mathcal{P}_\mathbb{S}^0(\mathcal{X}_\mathbb{S}^*)$ and $T_\mathbb{S}^* \in \mathcal{P}_\mathbb{S}(\mathcal{X}_\mathbb{S}^*)$.  Observe that we must have $Q_S^*(B) = Q_S^*(B \cap \mathcal{X}_S)$ and $T_S^*(B) = T_S^*(B \cap \mathcal{X}_S)$ for all $S \in \mathbb{S}$ and all Borel subsets $B \subseteq \mathcal{X}_\mathbb{S}^*$, because $P_S^*$ does not put any mass outside $\mathcal{X}_S$.  Then we can define families of probability measures $Q_\mathbb{S} = (Q_S:S \in \mathbb{S})$ and $T_\mathbb{S} = (T_S:S \in \mathbb{S})$ by $Q_S(B) := Q_S^*(B)$ and $T_S(B) := T_S^*(B)$ for each $S \in \mathbb{S}$ and each Borel subset $B$ of $\mathcal{X}_\mathbb{S}$, and have $P_\mathbb{S} = (1-\epsilon)Q_\mathbb{S} + \epsilon T_\mathbb{S} \in (1-\epsilon)\mathcal{P}_\mathbb{S}^0(\mathcal{X}_\mathbb{S}) + \epsilon \mathcal{P}_\mathbb{S}(\mathcal{X}_\mathbb{S})$.  The boundary cases $\epsilon \in \{0,1\}$ can also be handled similarly, and we deduce that   
\[
\inf \bigl\{ \epsilon \in [0,1] \!:\! P_\mathbb{S} \in (1-\epsilon) \mathcal{P}_\mathbb{S}^0(\mathcal{X}_\mathbb{S}) + \epsilon \mathcal{P}_\mathbb{S}(\mathcal{X}_\mathbb{S})\bigr\} = \inf \bigl\{ \epsilon \in [0,1] \!:\! P_\mathbb{S}^* \in (1-\epsilon) \mathcal{P}_\mathbb{S}^0(\mathcal{X}_\mathbb{S}^*) + \epsilon \mathcal{P}_\mathbb{S}(\mathcal{X}_\mathbb{S}^*)\bigr\}.
\]

The upshot of this argument is that we may assume without loss of generality that each~$\mathcal{X}_j$ is a compact Hausdorff space (not just locally compact), so that 
\[
R(P_\mathbb{S}) = \sup \bigl\{ R(P_\mathbb{S},f_\mathbb{S}) : f_\mathbb{S} \in \mathcal{C}_\mathbb{S}^+\bigr\},
\]
where we now have suppressed the dependence of these quantities on $\mathcal{X}_\mathbb{S}$.  We now seek to apply \citet[][Theorem~2.3]{isii1964inequalities} to rewrite this expression for $R(P_\mathbb{S})$ in its dual form; this will require some further definitions.  Let
\begin{align*}
    X &:= \{g_\mathbb{S} = (g_S : S \in \mathbb{S}) : g_S: \mathcal{X}_S \rightarrow [0,\infty) \text{ is continuous for all } S \in \mathbb{S} \},
\end{align*}
let $Z$ denote the set of real-valued, continuous functions on $\mathcal{X}$ endowed with the supremum norm topology, 
let $\mathcal{C} \subseteq Z$ denote those elements of $Z$ that are non-negative, let $\psi:X \rightarrow Z$ be given by $\psi(g_\mathbb{S})(x) :=(1/|\mathbb{S}|) \sum_{S \in \mathbb{S}} g_S(x_S)$, and let $\phi: X \rightarrow \mathbb{R}$ be given by $\phi(g_\mathbb{S}) := - (1/|\mathbb{S}|) \sum_{S \in \mathbb{S}} \int g_S \, dP_S$. Now $\mathcal{C}$ is a convex cone with non-empty interior. Moreover, for any $g \in Z$ we can take $g_\mathbb{S} = \|g\|_\infty$ and $g' = \|g\|_\infty - g \in \mathcal{C}$ to see that $\psi(g_\mathbb{S}) - g' = \|g\|_\infty - g' = g$, and so $\psi(X) - \mathcal{C} = Z$. This shows that Assumption A of \citet{isii1964inequalities} holds. Since $X$ is a convex cone and $\phi$ and $\psi$ are linear we see that the conditions of  \citet[][Theorem~2.3]{isii1964inequalities} are satisfied. Now, $\mathcal{X}$ is compact by Tychanov's theorem~\citep[e.g.][Theorem~4.42]{folland1999real} (which is equivalent to the axiom of choice), so by a version of the Riesz representation theorem~\citep[e.g.][Theorem~7.2]{folland1999real}, the set of non-negative elements of the continuous dual $Z^*$ of $Z$ is the set of Radon measures on $\mathcal{X}$, denoted $\mathcal{M}_+(\mathcal{X})$.  Thus, writing $\mu^S$ for the marginal measure on $\mathcal{X}_S$ of $\mu \in \mathcal{M}_+(\mathcal{X})$, we have
\begin{align}
\label{Eq:Duality}
    R(P_\mathbb{S}) &= 1 + \sup\{ \phi(g_\mathbb{S}) : g_\mathbb{S} \in X, \psi(g_\mathbb{S}) - 1 \geq 0 \} \nonumber \\
    & = 1 + \inf\bigl\{ z^*( -1) : z^* \in Z^*, z^* \geq 0, z^*\bigl( \psi(g_\mathbb{S})\bigr) + \phi(g_\mathbb{S}) \leq 0 \text{ for all } g_\mathbb{S} \in X \bigr\} \nonumber \\
    & = 1 + \inf \biggl\{ - \mu(\mathcal{X}) : \mu \in \mathcal{M}_+(\mathcal{X}), \int_{\mathcal{X}} \biggl(\sum_{S \in \mathbb{S}} g_S \biggr) d \mu \leq \sum_{S \in \mathbb{S}} \int_{\mathcal{X}_S} g_S \, dP_S \text{ for all } g_\mathbb{S} \in X\biggr\} \nonumber \\
    &= 1 - \sup\biggl\{\mu(\mathcal{X}) : \mu \in \mathcal{M}_+(\mathcal{X}), \int_{\mathcal{X}_S} g_S \, d \mu^S \leq \int_{\mathcal{X}_S} g_S \, dP_S \text{ for all } S \in \mathbb{S}, g_\mathbb{S} \in X\biggr\}. 
\end{align}
We finally claim that this last display is equal to the claimed form in the statement of the result.  Let $\epsilon \in [0,1]$ be such that $P_\mathbb{S} \in (1-\epsilon) \mathcal{P}_\mathbb{S}^0 + \epsilon \mathcal{P}_\mathbb{S}$.  Then there exists a probability measure $\mu$ on $\mathcal{X}$ with marginals $\mu_\mathbb{S} := (\mu^S:S \in \mathbb{S})$ for which we can write $P_\mathbb{S} = (1-\epsilon)\mu_\mathbb{S} + \epsilon Q_\mathbb{S}$, where $Q_\mathbb{S} \in \mathcal{P}_\mathbb{S}$.  Since every open set in $\mathcal{X}$ is $\sigma$-compact, the probability measure $\mu$ is necessarily Radon \citep[][Theorem~7.8]{folland1999real}.  Now for all $S \in \mathbb{S}$, and $g_\mathbb{S} \in X$,
\[
    (1-\epsilon) \int_{\mathcal{X}_S} g_S \, d \mu^S = \int_{\mathcal{X}_S} g_S \, d(P_S - \epsilon Q_S) \leq \int_{\mathcal{X}_S} g_S \, dP_S
\]
so $(1-\epsilon) \mu$ is feasible and we deduce from~\eqref{Eq:Duality} that $R(P_\mathbb{S}) \leq \epsilon$.  Hence $R(P_\mathbb{S}) \leq \inf\bigl\{ \epsilon \in [0,1] : P_\mathbb{S} \in (1-\epsilon) \mathcal{P}_\mathbb{S}^0 + \epsilon \mathcal{P}_\mathbb{S}\bigr\}$.  For the bound in the other direction, first suppose that $R(P_\mathbb{S}) = 1$.  Then, from~\eqref{Eq:Duality}, the only element $\mu$ of $\mathcal{M}_+(\mathcal{X})$ satisfying $\int_{\mathcal{X}_S} g_S \, d \mu^S \leq \int_{\mathcal{X}_S} g_S \, dP_S$ for all $S \in \mathbb{S}, g_\mathbb{S} \in X$ is the zero measure on $\mathcal{X}$.  If $\epsilon \in [0,1]$ is such that $P_\mathbb{S} = (1-\epsilon)Q_\mathbb{S} + \epsilon T_\mathbb{S}$ with $Q_\mathbb{S} \in \mathcal{P}_\mathbb{S}^0$ and $T_\mathbb{S} \in \mathcal{P}_\mathbb{S}$, then for any $S \in \mathbb{S}$ and $g_\mathbb{S} \in X$,
\[
\int_{\mathcal{X}_S} g_S \, d(1-\epsilon)Q_S \leq\int_{\mathcal{X}_S} g_S \, dP_S.
\]
It follows that $(1-\epsilon)Q_\mathbb{S} \in \mathcal{M}_+(\mathcal{X})$ must be the zero measure, so $\epsilon = 1$.  Hence, when $R(P_\mathbb{S}) = 1$, we also have $\inf\bigl\{ \epsilon \in [0,1] : P_\mathbb{S} \in (1-\epsilon) \mathcal{P}_\mathbb{S}^0 + \epsilon \mathcal{P}_\mathbb{S}\bigr\} = 1$.  Now suppose that $R(P_\mathbb{S}) < 1$, so by~\eqref{Eq:Duality}, given $\delta \in \bigl(0,1-R(P_\mathbb{S})\bigr)$, we can find $\mu \in \mathcal{M}_+(\mathcal{X})$ with marginals $(\mu^S:S \in \mathbb{S})$ that satisfies $\int_{\mathcal{X}_S} g_S \, d \mu^S \leq \int_{\mathcal{X}_S} g_S \, dP_S$ for all $S \in \mathbb{S}, g_\mathbb{S} \in X$ and $\mu(\mathcal{X}) = 1 - R(P_\mathbb{S}) - \delta$. Writing $\epsilon := 1 - \mu(\mathcal{X}) = R(P_\mathbb{S}) + \delta$, let $Q_\mathbb{S} := (\mu^S/(1-\epsilon) : S \in \mathbb{S}) \in \mathcal{P}_\mathbb{S}^0$, and let $T_\mathbb{S} := \epsilon^{-1}\bigl(P_\mathbb{S} - (1-\epsilon) Q_\mathbb{S}\bigr)$.  Then $T_S(\mathcal{X}_S) = 1$ for all $S \in \mathbb{S}$, and for any $S \in \mathbb{S}$ and $g_\mathbb{S} \in X$,
\[
    \int_{\mathcal{X}_S} g_S \, dT_S = \frac{1}{\epsilon} \int_{\mathcal{X}_S} g_S \, d(P_S - \mu^S) \geq 0.
\]
Thus $T_S$ is a probability measure on $\mathcal{X}_S$ for all $S \in \mathbb{S}$, so $T_\mathbb{S} \in \mathcal{P}_\mathbb{S}$ and $P_\mathbb{S} \in (1-\epsilon) \mathcal{P}_\mathbb{S}^0 + \epsilon \mathcal{P}_\mathbb{S}$.  Since $\delta \in \bigl(0,1-R(P_\mathbb{S})\bigr)$ was arbitrary, we deduce that $\inf\bigl\{ \epsilon \in [0,1] : P_\mathbb{S} \in (1-\epsilon) \mathcal{P}_\mathbb{S}^0 + \epsilon \mathcal{P}_\mathbb{S}\bigr\} \leq R(P_\mathbb{S})$.  This completes the proof.
\end{proof}

\begin{proof}[Proof of Proposition~\ref{Prop:DiscreteTest1}]

Our strategy here is to apply results on the concentration properties and the mean of the supremum $R( \hat{P}_{\mathbb{S}})$ of the empirical process 
\begin{equation}
\label{Eq:EmpiricalProcess}
R( \hat{P}_{\mathbb{S}},f_{\mathbb{S}}) = - \frac{1}{|\mathbb{S}|} \sum_{S \in \mathbb{S}} \frac{1}{n_S}\sum_{i=1}^{n_S} f_S(X_{S,i})
\end{equation}
over $f_{\mathbb{S}} \in \mathcal{G}_{\mathbb{S}}^+$. If $f_{\mathbb{S}} \in \mathcal{G}_{\mathbb{S}}^+$, then $\min(f_{\mathbb{S}},|\mathbb{S}|-1) \in \mathcal{G}_{\mathbb{S}}^+$, because if this were not the case, then there would exist $x^0 = (x_S^0 : S \in \mathbb{S}) \in \mathcal{X}$ and $S_0 \in \mathbb{S}$ with $f_{S_0}(x_{S_0}^0) > |\mathbb{S}|-1$ such that
\[
\sum_{S \in \mathbb{S}} \min\bigl\{f_S(x_S^0), |\mathbb{S}|-1\bigr\} < 0. 
\]
But, since $f_{\mathbb{S}} \geq -1$, we would then have
\[
\sum_{S \in \mathbb{S}} \min\bigl\{f_S(x_S^0), |\mathbb{S}|-1\bigr\} > |\mathbb{S}|-1 + \sum_{S \in \mathbb{S}:S \neq S_0} f_S(x_S^0) \geq 0,
\]
a contradiction.  Since $R\bigl(P_{\mathbb{S}},\min(f_{\mathbb{S}},|\mathbb{S}|-1)\bigr) \geq R(P_{\mathbb{S}},f_{\mathbb{S}})$, it follows that, in seeking a maximiser in~\eqref{Eq:EmpiricalProcess}, we may restrict our optimisation to $\bigl\{f_{\mathbb{S}} \in \mathcal{G}_{\mathbb{S}}^+: f_{\mathbb{S}} \leq |\mathbb{S}|-1\bigr\}$. 

Writing $V:=\sum_{S \in \mathbb{S}} n_S^{-1}$, by \citet[][Theorem~12.1]{boucheron2013concentration} --- a consequence of the bounded differences (McDiarmid's) inequality --- for any collection $P_{\mathbb{S}}$ and $\lambda \in \mathbb{R}$, we have
\[
    \log \mathbb{E}\exp\bigl( \lambda \bigl\{ R(\hat{P}_{\mathbb{S}}) - \mathbb{E} R (\hat{P}_{\mathbb{S}})\bigr\}\bigr) \leq \frac{V \lambda^2}{8}.
\]
In particular, by the usual sub-Gaussian tail bound,
\begin{align}
\label{Eq:DiscreteConcentration}
    \max \bigl\{ \mathbb{P}\bigl( R(\hat{P}_{\mathbb{S}}) - \mathbb{E} R (\hat{P}_{\mathbb{S}}) \leq - t \bigr), \mathbb{P}\bigl( R( \hat{P}_{\mathbb{S}}) - \mathbb{E} R (\hat{P}_{\mathbb{S}}) \geq t \bigr) \bigr\} &\leq \exp \Bigl( - \frac{2t^2}{V} \Bigr) \nonumber \\
    &= \exp \biggl( - \frac{2 t^2}{\sum_{S \in \mathbb{S}} n_S^{-1}} \biggr)
\end{align}
for all $t \geq 0$.  Moreover, by the triangle inequality and two applications of Cauchy--Schwarz,
\begin{align*}
\label{Eq:Expectation}
    \mathbb{E}\biggl\{ \sup_{f_{\mathbb{S}} \in \mathcal{G}_{\mathbb{S}}^+: f_\mathbb{S} \leq |\mathbb{S}|-1} \bigl| R(\hat{P}_{\mathbb{S}},f_{\mathbb{S}}) &- R(P_{\mathbb{S}},f_{\mathbb{S}}) \bigr| \biggr\} \\
    &= \frac{1}{|\mathbb{S}|} \mathbb{E}  \biggl\{ \sup_{f_{\mathbb{S}} \in \mathcal{G}_{\mathbb{S}}^+: f_\mathbb{S} \leq |\mathbb{S}|-1} \biggl| \sum_{S \in \mathbb{S}} \sum_{x_S \in \mathcal{X}_S} f_S(x_S) \bigl\{ \hat{P}_S(\{x_S\}) - P_S(\{x_S\}) \bigr\} \biggr| \biggr\} \nonumber \\
    & \leq\frac{1}{2} \sum_{S \in \mathbb{S}} \sum_{x_S \in \mathcal{X}_S} \mathbb{E}\bigl|\hat{P}_S(\{x_S\}) - P_S(\{x_S\})\bigr| \nonumber \\
    & \leq \frac{1}{2}\sum_{S \in \mathbb{S}} \frac{1}{n_S^{1/2}} \sum_{x_S \in \mathcal{X}_S} \bigl[ P_S(\{x_S\})\bigl\{1-P_S(\{x_S\})\bigr\} \bigr]^{1/2} \nonumber \\
    &\leq \frac{1}{2}\sum_{S \in \mathbb{S}} \Bigl( \frac{|\mathcal{X}_S| -1}{n_S} \Bigr)^{1/2}.
\end{align*}
Thus, 
\begin{equation}
\label{Eq:NullExpectation}
    \bigl|\mathbb{E} R(\hat{P}_{\mathbb{S}}) - R(P_{\mathbb{S}})\bigr| \leq \frac{1}{2} \sum_{S \in \mathbb{S}} \Bigl( \frac{|\mathcal{X}_S| -1}{n_S} \Bigr)^{1/2}. 
\end{equation}
It follows from~\eqref{Eq:NullExpectation} and~\eqref{Eq:DiscreteConcentration} that under $H_0'$, i.e.~when $R(P_\mathbb{S}) = 0$, we have
\[
    \mathbb{P}\bigl( R( \hat{P}_{\mathbb{S}}) \geq C_\alpha \bigr) \leq \mathbb{P}\biggl( R(\hat{P}_{\mathbb{S}}) - \mathbb{E} R(\hat{P}_{\mathbb{S}}) \geq \biggl\{ \frac{1}{2} \log(1/\alpha) \sum_{S \in \mathbb{S}} \frac{1}{n_S} \biggr\}^{1/2} \biggr) \leq \alpha.
\]
On the other hand, if $R(P_\mathbb{S}) \geq C_\alpha + C_\beta$, then from~\eqref{Eq:NullExpectation} and~\eqref{Eq:DiscreteConcentration} again,
\begin{align*}
    \mathbb{P}\bigl( R(\hat{P}_{\mathbb{S}}) \geq C_\alpha\bigr) &\geq \mathbb{P}\biggl( R(\hat{P}_{\mathbb{S}}) - R(P_{\mathbb{S}}) \geq -\frac{1}{2} \sum_{S \in \mathbb{S}} \Bigl( \frac{|\mathcal{X}_S| \!-\! 1}{n_S} \Bigr)^{1/2} - \biggl\{ \frac{1}{2} \log(1/\beta) \sum_{S \in \mathbb{S}} \frac{1}{n_S} \biggr\}^{1/2} \biggr) \\
    & \geq \mathbb{P}\biggl( R(\hat{P}_{\mathbb{S}}) - \mathbb{E} R(\hat{P}_{\mathbb{S}}) \geq - \biggl\{ \frac{1}{2} \log(1/\beta) \sum_{S \in \mathbb{S}} \frac{1}{n_S} \biggr\}^{1/2} \biggr) \geq 1 - \beta,
\end{align*}
as required.

\end{proof}

\begin{proof}[Proof of Proposition~\ref{Prop:L1Projection}] 
By the same argument given at the start of the proof of Proposition~\ref{Prop:DiscreteTest1}, in seeking a maximiser in~\eqref{Eq:RPS}, we may restrict our optimisation to $\bigl\{f_{\mathbb{S}} \in \mathcal{G}_{\mathbb{S}}^+: f_{\mathbb{S}} \leq |\mathbb{S}|-1\bigr\}$.  But $\bigl[-1,|\mathbb{S}|-1\bigr]^{d_{\mathbb{S}}}$ is a compact subset of $\mathbb{R}^{d_{\mathbb{S}}}$, and we may regard $f_{\mathbb{S}} \mapsto R(P_{\mathbb{S}},f_{\mathbb{S}})$ as a continuous function on this set, so the supremum in~\eqref{Eq:RPS} is attained.

By specialising Theorem~\ref{Thm:DualRepresentation} to the discrete case we see that
\begin{align*}
    R(P_\mathbb{S}) &= \sup\bigl\{ \epsilon \in [0,1] : p_\mathbb{S} = \epsilon q_\mathbb{S} + (1-\epsilon) r_\mathbb{S}, q_\mathbb{S} \in \mathcal{P}_\mathbb{S}^0, r_\mathbb{S} \in \mathcal{P}_\mathbb{S} \bigr\}.
\end{align*}
When $R(P_\mathbb{S})=0$ we can trivially attain the supremum by taking $q_\mathbb{S}=p_\mathbb{S} \in \mathcal{P}_\mathbb{S}^0$, since we already know that $R(P_\mathbb{S})=0$ if and only if $\mathcal{P}_\mathbb{S} \in \mathcal{P}_\mathbb{S}^0$. Supposing that $R(P_\mathbb{S})>0$, for each $m \geq 1/R(P_\mathbb{S})$ we can find $q_{\mathbb{S}}^{(m)} \in \mathcal{P}_\mathbb{S}^0$,  $r_\mathbb{S}^{(m)} \in \mathcal{P}_\mathbb{S}$, and $\epsilon^{(m)} \in [R(P_\mathbb{S}),R(P_\mathbb{S})-1/m]$ such that $p_\mathbb{S} = \epsilon^{(m)} q_\mathbb{S}^{(m)} + (1-\epsilon^{(m)}) r_\mathbb{S}^{(m)}$. There exists a subsequence $(m_k)_{k \in \mathbb{N}}$, $q_\mathbb{S} \in \mathcal{P}_\mathbb{S}^0$, and $r_\mathbb{S} \in \mathcal{P}_\mathbb{S}$ such that $q_\mathbb{S}^{(m_k)} \rightarrow q_\mathbb{S}$ and $r_\mathbb{S}^{(m_k)} \rightarrow r_\mathbb{S}$ as $k \rightarrow \infty$. We see that we must have
\[
    p_\mathbb{S} = R(P_\mathbb{S}) q_\mathbb{S} + \{1-R(P_\mathbb{S})\} r_\mathbb{S},
\]
so that the supremum in~\eqref{Thm:DualRepresentation} is indeed attained.

\medskip

We now turn to the final part of the result. From Theorem~\ref{Thm:DualRepresentation} we know that for any $\epsilon>0$ we have $R(P_\mathbb{S}) \leq \epsilon$ if and only if $P_\mathbb{S} \in (1-\epsilon) \mathcal{P}_\mathbb{S}^0 + \epsilon \mathcal{P}_\mathbb{S}$. Now suppose that $P_\mathbb{S} \in \mathcal{P}_{\mathbb{S}}^{\mathrm{cons}}$ satisfies $R(P_\mathbb{S}) \leq \epsilon$. Then there exist $Q_\mathbb{S}^0 \in \mathcal{P}_\mathbb{S}^0$ and $Q_\mathbb{S} \in \mathcal{P}_\mathbb{S}$ such that $P_\mathbb{S} = (1-\epsilon) Q_\mathbb{S}^0 + \epsilon Q_\mathbb{S}$. Since $\mathcal{P}_\mathbb{S}^0 \subseteq \mathcal{P}_\mathbb{S}^\mathrm{cons}$, it follows that if $S_1,S_2 \in \mathbb{S}$ have $S_1 \cap S_2 \neq \emptyset$, then
\[
Q_{S_1}^{S_1 \cap S_2} = \frac{1}{\epsilon}\bigl\{P_{S_1}^{S_1 \cap S_2} - (1-\epsilon)Q_{S_1}^{0,S_1 \cap S_2}\bigr\} = \frac{1}{\epsilon}\bigl\{P_{S_2}^{S_1 \cap S_2} - (1-\epsilon)Q_{S_2}^{0,S_1 \cap S_2}\bigr\} = Q_{S_2}^{S_1 \cap S_2};
\]
in other words, $Q_\mathbb{S} \in \mathcal{P}_\mathbb{S}^\mathrm{cons}$.  Thus, if $P_\mathbb{S} \in \mathcal{P}_{\mathbb{S}}^{\mathrm{cons}}$, then $R(P_\mathbb{S}) \leq \epsilon$ if and only if $
P_\mathbb{S} \in (1-\epsilon) \mathcal{P}_\mathbb{S}^0 + \epsilon \mathcal{P}_\mathbb{S}^\mathrm{cons}$, which holds if and only if 
\begin{equation}
\label{Eq:iff}
P_\mathbb{S} \in \mathcal{P}_\mathbb{S}^{0,*} + \epsilon \mathcal{P}_\mathbb{S}^\mathrm{cons,**} = \epsilon ( \mathcal{P}_\mathbb{S}^{0,*} + \mathcal{P}_\mathbb{S}^{\mathrm{cons},**}) =: \mathcal{P}_\mathbb{S}^{\epsilon,*}.
\end{equation}
Now $\mathcal{P}_\mathbb{S}^{1,*}$ is a convex polyhedral set, so there exist $B \in \mathbb{R}^{F \times \mathcal{X}_\mathbb{S}}$ and $b \in \mathbb{R}^F$ such that
\[
    \mathcal{P}_\mathbb{S}^{\epsilon,*} = \bigl\{ p_{\mathbb{S}} \equiv P_\mathbb{S} \in \mathcal{P}_\mathbb{S}^{\mathrm{cons},*} : B p_\mathbb{S} \geq -\epsilon b \bigl\},
\]
where the equivalence here indicates that $p_{\mathbb{S}}$ is the probability mass sequence corresponding to $P_{\mathbb{S}}$.  Since $0_{\mathbb{S}} \in \mathcal{P}_\mathbb{S}^{\epsilon,*}$, we must have $b \in [0,\infty)^F$ and, by rescaling the rows of $B$ if necessary, we may assume that $b \in \{0,1\}^F$.  We may therefore partition $B = \begin{pmatrix} B_1 \\ B_2 \end{pmatrix}$, where $B_1 \in \mathbb{R}^{(F-m) \times \mathcal{X}_\mathbb{S}}$ and $B_2 \in \mathbb{R}^{m \times \mathcal{X}_\mathbb{S}}$ are such that
\begin{equation}
\label{Eq:PSepsilon*}
     \mathcal{P}_\mathbb{S}^{\epsilon,*} = \bigl\{ p_{\mathbb{S}} \equiv P_\mathbb{S} \in \mathcal{P}_\mathbb{S}^{\mathrm{cons},*} : B_1 p_\mathbb{S} \geq -\epsilon , B_2 p_\mathbb{S} \geq 0 \bigl\}.
\end{equation}
In fact, however, we claim that $m=0$, so that $b = 1_F$.  To see this, note first that $(\mathcal{P}_\mathbb{S}^{\epsilon,*})_{\epsilon \geq 0}$ is an increasing family, by~\eqref{Eq:PSepsilon*}.  Moreover, if $\lambda \geq 0$ and $P_\mathbb{S} \in \mathcal{P}_\mathbb{S}^{\mathrm{cons}}$, then $\lambda \cdot P_{\mathbb{S}} \in \lambda \mathcal{P}_\mathbb{S}^{\mathrm{cons},**} \subseteq \lambda(\mathcal{P}_\mathbb{S}^{0,*} + \mathcal{P}_\mathbb{S}^{\mathrm{cons},**}) = \mathcal{P}_{\mathbb{S}}^{\lambda,*}$, and hence $\bigcup_{\epsilon \geq 0} \mathcal{P}_\mathbb{S}^{\epsilon,*} = \mathcal{P}_\mathbb{S}^{\mathrm{cons},*}$.  But
\[
\bigcup_{\epsilon \geq 0} \mathcal{P}_\mathbb{S}^{\epsilon,*} = \bigl\{ p_{\mathbb{S}} \equiv P_\mathbb{S} \in \mathcal{P}_\mathbb{S}^{\mathrm{cons},*} : B_2 p_\mathbb{S} \geq 0 \bigl\},
\]
and we conclude that $m=0$, as required.  Therefore, by~\eqref{Eq:iff}, when $p_{\mathbb{S}} \equiv P_\mathbb{S} \in \mathcal{P}_\mathbb{S}^\mathrm{cons}$, we have
\begin{equation}
\label{Eq:RPS2}
    R(P_\mathbb{S}) = \inf\{\epsilon > 0: P_\mathbb{S} \in \mathcal{P}_{\mathbb{S}}^{\epsilon,*}\} = \|Bp_\mathbb{S}\|_\infty.
\end{equation}

We now argue that $f_\mathbb{S}^{(1)},\ldots,f_\mathbb{S}^{(F)}$ can be taken to be scalar multiples of the rows of $B$.  We may regard $\mathcal{P}_\mathbb{S}^{\mathrm{cons},*}$ as a convex cone in $[0,\infty)^{\mathcal{X}_{\mathbb{S}}}$, this cone is not full-dimensional (due to the consistency constraints), but if instead we regard it as a subset of its affine hull, then we will be able to express it uniquely as an intersection of halfspaces.  To see this, note that the consistency constraints are linear, so there exist $d_0 \leq |\mathcal{X}_\mathbb{S}|$ and $U \in \mathbb{R}^{\mathcal{X}_\mathbb{S} \times d_0}$ of full column rank such that
\[
    \mathcal{P}_\mathbb{S}^{\mathrm{cons},*} = \{ Uy : Uy \geq 0, y \in \mathbb{R}^{d_0} \}.
\]
Writing $f_\mathbb{S}^{(1)},\ldots,f_\mathbb{S}^{(M)}$ for the extreme points of $\{f_{\mathbb{S}} \in \mathcal{G}_\mathbb{S}^+:f_{\mathbb{S}} \leq |\mathbb{S}| - 1\}$, we have
\begin{align*}
    \mathcal{Y}^{1,*} := \{y \in \mathbb{R}^{d_0} : Uy \in \mathcal{P}_\mathbb{S}^{1,*} \} &= \{ y \in \mathbb{R}^{d_0}: Uy \geq 0, BUy \geq -1 \} \\
    &= \Bigl\{ y \in \mathbb{R}^{d_0}: Uy \geq 0, \min_{\ell \in [M]} (f_\mathbb{S}^{(\ell)})^T Uy  \geq -|\mathbb{S}|\Bigr\}.
\end{align*}
Since $\mathcal{Y}^{1,*}$ is a  full-dimensional, convex subset of $\mathbb{R}^{d_0}$, the uniqueness of halfspace representations means that by relabelling if necessary, we may assume that each row of $BU$ is $(f_\mathbb{S}^{(\ell)})^T U / |\mathbb{S}|$ for some $\ell\in [F]$. Hence $\mathcal{Y}^{1,*}=\bigl\{y\in \mathbb{R}^{d_0}:Uy \geq 0, \min_{\ell \in [F]} (f_\mathbb{S}^{(\ell)})^T Uy  \geq -|\mathbb{S}|\bigr\}$, and
\[
    \mathcal{P}_\mathbb{S}^{1,*} = \Bigl\{ p_\mathbb{S} \in \mathcal{P}_\mathbb{S}^{\mathrm{cons},*} : \min_{\ell \in [F]} (f_\mathbb{S}^{(\ell)})^T p_\mathbb{S} \geq -|\mathbb{S}| \Bigr\}.
\]
It therefore follows from~\eqref{Eq:RPS2} that, when $P_\mathbb{S} \in \mathcal{P}_\mathbb{S}^\mathrm{cons}$, we have
\begin{equation}
\label{Eq:FacetMax}
    R(P_\mathbb{S}) = \max_{\ell \in [F]} R(P_\mathbb{S}, f_\mathbb{S}^{(\ell)})_+.
\end{equation}
Having characterised the incompatibility index for consistent distributions, we finally prove the given bounds on this index in the general case.  To see the lower bound, let $S_1,S_2 \in \mathbb{S}$ be such that $S_1 \cap S_2 \neq \emptyset$, and let $E \subseteq \mathcal{X}_{S_1 \cap S_2}$.  Define $f_\mathbb{S}^{S_1,S_2,E} = (f_S^{S_1,S_2,E}:S \in \mathbb{S}) \in \mathcal{G}_{\mathbb{S}}$ by
\[
    f_S^{S_1,S_2,E}(x_S) := \left\{\begin{array}{ll}
      1   & \text{ if } S=S_1, \, x_{S_1 \cap S_2} \in E  \\
        -1 & \text{ if } S=S_2, \, x_{S_1 \cap S_2} \in E \\
        0 & \text{ otherwise.}
    \end{array} \right.
\]
It is straightforward to check that $f_\mathbb{S}^{S_1,S_2,E} \in \mathcal{G}_\mathbb{S}^+$: if $x \in \mathcal{X}$ is such that $x_{S_1 \cap S_2} \in E$ then
\[
    \sum_{S \in \mathbb{S}} f_S^{S_1,S_2,E}(x_S) = f_{S_1}^{S_1,S_2,E}(x_{S_1}) + f_{S_2}^{S_1,S_2,E}(x_{S_2}) = 1 - 1 =0,
\]
and if $x$ is such that $x_{S_1 \cap S_2} \not\in E$ then $f_S^{S_1,S_2,E}(x_S)=0$ for all $S \in \mathbb{S}$. We also have that
\begin{align*}
     R(P_\mathbb{S}, f_\mathbb{S}^{S_1,S_2,E}) &= - \frac{1}{|\mathbb{S}|} \biggl\{ \sum_{x_{S_1} \in \mathcal{X}_{S_1} : x_{S_1 \cap S_2} \in E} P_{S_1}(\{x_{S_1}\}) - \sum_{x_{S_2} \in \mathcal{X}_{S_2} : x_{S_1 \cap S_2} \in E} P_{S_2}(\{x_{S_2}\}) \biggr\} \\
    &= \frac{1}{|\mathbb{S}|} \bigl\{ P_{S_2}^{S_1 \cap S_2}(E) - P_{S_1}^{S_1 \cap S_2}(E) \bigr\}.
\end{align*}
We conclude that
\begin{align*}
    R(P_\mathbb{S}) &\geq \max \biggl\{ \max_{\ell \in [F]} R(P_\mathbb{S}, f_\mathbb{S}^{(\ell)})_+, \max_{S_1,S_2 \in \mathbb{S} : S_1 \cap S_2 \neq \emptyset} \max_{E \subseteq \mathcal{X}_{S_1 \cap S_2}} R(P_\mathbb{S}, f_\mathbb{S}^{S_1,S_2,E}) \biggr\} \\
    &= \max\biggl\{ \max_{\ell\in [F]} R(P_\mathbb{S}, f_\mathbb{S}^{(\ell)})_+, \frac{1}{|\mathbb{S}|} \max_{S_1,S_2 \in \mathbb{S} : S_1 \cap S_2 \neq \emptyset} d_\mathrm{TV} \bigl( P_{S_1}^{S_1 \cap S_2}, P_{S_2}^{S_1 \cap S_2} \bigr) \biggr\}.
\end{align*}
This establishes the lower bound, and we now turn to the upper bound.  Given sequences of signed measures $P_\mathbb{S},Q_\mathbb{S} \in \{\lambda_1\mathcal{P}_\mathbb{S} - \lambda_2\mathcal{P}_\mathbb{S}:\lambda_1,\lambda_2\geq 0\}$, we define their total variation distance by
\[
d_{\mathrm{TV}}(P_\mathbb{S},Q_{\mathbb{S}}) := \sum_{S \in \mathbb{S}} \sup_{A_S \in \mathcal{A}_S} |P_S(A_S) - Q_S(A_S)|.
\]
Now, given any $P_\mathbb{S} \in \mathcal{P}_\mathbb{S}$ and $P_\mathbb{S}^{\mathrm{cons},*} \in \mathcal{P}_\mathbb{S}^{\mathrm{cons},*}$, 
we have by~\eqref{Eq:FacetMax} and the fact (quoted at the start of the proof) that all extreme points of $\mathcal{G}_\mathbb{S}^+$ take values in $[-1,|\mathbb{S}|-1]^{\mathcal{X}_\mathbb{S}}$ that
\begin{align}
\label{Eq:RDecomp}
    R(P_\mathbb{S}) &= \frac{1}{|\mathbb{S}|} \sup_{f_\mathbb{S} \in \mathcal{G}_\mathbb{S}^+}\{ - f_\mathbb{S}^T (p_\mathbb{S} - p_\mathbb{S}^{\mathrm{cons},*} + p_\mathbb{S}^{\mathrm{cons},*}) \} \nonumber \\
    & \leq \frac{1}{|\mathbb{S}|} \biggl[ \sup_{f_\mathbb{S} \in \mathcal{G}_\mathbb{S}^+}\{ - f_\mathbb{S}^T (p_\mathbb{S} - p_\mathbb{S}^{\mathrm{cons},*}) \}  + \sup_{f_\mathbb{S} \in \mathcal{G}_\mathbb{S}^+} (- f_\mathbb{S}^T p_\mathbb{S}^{\mathrm{cons},*}) \biggr] \nonumber \\
    & = \frac{1}{|\mathbb{S}|} \biggl[ \sup_{f_\mathbb{S} \in \mathcal{G}_\mathbb{S}^+}\{ - f_\mathbb{S}^T (p_\mathbb{S} - p_\mathbb{S}^{\mathrm{cons},*}) \} + \max_{\ell \in [F]} \{-  (f_\mathbb{S}^{(\ell)})^T (p_\mathbb{S}^{\mathrm{cons},*} - p_\mathbb{S} + p_\mathbb{S}) \} \biggr] \nonumber \\
    & \leq \frac{1}{|\mathbb{S}|} \biggl[ \sup_{f_\mathbb{S} \in \mathcal{G}_\mathbb{S}^+}\{ - f_\mathbb{S}^T (p_\mathbb{S} - p_\mathbb{S}^{\mathrm{cons},*}) \} + \sup_{f_\mathbb{S} \in \mathcal{G}_\mathbb{S}^+}\{ - f_\mathbb{S}^T ( p_\mathbb{S}^{\mathrm{cons},*} - p_\mathbb{S}) \} +\max_{\ell \in [F]} \{-  (f_\mathbb{S}^{(\ell)})^T p_\mathbb{S} \} \biggr] \nonumber\\
    & \leq 2d_{\mathrm{TV}}(P_\mathbb{S},P_\mathbb{S}^{\mathrm{cons},*}) + \frac{1}{|\mathbb{S}|} \max_{\ell \in [F]} \{-  (f_\mathbb{S}^{(\ell)})^T p_\mathbb{S} \}.
\end{align}
We proceed by constructing an element of $\mathcal{P}_\mathbb{S}^{\mathrm{cons},*}$ whose total variation distance to $P_\mathbb{S}$ can be controlled.  For $\omega \in \{0,1\}^\mathbb{S}$, write $T_\omega:= \cap_{S: \omega_S=1}S$ and $|\omega| :=\sum_{S \in \mathbb{S}} \omega_S$. Define $\tilde{p}_\mathbb{S} \in \mathbb{R}^{\mathcal{X}_\mathbb{S}}$ by
\[
    \tilde{p}_{S_0}(x_{S_0}) := p_{S_0}(x_{S_0}) + \sum_{\omega \in \{0,1\}^\mathbb{S}: \omega_{S_0}=1, T_\omega \neq \emptyset} \frac{\lambda_{|\omega|} |\mathcal{X}_{T_\omega}|}{ |\mathcal{X}_{S_0}|} \sum_{S: \omega_S=1}\{p_S^{T_\omega}(x_{T_\omega}) - p_{S_0}^{T_\omega}(x_{T_\omega})\}
\]
with $\lambda_{|\omega|} := \frac{(-1)^{|\omega|}}{|\omega|(|\omega|-1)} \mathbbm{1}_{\{|\omega| \geq 2\}}$. Although $\tilde{p}_\mathbb{S}$ may take negative values, we will see that it satisfies all the linear constraints of consistency.  To see this, let $S_1,S_2 \in \mathbb{S}$ be such that $S_1 \cap S_2 \neq \emptyset$ and $x_{S_1 \cap S_2} \in \mathcal{X}_{S_1 \cap S_2}$, and write $\Omega_{\mathbb{S}}^{a,b} := \{\omega \in \{0,1\}^{\mathbb{S}}:T_\omega \neq \emptyset, \omega_{S_1} = a, \omega_{S_2} = b\}$ for $a,b \in \{0,1\}$.  Observe that if $A \subseteq B \subseteq [d]$, then $|\mathcal{X}_B|/|\mathcal{X}_A| = |\mathcal{X}_{B \cap A^c}|$.  Thus, in particular, when $\omega \in \Omega_\mathbb{S}^{1,0}$ for instance, we have
\[
    \frac{|\mathcal{X}_{T_\omega}| | \mathcal{X}_{S_1 \cap S_2^c \cap T_\omega^c}| |\mathcal{X}_{S_1 \cap S_2}|}{|\mathcal{X}_{T_\omega \cap S_2}||\mathcal{X}_{S_1}|} = \frac{|\mathcal{X}_{T_\omega \cap S_2^c}||\mathcal{X}_{S_1 \cap S_2^c \cap T_\omega^c}|}{|\mathcal{X}_{S_1 \cap S_2^c}|} = \frac{|\mathcal{X}_{T_\omega \cap S_2^c}|}{|\mathcal{X}_{S_1 \cap S_2^c \cap T_\omega}|} = 1.
\]
Hence
\begin{align*}
    &\tilde{p}_{S_1}^{S_1 \cap S_2}(x_{S_1 \cap S_2}) - \tilde{p}_{S_2}^{S_1 \cap S_2}(x_{S_1 \cap S_2}) = \sum_{\substack{x_{S_1} \in \mathcal{X}_{S_1} : \\ (x_{S_1})_{S_1 \cap S_2} = x_{S_1 \cap S_2}}} \tilde{p}_{S_1}(x_{S_1}) - \sum_{\substack{x_{S_2} \in \mathcal{X}_{S_2} : \\ (x_{S_2})_{S_1 \cap S_2} = x_{S_1 \cap S_2}}} \tilde{p}_{S_2}(x_{S_2}) \\
    & = p_{S_1}^{S_1 \cap S_2}(x_{S_1 \cap S_2}) - p_{S_2}^{S_1 \cap S_2}(x_{S_1 \cap S_2}) + \sum_{\omega \in \Omega_{\mathbb{S}}^{1,1}} \lambda_{|\omega|} \sum_{S : \omega_S=1} \biggl[ \frac{|\mathcal{X}_{T_\omega}|}{|\mathcal{X}_{S_1}|} |\mathcal{X}_{S_1\cap S_2^c}| \{p_S^{T_\omega}(x_{T_\omega}) - p_{S_1}^{T_\omega}(x_{T_\omega})\} \\
    &  \hspace{250pt} - \frac{|\mathcal{X}_{T_\omega}|}{|\mathcal{X}_{S_2}|} |\mathcal{X}_{S_1^c\cap S_2}| \{p_S^{T_\omega}(x_{T_\omega}) - p_{S_2}^{T_\omega}(x_{T_\omega})\} \biggr] \\
    & \hspace{100pt}+ \sum_{\omega \in \Omega_{\mathbb{S}}^{1,0}} \lambda_{|\omega|} \sum_{S : \omega_S=1} \frac{|\mathcal{X}_{T_\omega}|}{|\mathcal{X}_{S_1}|} |\mathcal{X}_{S_1 \cap S_2^c \cap T_\omega^c}| \{p_S^{T_\omega \cap S_2}(x_{T_\omega \cap S_2}) - p_{S_1}^{T_\omega \cap S_2}(x_{T_\omega \cap S_2}) \} \\
    & \hspace{100pt}- \sum_{\omega \in \Omega_{\mathbb{S}}^{0,1}} \lambda_{|\omega|} \sum_{S : \omega_S=1} \frac{|\mathcal{X}_{T_\omega}|}{|\mathcal{X}_{S_2}|} |\mathcal{X}_{S_1^c \cap S_2 \cap T_\omega^c}| \{p_S^{T_\omega \cap S_1}(x_{T_\omega \cap S_1}) - p_{S_2}^{T_\omega \cap S_1}(x_{T_\omega \cap S_1}) \} \\
    & = p_{S_1}^{S_1 \cap S_2}(x_{S_1 \cap S_2}) - p_{S_2}^{S_1 \cap S_2}(x_{S_1 \cap S_2}) - \sum_{\omega \in \Omega_{\mathbb{S}}^{1,1}} |\omega| \lambda_{|\omega|}   \frac{|\mathcal{X}_{T_{\omega}}|}{|\mathcal{X}_{S_1 \cap S_2}|} \{ p_{S_1}^{T_{\omega}}(x_{T_{\omega}}) - p_{S_2}^{T_{\omega}}(x_{T_{\omega}}) \} \\
    & \hspace{100pt}+ \sum_{\omega' \in\Omega_{\mathbb{S}}^{1,1}} \lambda_{|\omega'|-1} \frac{|\mathcal{X}_{T_{\omega'}}|}{|\mathcal{X}_{S_1 \cap S_2}|} \sum_{S: \omega'_S=1}(1-\mathbbm{1}_{\{S=S_2\}} ) \{p_{S}^{T_{\omega'}}(x_{T_{\omega'}}) - p_{S_1}^{T_{\omega'}}(x_{T_{\omega'}}) \}\\
    & \hspace{100pt}- \sum_{\omega'\in \Omega_{\mathbb{S}}^{1,1}} \lambda_{|\omega'|-1} \frac{|\mathcal{X}_{T_{\omega'}}|}{|\mathcal{X}_{S_1 \cap S_2}|} \sum_{S: \omega'_S=1}(1-\mathbbm{1}_{\{S=S_1\}} ) \{p_{S}^{T_{\omega'}}(x_{T_{\omega'}}) - p_{S_2}^{T_{\omega'}}(x_{T_{\omega'}}) \} \\
    & = p_{S_1}^{S_1 \cap S_2}(x_{S_1 \cap S_2}) - p_{S_2}^{S_1 \cap S_2}(x_{S_1 \cap S_2}) - \sum_{\omega \in \Omega_{\mathbb{S}}^{1,1}} |\omega| \lambda_{|\omega|}   \frac{|\mathcal{X}_{T_\omega}|}{|\mathcal{X}_{S_1 \cap S_2}|} \{ p_{S_1}^{T_\omega}(x_{T_\omega}) - p_{S_2}^{T_\omega}(x_{T_\omega}) \} \\
    & \hspace{20pt}- \sum_{\omega \in \Omega_{\mathbb{S}}^{1,1}} \lambda_{|\omega|-1} (|\omega|-2) \frac{|\mathcal{X}_{T_{\omega}}|}{|\mathcal{X}_{S_1 \cap S_2}|} \{p_{S_1}^{T_{\omega}}(x_{T_{\omega}}) - p_{S_2}^{T_{\omega}}(x_{T_{\omega}}) \}=0,
\end{align*}
where the final equality holds because $(\lambda_r)$ satisfies $\lambda_2=1/2$ and $r\lambda_r=-(r-2)\lambda_{r-1}$ for $r \geq 3$. The total negative mass of $\tilde{p}_\mathbb{S}$ satisfies
\begin{align}
\label{Eq:NegMassBound}
    \sum_{S_0 \in \mathbb{S}} &\sum_{x_{S_0} \in \mathcal{X}_{S_0}} \tilde{p}_{S_0}(x_{S_0})_- \leq d_\mathrm{TV}(P_\mathbb{S},\tilde{P}_\mathbb{S}) \nonumber \\
    &\leq \sum_{S_0 \in \mathbb{S}}  \sum_{\substack{\omega:|\omega| \geq 2,\\ \omega_{S_0}=1, T_\omega \neq \emptyset}} \frac{ |\mathcal{X}_{T_\omega}|}{|\mathcal{X}_{S_0}||\omega|(|\omega|-1)} \sum_{S : \omega_S=1} \sum_{x_{S_0} \in \mathcal{X}_{S_0}} \bigl[(-1)^{|\omega|} \{p_S^{T_\omega}(x_{T_\omega}) - p_{S_0}^{T_\omega}(x_{T_\omega}) \}\bigr]_-  \nonumber \\
    &\leq \sum_{S_0 \in \mathbb{S}} \sum_{\substack{\omega : |\omega|\geq 2, \\ \omega_{S_0}=1,T_\omega \neq \emptyset}} \frac{1}{|\omega|-1} \max_{S : \omega_S=1} d_\mathrm{TV}( p_S^{T_\omega}, p_{S_0}^{T_\omega}) \nonumber \\
    & \leq \biggl( \sum_{\omega: |\omega| \geq 2, T_\omega \neq \emptyset} \frac{|\omega|}{|\omega|-1} \biggr) \max_{S,S_0 \in \mathbb{S}: S \cap S_0 \neq \emptyset} d_\mathrm{TV}(p_{S}^{S \cap S_0}, p_{S_0}^{S \cap S_0}) \nonumber \\
    &\leq 2^{|\mathbb{S}|+1} \max_{S,S_0 \in \mathbb{S}: S \cap S_0 \neq \emptyset} d_\mathrm{TV}(p_{S}^{S \cap S_0}, p_{S_0}^{S \cap S_0}).
\end{align}
Now define $\check{P}_{\mathbb{S}} \in \{\lambda \cdot \mathcal{P}_\mathbb{S} : \lambda \geq 0 \}$ with mass function $\check{p}_{\mathbb{S}}$ given by
\[
\check{p}_\mathbb{S} := \tilde{p}_\mathbb{S} +\mathbb{A}\biggl(\sum_{x \in \mathcal{X}} \delta_x \sum_{S \in \mathbb{S}} \frac{\tilde{p}_S(x_S)_-}{|\mathcal{X}_{S^c}|}  \biggr) 
\]
where $\delta_y \in \{0,1\}^\mathcal{X}$ denotes a Dirac point mass on $y \in \mathcal{X}$. We see that this is non-negative by writing
\[
\check{p}_S(x_S)=\tilde{p}_S(x_S) + \sum_{y:y_S=x_S} \sum_{T \in \mathbb{S}} \frac{\tilde{p}_T(y_T)_-}{|\mathcal{X}_{T^c}|} \geq \tilde{p}_S(x_S) + \tilde{p}_S(x_S)_- \geq 0.
\]
Since $\tilde{p}_\mathbb{S}$ satisfies the consistency constraints and $\check{p}_\mathbb{S}$ is formed by adding a compatible sequence of marginal measures to it, we have $\check{P}_{\mathbb{S}} \in \mathcal{P}_\mathbb{S}^{\mathrm{cons},*}$.  Moreover, $\check{p}_\mathbb{S} \geq \tilde{p}_\mathbb{S}$ and
\begin{align*}
\sum_{S \in \mathbb{S}}\sum_{x_S \in \mathcal{X}_S} \bigl\{\check{p}_S(x_S) - \tilde{p}_S(x_S)\bigr\} &= 1_{\mathcal{X}_\mathbb{S}}^T \mathbb{A} \biggl(\sum_{x \in \mathcal{X}} \delta_x \sum_{S \in \mathbb{S}} \frac{\tilde{p}_S(x_S)_-}{|\mathcal{X}_{S^c}|} \biggr) \\
    & = |\mathbb{S}| 1_\mathcal{X}^T \biggl(\sum_{x \in \mathcal{X}} \delta_x \sum_{S \in \mathbb{S}} \frac{\tilde{p}_S(x_S)_-}{|\mathcal{X}_{S^c}|} \biggr) \\
    & = |\mathbb{S}| \sum_{x \in \mathcal{X}} \sum_{S \in \mathbb{S}} \frac{\tilde{p}_S(x_S)_-}{|\mathcal{X}_{S^c}|} \\ 
    & = |\mathbb{S}| \sum_{S \in \mathbb{S}} \sum_{x_S \in \mathcal{X}_S} \tilde{p}_S(x_S)_- \\
    & \leq |\mathbb{S}| 2^{|\mathbb{S}|+1} \max_{S_1,S_2 \in \mathbb{S}: S_1 \cap S_2 \neq \emptyset} d_\mathrm{TV}(p_{S_1}^{S_1 \cap S_2}, p_{S_2}^{S_1 \cap S_2}).
\end{align*}
From this and~\eqref{Eq:NegMassBound}, we conclude that
\[
    d_{\mathrm{TV}}(P_\mathbb{S}, \mathcal{P}^{\mathrm{cons},*}) \leq |\mathbb{S}| 2^{|\mathbb{S}|+2} \max_{S_1,S_2 \in \mathbb{S}: S_1 \cap S_2 \neq \emptyset} d_\mathrm{TV}(p_{S_1}^{S_1 \cap S_2}, p_{S_2}^{S_1 \cap S_2}),
\]
and the result follows.
\end{proof}

\begin{proof}[Proof of Theorem~\ref{Prop:DiscreteTest}]
 We prove the result when $F' \geq 1$, and note that if $F'=0$ then simpler arguments apply. By Proposition~\ref{Prop:L1Projection} and the discussion after~\eqref{Eq:Equivalence}, we have
\begin{align}
\label{Eq:SizeDecomp}
&\mathbb{P}\bigl( R(\hat{P}_\mathbb{S}) \geq C_\alpha'\bigr) \nonumber \\ 
&\leq \mathbb{P}\biggl(D_R  \max_{\ell \in [F']}  R(\hat{P}_\mathbb{S}, f_\mathbb{S}^{(\ell),'} ) \geq \frac{C_\alpha'}{2} \biggr) + \mathbb{P} \biggl( \max_{S_1,S_2 \in \mathbb{S}} d_\mathrm{TV}( \hat{P}_{S_1}^{S_1 \cap S_2}, \hat{P}_{S_2}^{S_1 \cap S_2} ) \geq \frac{C_\alpha'}{|\mathbb{S}| 2^{|\mathbb{S}|+3}} \biggr) \nonumber \\
    & \leq F' \max_{\ell \in [F']} \mathbb{P}\biggl( R(\hat{P}_\mathbb{S}, f_\mathbb{S}^{(\ell),'} ) \geq \frac{C_\alpha'}{2D_R } \biggr) + \frac{|\mathbb{S}|(|\mathbb{S}| \!-\! 1)}{2} \! \max_{S_1,S_2 \in \mathbb{S}} \! \mathbb{P} \biggl( d_\mathrm{TV}( \hat{P}_{S_1}^{S_1 \cap S_2}, \hat{P}_{S_2}^{S_1 \cap S_2} ) \geq \frac{C_\alpha'}{|\mathbb{S}| 2^{|\mathbb{S}|+3}} \biggr).
\end{align}
Observe that when $P_\mathbb{S} \in \mathcal{P}_\mathbb{S}^0$, we have for any $f_\mathbb{S} \in \mathcal{G}_\mathbb{S}^+$ that $\mathbb{E} R(\hat{P}_\mathbb{S}, f_\mathbb{S}) = R(P_\mathbb{S}, f_\mathbb{S}) \leq 0$.  By~\eqref{Eq:EmpiricalProcess} and Hoeffding's inequality, whenever $P_\mathbb{S} \in \mathcal{P}_\mathbb{S}^0$, we have for any $\ell \in [F']$ that
\begin{align*}
    \mathbb{P}\biggl( R(\hat{P}_\mathbb{S}, f_\mathbb{S}^{(\ell),'} ) \geq \frac{C_\alpha'}{2D_R} \biggr) &\leq \mathbb{P}\biggl( R(\hat{P}_\mathbb{S}, f_\mathbb{S}^{(\ell),'} ) - \mathbb{E} R(\hat{P}_\mathbb{S},f_\mathbb{S}^{(\ell),'}) \geq \frac{C_\alpha'}{2D_R} \biggr) \\
    &\leq |\mathbb{S}| \max_{S \in \mathbb{S}} \mathbb{P} \biggl( -\frac{1}{n_S} \sum_{i=1}^{n_S} \bigl\{ f_S^{(\ell),'}(X_{S,i}) - \mathbb{E} f_S^{(\ell),'}(X_{S,i}) \bigr\} \geq \frac{C_\alpha'}{2D_R} \biggr) \\
    & \leq |\mathbb{S}| \max_{S \in \mathbb{S}} \exp \biggl( - \frac{n_S (C_\alpha'/D_R)^2}{2 |\mathbb{S}|^2} \biggr) \leq \frac{\alpha}{2F'}.
\end{align*}
For the second term in~\eqref{Eq:SizeDecomp}, under $H_0'$, for any $S_1,S_2 \in \mathbb{S}$ with $S_1 \neq S_2$ and $S_1 \cap S_2 \neq \emptyset$, we have
\begin{align*}
    &\mathbb{P} \biggl( d_\mathrm{TV}( \hat{P}_{S_1}^{S_1 \cap S_2}, \hat{P}_{S_2}^{S_1 \cap S_2} ) \geq \frac{C_\alpha'}{|\mathbb{S}| 2^{|\mathbb{S}|+3}} \biggr) \\
    &= \mathbb{P} \biggl( \max_{A \subseteq \mathcal{X}_{S_1 \cap S_2}} \bigl| \hat{P}_{S_1}^{S_1 \cap S_2}(A) - P_{S_1}^{S_1 \cap S_2}(A) + P_{S_2}^{S_1 \cap S_2}(A) - \hat{P}_{S_2}^{S_1 \cap S_2}(A) \bigr| \geq \frac{C_\alpha'}{|\mathbb{S}| 2^{|\mathbb{S}|+3}} \biggr) \\
    & \leq 2^{|\mathcal{X}_{S_1 \cap S_2}|} \max_{A \subseteq \mathcal{X}_{S_1 \cap S_2}} \max_{k \in \{1,2\}} \mathbb{P} \biggl( \bigl| \hat{P}_{S_k}^{S_1 \cap S_2}(A) - P_{S_k}^{S_1 \cap S_2}(A)\bigr| \geq \frac{C_\alpha'}{|\mathbb{S}| 2^{|\mathbb{S}|+4}} \biggr) \\
    & \leq 2^{|\mathcal{X}_{S_1 \cap S_2}|+1} \exp \biggl( - \frac{(n_{S_1} \wedge n_{S_2}) (C_\alpha')^2 }{|\mathbb{S}|^2 2^{2|\mathbb{S}| + 7}} \biggr) \leq \frac{\alpha}{|\mathbb{S}|(|\mathbb{S}|-1)},
\end{align*}
where we have used the fact that $\bigl| \hat{P}_{S_k}^{S_1 \cap S_2}(A) - P_{S_k}^{S_1 \cap S_2}(A)\bigr| = \bigl| \hat{P}_{S_k}^{S_1 \cap S_2}(A^c) - P_{S_k}^{S_1 \cap S_2}(A^c)\bigr|$, and where the penultimate bound follows from Hoeffding's inequality. We have now established that $\mathbb{P}\bigl( R(\hat{P}_\mathbb{S}) \geq C_\alpha'\bigr) \leq \alpha$ whenever $P_\mathbb{S} \in \mathcal{P}_\mathbb{S}^0$. 

We now turn to the final part of Proposition~\ref{Prop:DiscreteTest}. Very similar arguments to those above based on Hoeffding's inequality show that
\[
    \mathbb{P} \biggl( \max_{\ell \in [F']} R(\hat{P}_\mathbb{S}, f_\mathbb{S}^{(\ell),'} ) < C_\alpha' \biggr) \leq \beta
\]
whenever
\[
    \max_{\ell \in [F']} R(P_\mathbb{S}, f_\mathbb{S}^{(\ell),'}) \geq C_\alpha' + |\mathbb{S}| \biggl\{ \frac{2 \log(F'|\mathbb{S}|/\beta)}{\min_{S \in \mathbb{S}} n_S} \biggr\}^{1/2}.
\]
Likewise, for any $S_1,S_2 \in \mathbb{S}$ with $S_1 \cap S_2 \neq \emptyset$,
\[
    \mathbb{P} \Bigl( d_\mathrm{TV}( \hat{P}_{S_1}^{S_1 \cap S_2}, \hat{P}_{S_2}^{S_1 \cap S_2} ) < |\mathbb{S}| C_\alpha' \Bigr) \leq \beta
\]
whenever
\[
    d_\mathrm{TV}( P_{S_1}^{S_1 \cap S_2}, P_{S_2}^{S_1 \cap S_2} ) \geq |\mathbb{S}| C_\alpha' +  \biggl\{ \frac{2}{n_{S_1} \wedge n_{S_2}} \log \biggl( \frac{2^{|\mathcal{X}_{S_1 \cap S_2}|+1}}{\beta} \biggr) \biggr\}^{1/2}.
\]
Now, by Proposition~\ref{Prop:L1Projection}, if $R(P_\mathbb{S}) \geq M(C_\alpha'+C_\beta')$ then we must either have
\[
    \max_{\ell \in [F']} R(P_\mathbb{S}, f_\mathbb{S}^{(\ell),'}) \geq \frac{M}{2D_R}(C_\alpha'+C_\beta') \quad \text{or} \quad \max_{S_1,S_2 \in \mathbb{S}} d_\mathrm{TV} \bigl( P_{S_1}^{S_1 \cap S_2}, P_{S_2}^{S_1 \cap S_2} \bigr) \geq \frac{M}{2^{|\mathbb{S}|+3}|\mathbb{S}|}(C_\alpha'+C_\beta').
\]
Since
\[
    C_\beta' \asymp_{|\mathbb{S}|} |\mathbb{S}| D_R \biggl\{ \frac{2 \log(F'|\mathbb{S}|/\beta)}{\min_{S \in \mathbb{S}} n_S} \biggr\}^{1/2} + \max_{\substack{S_1,S_2 \in \mathbb{S}:\\S_1 \cap S_2 \neq \emptyset}} \biggl\{ \frac{2}{n_{S_1} \wedge n_{S_2}} \log \biggl( \frac{2^{|\mathcal{X}_{S_1 \cap S_2}|+1}}{\beta} \biggr) \biggr\}^{1/2},
\]
the result follows.
\end{proof}

\begin{proof}[Proof of Theorem~\ref{Prop:rs2example}]
We establish the equality~\eqref{Eq:RPSd3} by providing matching upper and lower bounds, first providing the required lower bound on $R(P_\mathbb{S})$.  Given $A \subseteq [r]$ and $B \subseteq [s]$, we construct $f_\mathbb{S} \in \mathcal{G}_\mathbb{S}$ as follows.  Writing, for example, $f_{ij \bullet} := f_{\{1,2\}}(i,j)$, define 
\[
    f_{ij \bullet} := \left\{ \begin{array}{cc}
      2   & \text{ if } (i,j) \in A \times B  \\
      -1   & \text{ if } (i,j) \in A \times B^c \\
      -1   & \text{ if } (i,j) \in A^c \times B \\
      2   & \text{ if } (i,j) \in A^c \times B^c
    \end{array} \right., \quad (f_{i \bullet 1}, f_{i \bullet 2}) := \left\{ \begin{array}{cc}
     (-1, 2)    & \text{ if } i \in A  \\
     (2, -1)    & \text{ if } i \in A^c
    \end{array} \right.,
\]
and
\[
    (f_{\bullet j1},f_{\bullet j2}):= \left\{ \begin{array}{cc}
     (-1, 2)    & \text{ if } j \in B  \\
     (2, -1)    & \text{ if } j \in B^c
    \end{array} \right..
\]
It is straightforward to check that 
$f_\mathbb{S} \in \mathcal{G}_\mathbb{S}^+$ because, for instance, if $i \in A$ and $j \in B$, then
\[
\min(f_{ij\bullet} + f_{\bullet j1} + f_{i \bullet 1},f_{ij\bullet} + f_{\bullet j2} + f_{i \bullet 2})  = \min(2 -1 -1, 2+2+2) =0.
\]
Hence
\begin{align}
\label{Eq:rs2lowerbound}
    &3R(P_\mathbb{S}) \geq 3R(P_\mathbb{S}, f_\mathbb{S}) =- \sum_{i=1}^r \sum_{j=1}^s p_{ij \bullet} f_{ij \bullet} -\sum_{i=1}^r(p_{i\bullet 1} f_{i\bullet 1} + p_{i\bullet 2} f_{i\bullet 2}) - \sum_{j=1}^s(p_{\bullet j1} f_{\bullet j1} + p_{\bullet j2} f_{\bullet j2}) \nonumber \\
    &= -2(p_{AB\bullet} + p_{A^c B^c \bullet}) + (p_{A^cB\bullet} + p_{AB^c\bullet}) - 2(p_{A^c \bullet 1} + p_{A \bullet 2}) \nonumber \\
    & \hspace{100pt} + (p_{A \bullet 1} + p_{A^c \bullet 2}) - 2(p_{\bullet B^c 1} + p_{\bullet B 2}) + (p_{\bullet B1} + p_{\bullet B^c 2}) \nonumber \\
    & = -2( 2 p_{AB \bullet} + p_{\bullet \bullet \bullet} - p_{A \bullet \bullet} - p_{\bullet B \bullet}) + (p_{\bullet B \bullet} + p_{A \bullet \bullet} - 2p_{A B \bullet}) - 2(p_{\bullet \bullet 1}+p_{A\bullet \bullet} - 2p_{A \bullet 1}) \nonumber \\
    & \hspace{30pt} +(2p_{A \bullet 1} + p_{\bullet \bullet \bullet} -p_{A \bullet \bullet} - p_{\bullet \bullet 1}) - 2(p_{\bullet \bullet 1} + p_{\bullet B \bullet} - 2p_{\bullet B 1}) + (2p_{\bullet B 1} + p_{\bullet \bullet \bullet} - p_{\bullet B \bullet} - p_{\bullet \bullet 1}) \nonumber \\
    & = -6(p_{AB \bullet} + p_{\bullet \bullet 1} - p_{A \bullet 1} - p_{\bullet B 1}).
\end{align}
Since $A \subseteq [r], B \subseteq [s]$ were arbitrary, and since $f_\mathbb{S} \equiv 0 \in \mathcal{G}_\mathbb{S}^+$, the desired lower bound follows.

We now give the matching upper bound on $R(P_\mathbb{S})$. When $P_\mathbb{S} \in \mathcal{P}_{\mathbb{S}}^0$ we automatically have $R(P_\mathbb{S})=0$.  On the other hand, when $P_\mathbb{S} \notin \mathcal{P}_{\mathbb{S}}^0$, we relate $R(P_\mathbb{S})$ to the maximum two-commodity flow through the network shown in Figure~\ref{fig:rs2network}.  Recalling the matrix $\mathbb{A} = (\mathbb{A}_{(S,y_S),x})_{(S,y_S) \in \mathcal{X}_\mathbb{S},x \in \mathcal{X}} \in \{0,1\}^{\mathcal{X}_{\mathbb{S}} \times \mathcal{X}}$ from~\eqref{Eq:A}, for any $P_{\mathbb{S}} = (P_S:S \in \mathbb{S})$ with corresponding probability mass sequence  $p_{\mathbb{S}} = (p_{(S,y_S)}) \in [0,1]^{\mathcal{X}_{\mathbb{S}}}$, we may write
\begin{align}
\label{Eq:RPSEquality}
    R(P_\mathbb{S}) &= - \frac{1}{|\mathbb{S}|} \min \bigl\{ p_\mathbb{S}^T f_\mathbb{S}:f_\mathbb{S} \geq -1, \mathbb{A}^T f_\mathbb{S} \geq 0 \bigr\} \nonumber \\
    & = 1 - \frac{1}{|\mathbb{S}|} \min \bigl\{ p_\mathbb{S}^T y:y \in [0,\infty)^{\mathcal{X}_{\mathbb{S}}}, \mathbb{A}^T y \geq |\mathbb{S}| \cdot 1_{\mathcal{X}} \bigr\} \nonumber \\
    &= 1 - \min \bigl\{ p_\mathbb{S}^T z:z \in [0,\infty)^{\mathcal{X}_{\mathbb{S}}}, \mathbb{A}^T z \geq 1_{\mathcal{X}} \bigr\} \nonumber \\
    & = 1 - \max \bigl\{1_{\mathcal{X}}^T p: p \in [0,\infty)^{\mathcal{X}}, \mathbb{A}p \leq p_\mathbb{S} \bigr\}.
\end{align}
Here, the final equality follows from the strong duality theorem for linear programming \citep[e.g.][p.~83]{matousek2007understanding}, where we note that both the primal and dual problems have feasible solutions. 
It follows from this that
\begin{align}
\label{Eq:1minusdual}
    1 - R(P_\mathbb{S}) &= \max\bigl\{ 1_\mathcal{X}^T p : p \in [0,\infty)^\mathcal{X}, \mathbb{A}p \leq p_\mathbb{S} \bigr\}, \nonumber \\
    &= \max \biggl\{ \sum_{i=1}^r \sum_{j=1}^s (q_{ij1}+q_{ij2}) : \min_{i,j,k} q_{ijk} \geq 0, \, \max_{i,j} (q_{ij1}+q_{ij2} -p_{ij\bullet} ) \leq 0,  \nonumber \\
    &\hspace{3.5cm} \max_{i,k} \biggl( \sum_{j=1}^s q_{ijk} - p_{i \bullet k} \biggr) \leq 0, \max_{j,k} \biggl( \sum_{i=1}^r q_{ijk} - p_{\bullet jk} \biggr) \leq 0 \biggr\}.
\end{align}
Figure~\ref{fig:rs2network} represents a flow network where, for $k \in \{1,2\}$, commodity $k$ is transferred from source $s_k$ to sink $t_k$.  We think of $q_{ijk}$ as the flow of commodity $k$ from node $x_{ik}$ to node $y_{ij}^{(1)}$, and $\sum_{i=1}^r \sum_{j=1}^s q_{ijk}$ as being the total flow of commodity $k$ from source~$s_k$ to sink $t_k$.  Of this flow, at most $p_{i \bullet k}$ may go through $x_{ik}$, for each $i \in [r]$, corresponding to the constraint $\sum_{j=1}^s q_{ijk} \leq p_{i \bullet k}$. For each $i \in [r], j \in [s]$, the combined flow of both commodities from $x_{i1}$ and $x_{i2}$ through to $y_{ij}^{(2)}$ is bounded above by $p_{ij \bullet}$, corresponding to the constraint $q_{ij1} + q_{ij2} \leq p_{ij\bullet}$. For each $j \in [s]$ and $k\in\{1,2\}$, the subsequent flow of commodity $k$ through node $z_{jk}$ to $t_k$ is bounded by $p_{\bullet j k}$, corresponding to the constraint $\sum_{i=1}^r q_{ijk} \leq p_{\bullet jk}$.
 \begin{figure}
    \begin{tikzpicture}
    \centering
  \node (v1) at (0,-2) {};
    \node (v2) at (0,-3) {};
    \node (v3) at (3,0) {};
    \node (v4) at (3,-1) {};
    \node (v5) at (3,-4) {};
    \node (v6) at (3,-5) {};
    \node (v7) at (6,0) {};
    \node (v8) at (6,-4) {};
    \node (v9) at (8,0) {};
    \node (v10) at (8,-4) {};
    \node (v11) at (12,0) {};
    \node (v12) at (12,-4) {};
    \node (v13) at (14,0) {};
    \node (v14) at (14,-4) {};
    \node (v15) at (7.5,-6) {};
    \node (v16) at (8.5,-6) {};
    \node (v17) at (13.5,-6) {};
    \node (v18) at (14.5,-6) {};
    \node (v19) at (10.25,-8) {};
    \node (v20) at (11.75,-8) {};
    \fill (v1) circle (0.1) node [left] {$s_1$};
    \fill (v2) circle (0.1) node [left] {$s_2$};
    \fill (v3) circle (0.1) node [above] {$x_{11}$};
    \fill (v4) circle (0.1) node [above] {$x_{12}$};
    \fill (v5) circle (0.1) node [above] {$x_{r1}$};
    \fill (v6) circle (0.1) node [above] {$x_{r2}$};
    \fill (v7) circle (0.1) node [above] {$y_{11}^{(1)}$};
    \fill (v8) circle (0.1) node [above] {$y_{r1}^{(1)}$};
    \fill (v9) circle (0.1) node [above] {$y_{11}^{(2)}$};
    \fill (v10) circle (0.1) node [above] {$y_{r1}^{(2)}$};
    \fill (v11) circle (0.1) node [above] {$y_{1s}^{(1)}$};
    \fill (v12) circle (0.1) node [above] {$y_{rs}^{(1)}$};
    \fill (v13) circle (0.1) node [above] {$y_{1s}^{(2)}$};
    \fill (v14) circle (0.1) node [above] {$y_{rs}^{(2)}$};
    \fill (v15) circle (0.1) node [left] {$z_{11}$};
    \fill (v16) circle (0.1) node [right] {$z_{12}$};
    \fill (v17) circle (0.1) node [left] {$z_{s1}$};
    \fill (v18) circle (0.1) node [right] {$z_{s2}$};
    \fill (v19) circle (0.1) node [below] {$t_1$};
    \fill (v20) circle (0.1) node [below] {$t_2$};
    \draw [-{Latex[scale=1.5]}](0.3,-1.8) -- (2.7,-0.2);
    \draw [-{Latex[scale=1.5]}](0.3,-2.2) -- (2.7,-3.8);
    \draw [-{Latex[scale=1.5]}](0.3,-2.8) -- (2.7,-1.2);
    \draw [-{Latex[scale=1.5]}](0.3,-3.2) -- (2.7,-4.8);
    \draw[loosely dashed](3,-1.5) -- (3,-3.4);
    \draw [-{Latex[scale=1.5]}](3.3,0) -- (5.7,0);
    \draw [-{Latex[scale=1.5]}](3.3,-0.9) -- (5.7,-0.1);
    \draw [-{Latex[scale=1.5]}](3.3,-4) -- (5.7,-4);
    \draw [-{Latex[scale=1.5]}](3.3,-4.9) -- (5.7,-4.1);
    \draw[loosely dashed](6,-0.5) -- (6,-3.2);
    \draw [-{Latex[scale=1.5]}](6.3,0) -- (7.7,0);
    \draw [-{Latex[scale=1.5]}](6.3,-4) -- (7.7,-4);
    \draw[loosely dashed](8,-0.5) -- (8,-3.2);
    \draw[loosely dashed](8.4,0) -- (11.6,0);
    \draw[loosely dashed](8.4,-4) -- (11.6,-4);
    \draw[loosely dashed](8.4,-0.4) -- (11.6,-3.6);
    \draw[loosely dashed](12,-0.5) -- (12,-3.2);
    \draw[loosely dashed](14,-0.5) -- (14,-3.2);
    \draw [-{Latex[scale=1.5]}](12.3,0) -- (13.7,0);
    \draw [-{Latex[scale=1.5]}](12.3,-4) -- (13.7,-4);
    \draw [-{Latex[scale=1.5]}] ($(v3) + (0.3,0.1)$) to[out=20,in=160] ($(v11) + (-0.3,0.1)$);
    \draw [-{Latex[scale=1.5]}] ($(v4) + (0.3,-0.1)$) to[out=-10,in=200] ($(v11) + (-0.3,-0.1)$);
    \draw [-{Latex[scale=1.5]}] ($(v5) + (0.3,0.1)$) to[out=20,in=160] ($(v12) + (-0.3,0.1)$);
    \draw [-{Latex[scale=1.5]}] ($(v6) + (0.3,-0.1)$) to[out=-10,in=200] ($(v12) + (-0.3,-0.1)$);
    \draw [-{Latex[scale=1.5]}](7.95,-4.2) -- (7.55,-5.8);
    \draw [-{Latex[scale=1.5]}](8.05,-4.2) -- (8.45,-5.8);
    \draw [-{Latex[scale=1.5]}](13.95,-4.2) -- (13.55,-5.8);
    \draw [-{Latex[scale=1.5]}](14.05,-4.2) -- (14.45,-5.8);
    \draw[loosely dashed](9.5,-6) -- (12.5,-6);
    \draw [-{Latex[scale=1.5]}] ($(v9) + (-0.1,-0.3)$) to[out=250,in=100] ($(v15) + (-0.1,0.3)$);
    \draw [-{Latex[scale=1.5]}] ($(v9) + (0.1,-0.3)$) to[out=290,in=80] ($(v16) + (0.1,0.3)$);
    \draw [-{Latex[scale=1.5]}] ($(v13) + (-0.1,-0.3)$) to[out=250,in=100] ($(v17) + (-0.1,0.3)$);
    \draw [-{Latex[scale=1.5]}] ($(v13) + (0.1,-0.3)$) to[out=290,in=80] ($(v18) + (0.1,0.3)$);
    \draw[-{Latex[scale=1.5]}] (7.775,-6.2) -- (10.1125,-7.9);
    \draw[-{Latex[scale=1.5]}] (8.825,-6.2) -- (11.5875,-7.9);
    \draw[-{Latex[scale=1.5]}] (13.175,-6.2) -- (10.4125,-7.9);
    \draw[-{Latex[scale=1.5]}] (14.225,-6.2) -- (11.8875,-7.9);
\end{tikzpicture}
\caption{\label{fig:rs2network}Illustration of the flow network described in the proof of Theorem~\ref{Prop:rs2example}.  The capacity constraints are $c(s_k,x_{ik}) = p_{i\bullet k}$, $c(x_{ik},y_{ij}^{(1)
}) = \infty$, $c(y_{ij}^{(1)},y_{ij}^{(2)}) = p_{ij\bullet}$, $c(y_{ij}^{(2)},z_{jk}) = \infty$ and $c(z_{jk},t_2) = p_{\bullet jk}$ for $i \in [r]$, $j \in [s]$ and $k \in [2]$.}
\end{figure}
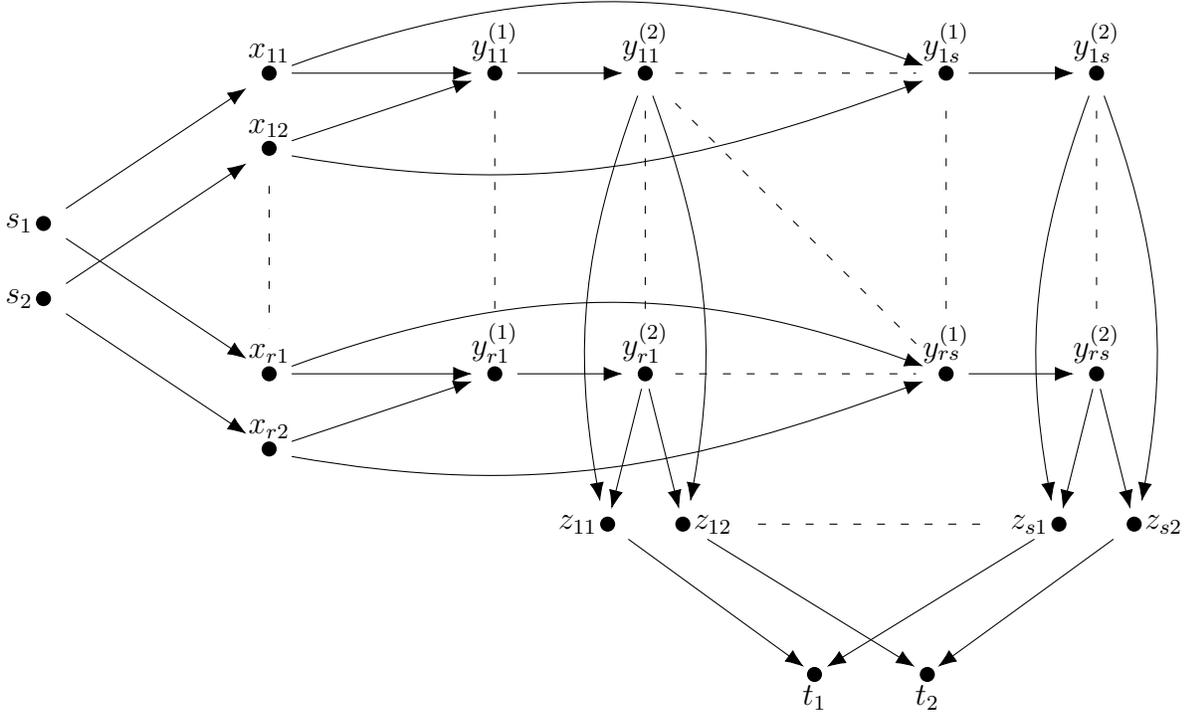

Having established the link between $R(P_\mathbb{S})$ and this network flow problem, we proceed to find a total flow that matches the upper bound implied by~\eqref{Eq:rs2lowerbound} and~\eqref{Eq:1minusdual}, i.e.
\begin{equation}
\label{Eq:ABObjFun}
    1 + 2 \min_{A \subseteq [r], B \subseteq [s]} (p_{AB \bullet} + p_{\bullet \bullet 1} - p_{A \bullet 1} - p_{\bullet B 1}) = \min_{A \subseteq [r], B \subseteq [s]}( p_{A^c \bullet 1} + p_{A \bullet 2} + p_{\bullet B^c 1} + p_{\bullet B 2} + p_{AB \bullet} + p_{A^c B^c \bullet}).
\end{equation}
The fact that the left-hand side of~\eqref{Eq:ABObjFun} is equal to the right-hand side relies on the consistency of $p_\mathbb{S}$. Let $A \subseteq [r]$ and $B \subseteq [s]$ be minimising sets in the above display, observing that the same choices minimise both left- and right-hand sides.  Then, for $i \in A$, we have
\[
    p_{AB \bullet} - p_{A \bullet 1} \leq p_{A\setminus \{i\} B \bullet} - p_{A \setminus \{i\} \bullet 1} = p_{AB \bullet} - p_{i B \bullet} - p_{A \bullet 1} + p_{i \bullet 1},
\]
so that $p_{i B \bullet} \leq p_{i \bullet 1}$. It is therefore possible to send a flow of commodity 1 of $p_{ij\bullet}$ from $s_1$ through $x_{i1}$ to $y_{ij}^{(2)}$, for each $(i,j) \in A \times B$. Similarly, by considering $i \in A^c$ and repeating the calculation above with $A \cup \{i\}$ in place of $A \setminus \{i\}$, we see that $p_{i B^c \bullet} \leq p_{i \bullet 2}$.  Hence a flow of commodity 2 of $p_{ij \bullet}$ can be sent from $s_2$ through $x_{i 2}$ to $y_{ij}^{(2)}$ for each $(i,j) \in A^c \times B^c$.  So far, then, we have shown how to send a flow of commodity 1 of $p_{AB\bullet}$ from $s_1$ to $\{z_{j1} : j \in B\}$, and a flow of commodity 2 of $p_{A^c B^c \bullet}$ from $s_2$ to $\{z_{j2}: j \in B^c\}$. 

We now claim that, for each $i \in A^c$, we may send a flow of commodity 1 of $p_{i\bullet 1}$ from~$s_1$ through $x_{i1}$ and $y_{iB}^{(2)}:=\{y_{ij}^{(2)}:j \in B\}$ to $z_{B1}:=\{z_{j1}:j\in B\}$, and that this flow together with the previous flow of commodity 1 can pass through $z_{B1}$ to $t_1$. 
To do this we use a generalisation of Hall's marriage theorem to one-commodity flows due to~\cite{gale1957theorem}. Each~$z_{j1}$ for $j \in B$ already has an incoming flow of $p_{Aj\bullet}$, so has a remaining capacity of $p_{\bullet j1} - p_{A j \bullet}$.  By Gale's theorem, the desired flow is therefore feasible if and only if, for every $A' \subseteq A^c$ and $B' \subseteq B$, we have
\[
    \sum_{i \in A'} p_{i \bullet 1} - \sum_{j \in B \setminus B'} (p_{\bullet j 1} - p_{A j \bullet}) \leq \sum_{i \in A'} \sum_{j \in B'}  p_{ij \bullet}.
\]
This condition is equivalent to the condition that, for all $A' \subseteq A^c$ and $B' \subseteq B$ we have
\[
    p_{(A \cup A')B' \bullet} - p_{(A \cup A') \bullet 1} - p_{\bullet B' 1} \geq p_{AB \bullet} - p_{A \bullet 1} - p_{\bullet B 1},
\]
but we know that this is true because $(A,B)$ are minimisers of the left-hand side of~\eqref{Eq:ABObjFun}.  Thus, the desired flow of commodity 1 is feasible.  Similarly, for each $i \in A$, we may send a flow of $p_{i \bullet 1}$ of commodity 2 from $s_2$ through $x_{i1}$ and $y_{iB^c}^{(2)}:=\{y_{ij}^{(2)}:j\in B^c\}$ to $z_{B^c1}:=\{z_{j1}:j\in B^c\}$, and this flow can pass through to $t_2$.  We have therefore now shown that we can send a combined flow of $p_{AB \bullet} + p_{A^c B^c \bullet} + p_{A^c \bullet 1} + p_{A \bullet 2}$ from the sources to the sinks.

Until this point, no flow has been routed through $z_{B2}$ or $z_{B^c 1}$.  To conclude our proof, then, we now claim that it is possible to introduce an additional flow of $p_{\bullet B 2}$ of commodity 2, as well as $p_{\bullet B^c 1}$ of commodity 1 into the network, to put all edges from $z_{B2}$ to $t_2$ and from $z_{B^c1}$ to $t_1$ at full capacity.  Consider any maximal flow in the network; we wish to determine the maximal amount of commodity~2 that can be sent from $s_2$ through $x_{A^c2}$ and $y_{A^c B}^{(2)}$ to $z_{B2}$ and thus to $t_2$, in addition to the existing flow.  To this end, suppose that there exists $j \in B$ with the edge from $z_{j2}$ to $t_2$ at less than full capacity. Then, since the flow is maximal, it must be the case that for each $i \in A^c$, the flow of commodity 2 from $s_1$ to $x_{i2}$ is full (i.e.~equal to $p_{i \bullet 2})$, or the flow from $y_{ij}^{(1)}$ to $y_{ij}^{(2)}$ is full. However, if the flow from $s_1$ to $x_{i2}$ is equal to $p_{i \bullet 2}$, then the total flow from $\{x_{i1},x_{i2}\}$ must be equal to $p_{i \bullet 1} + p_{i \bullet 2} = p_{i \bullet \bullet} = \sum_{j'=1}^s p_{ij' \bullet}$. In this case, the edge from $y_{ij}^{(1)}$ to $y_{ij}^{(2)}$ must be full.  So, if the edge from $z_{j2}$ to $t_2$ is not full, then the edge from $y_{ij}^{(1)}$ to $y_{ij}^{(2)}$ is full for each $i \in A^c$ (and each $i \in A$ from the earlier flow).  It follows that, in this case, there is a flow of $\sum_{i=1}^r p_{ij \bullet} = p_{\bullet j \bullet} = p_{\bullet j 1} + p_{\bullet j 2}$ from $y_{[r]j}^{(1)}$ to $y_{[r]j}^{(2)}$.  But such a flow would put both edges $z_{j1}$ to $t_1$ and $z_{j2}$ to $t_2$ at full capacity, contradicting our original hypothesis.  Hence, at any maximal flow, all edges from $z_{B2}$ to $t_2$ are full, and similarly, all edges from $z_{B^c1}$ to $t_1$ are full.  Thus, we can indeed send the desired additional flow through the network, and we deduce that the total capacity of the network is at least the expression on the right-hand side of~\eqref{Eq:ABObjFun}.  We conclude from~\eqref{Eq:1minusdual} and~\eqref{Eq:ABObjFun} that
\[
R(P_\mathbb{S}) \leq 2 \max \Bigl\{ 0, \max_{A \subseteq [r], B \subseteq[s]} (-p_{AB \bullet} +p_{A \bullet 1} + p_{\bullet B 1} - p_{\bullet \bullet 1}) \Bigr\},
\]
and this completes the proof of the first part of the theorem.

We now turn to the second part of our result.  We first show
that $p_\mathbb{S}^* \in \mathcal{P}_\mathbb{S}^{0,*} + \mathcal{P}_\mathbb{S}^{\mathrm{cons},**}$ if and only if $p_\mathbb{S}^* \in \mathcal{P}_\mathbb{S}^{\mathrm{cons},*}$ and
\[
    \max\bigl\{ 1_\mathcal{X}^T p : p \in [0,\infty)^\mathcal{X}, \mathbb{A}p \leq p_\mathbb{S}^* \bigr\} \geq (p_{\bullet \bullet \bullet}^*-1)_+.
\]
If $p_{\bullet \bullet \bullet}^* \leq 1$, then $p_\mathbb{S}^* \in \mathcal{P}_\mathbb{S}^{\mathrm{cons},**}$ and there is nothing to prove, so we assume that $p_{\bullet \bullet \bullet}^* >1$. If $p_\mathbb{S}^* \in \mathcal{P}_\mathbb{S}^{0,*} + \mathcal{P}_\mathbb{S}^{\mathrm{cons},**}$, then we may write $p_\mathbb{S}^* = \mathbb{A}p + r_\mathbb{S}$ with $p \in [0,\infty)^\mathcal{X}$ and $r_\mathbb{S} \in \mathcal{P}_\mathbb{S}^{\mathrm{cons},**}$. Then
\[
     \max\bigl\{ 1_\mathcal{X}^T p' : p' \in [0,\infty)^\mathcal{X}, \mathbb{A}p' \leq p_\mathbb{S}^* \bigr\} \geq 1_\mathcal{X}^T p = \frac{1}{|\mathbb{S}|} \biggl( \sum_{S \in \mathbb{S}} 1_S \biggr)^T \mathbb{A}p = p_{\bullet \bullet \bullet}^* - r_{\bullet \bullet \bullet}\geq p_{\bullet \bullet \bullet}^* -1.
\]
On the other hand, suppose that $p_\mathbb{S}^* \in \mathcal{P}_\mathbb{S}^{\mathrm{cons},*}$ and that there exists $p\in [0,\infty)^\mathcal{X}$ with $\mathbb{A}p \leq p_\mathbb{S}^*$ and $1_\mathcal{X}^T p \geq p_{\bullet \bullet \bullet}^*-1$.  Then we certainly have $r_\mathbb{S}=p_\mathbb{S}^* - \mathbb{A}p \in \mathcal{P}_\mathbb{S}^{\mathrm{cons},*}$.  But since we also have $r_{\bullet \bullet \bullet} = p_{\bullet \bullet \bullet}^* - 1_\mathcal{X}^Tp \leq 1$, it follows that $r_\mathbb{S} \in \mathcal{P}_\mathbb{S}^{\mathrm{cons},**}$, and we have proved our claim. Now, the proof of the first part of the result shows that
\begin{align*}
    \max\bigl\{ 1_\mathcal{X}^T p : p \in [0,\infty)^\mathcal{X}, \mathbb{A}p \leq p_\mathbb{S}^* \bigr\} &= \min_{A \subseteq [r], B \subseteq [s]}( p_{A^c \bullet 1}^* + p_{A \bullet 2}^* + p_{\bullet B^c 1}^* + p_{\bullet B 2}^* + p_{AB \bullet}^* + p_{A^c B^c \bullet}^*) \\
    &= p_{\bullet \bullet \bullet}^* + 2 \min_{A \subseteq [r], B \subseteq [s]} (p_{AB \bullet}^* + p_{\bullet \bullet 1}^* - p_{A \bullet 1}^* - p_{\bullet B 1}^*).
\end{align*}
When $p_{\bullet \bullet \bullet}^* \geq 1$, we therefore have $p_\mathbb{S}^* \in \mathcal{P}_\mathbb{S}^{0,*} + \mathcal{P}_\mathbb{S}^{\mathrm{cons},**}$ if and only if $p_\mathbb{S}^* \in \mathcal{P}_\mathbb{S}^{\mathrm{cons},*}$ and 
\[
    1 + 2 \min_{A \subseteq [r], B \subseteq [s]} (p_{AB \bullet}^* + p_{\bullet \bullet 1}^* - p_{A \bullet 1}^* - p_{\bullet B 1}^*) \geq 0,
\]
as claimed.  On the other hand, when $p_{\bullet \bullet \bullet}^* <1$ and $p_\mathbb{S}^* \in \mathcal{P}_\mathbb{S}^{\mathrm{cons},*}$, we always have $p_\mathbb{S}^* \in \mathcal{P}_\mathbb{S}^{\mathrm{cons},**} \subseteq \mathcal{P}_\mathbb{S}^{0,*} + \mathcal{P}_\mathbb{S}^{\mathrm{cons},**}$, and moreover
\begin{align*}
    1 + 2 \min_{A \subseteq [r], B \subseteq [s]} (p_{AB \bullet}^* + p_{\bullet \bullet 1}^* - p_{A \bullet 1}^* - p_{\bullet B 1}^*) &> p_{\bullet \bullet \bullet}^* + 2 \min_{A \subseteq [r], B \subseteq [s]} (p_{AB \bullet}^* + p_{\bullet \bullet 1}^* - p_{A \bullet 1}^* - p_{\bullet B 1}^*) \\
    &= \max\bigl\{ 1_\mathcal{X}^T p : p \in [0,\infty)^\mathcal{X}, \mathbb{A}p \leq p_\mathbb{S}^* \bigr\} \geq 0.
\end{align*}
Combining both cases, we have now shown that
\[
    \mathcal{P}_\mathbb{S}^{0,*} + \mathcal{P}_\mathbb{S}^{\mathrm{cons},**} = \Bigl\{p_\mathbb{S}^* \in \mathcal{P}_\mathbb{S}^{\mathrm{cons},*} : 1 + 2 \min_{A \subseteq [r], B \subseteq [s]} (p_{AB \bullet}^* + p_{\bullet \bullet 1}^* - p_{A \bullet 1}^* - p_{\bullet B 1}^*) \geq 0 \Bigr\},
\]
as required.
\end{proof}

The proof of our lower bound in Theorem~\ref{Prop:rs2lowerbound} relies on the following lemma, which is an extension of both \citet[][Lemma~3]{wu2016minimax} and \citet[][Lemma~32]{jiao2018minimax}.
\begin{lemma}
\label{Lemma:PoissonLemma}
Let $V,V'$ be random variables supported on $[\lambda/2-M,\lambda/2+M]$ for some $M \leq \lambda/2$, and suppose that $\mathbb{E}(V^\ell) = \mathbb{E}\bigl((V')^\ell\bigr)$ for $\ell \in [L]$.  Let $Q$ denote the distribution on $\mathbb{Z}^2$ of $(W_1,W_2)^T$, where, conditional on $V=v$, we have that $W_1$ and $W_2$ are independent, with $W_1|V=v \sim \mathrm{Poi}(v)$ and $W_2|V=v \sim \mathrm{Poi}(\lambda-v)$.  Define $Q'$ in terms of $V'$ analogously.  Then
\[
    d_\mathrm{TV}(Q,Q') \leq \frac{2^{1/2}}{\pi^{1/4}} \Bigl( \frac{2eM^2}{\lambda(L+1)} \Bigr)^{(L+1)/2}
\]
whenever $L + 2 \geq 8M^2/\lambda$.
\end{lemma}
\begin{proof}[Proof of Lemma~\ref{Lemma:PoissonLemma}]
Let $U := (V-\lambda/2)/M$ and $U' := (V'-\lambda/2)/M$, and for $m \in \mathbb{N}$ and $x \in \mathbb{R}$, let $(x)_m := x(x-1)\ldots(x-m+1)$ for the falling factorial (with $(x)_0 := 1$).  Letting $Y,Z \sim \mathrm{Poi}(\lambda/2)$ be  independent, we have
\begin{align}
\label{Eq:FirstStep}
    &d_\mathrm{TV}(Q,Q') = \frac{1}{2} e^{-\lambda} \sum_{i,j=0}^\infty \frac{1}{i! j!} \Bigl| \mathbb{E} \bigl\{V^i (\lambda - V)^j - (V')^i(\lambda-V')^j \bigr\} \Bigr| \nonumber \\
    & = \frac{1}{2} e^{-\lambda} \sum_{i,j=0}^\infty \frac{(\lambda/2)^{i+j}}{i! j!} \Bigl| \mathbb{E} \bigl\{(1+2MU/\lambda)^i (1 - 2MU/\lambda)^j - (1+2MU'/\lambda)^i(1-2MU'/\lambda)^j \bigr\} \Bigr| \nonumber \\
     & = \frac{1}{2} e^{-\lambda} \sum_{i,j=0}^\infty \frac{(\lambda/2)^{i+j}}{i! j!} \biggl| \mathbb{E} \sum_{k=0}^{i+j} \Bigl( \frac{2M}{\lambda} \Bigr)^k \{U^k - (U')^k \} \sum_{m=0}^k \binom{i}{m} \binom{j}{k-m} (-1)^{k-m} \biggr| \nonumber \\
     & \leq e^{-\lambda} \sum_{i,j=0}^\infty \frac{(\lambda/2)^{i+j}}{i! j!} \sum_{k=0}^{i+j} \Bigl( \frac{2M}{\lambda} \Bigr)^k \mathbbm{1}_{\{k \geq L+1\}} \biggl| \sum_{m=0}^k \binom{i}{m} \binom{j}{k-m} (-1)^{k-m} \biggr| \nonumber \\
     & =  \sum_{k=L+1}^\infty \frac{1}{k!} \Bigl( \frac{2M}{\lambda} \Bigr)^k \mathbb{E} \biggl| \sum_{m=0}^k (-1)^m \binom{k}{m} (Y)_m (Z)_{k-m} \biggr| \nonumber \\
     & \leq \sum_{k=L+1}^\infty \frac{1}{k!} \Bigl( \frac{2M}{\lambda} \Bigr)^k \mathbb{E}^{1/2} \biggl[ \biggl\{ \sum_{m=0}^k (-1)^m \binom{k}{m} (Y)_m (Z)_{k-m} \biggr\}^2 \biggr].
\end{align}
We now bound this second moment using the facts that $(x)_m(x)_n = \sum_{\ell=0}^m \binom{m}{\ell} \binom{n}{\ell} \ell! (x)_{m+n-\ell}$ and $\mathbb{E}(Y)_m = (\lambda/2)^m$ for all $m,n \in \mathbb{N}_0$ to write
\begin{align}
\label{Eq:SecondStep}
    \mathbb{E} \biggl[ \biggl\{ \sum_{m=0}^k (-1)^m &\binom{k}{m} (Y)_m (Z)_{k-m} \biggr\}^2 \biggr] \nonumber \\
    &= \sum_{m,n=0}^k (-1)^{m+n} \binom{k}{m} \binom{k}{n} \mathbb{E} \{ (Y)_m (Y)_n \} \mathbb{E}\{ (Z)_{k-m} (Z)_{k-n} \} \nonumber \\
    &= \sum_{m,n=0}^k (-1)^{m+n} \binom{k}{m} \binom{k}{n} \sum_{\ell,r=0}^\infty \binom{m}{\ell} \binom{n}{\ell } \binom{k-m}{r} \binom{k-n}{r} \ell! r! (\lambda/2)^{2k - \ell - r} \nonumber \\
    & = \sum_{\ell,r=0}^\infty \ell! r! (\lambda/2)^{2k-\ell -r } \biggl\{ \sum_{m=0}^k (-1)^m \binom{k}{m} \binom{m}{\ell}  \binom{k-m}{r} \biggr\}^2.
\end{align}
Now, terms with $\ell+r > k$ are zero, because either $\ell>m$ or $r > k-m$. We can think of $\binom{m}{\ell} \binom{k-m}{r}$ as a polynomial of degree $\ell+r$ in $m$, and use the fact that $\sum_{m=0}^k (-1)^m \binom{k}{m} m^s = 0$ for non-negative integers $s < k$ to conclude that the only non-zero terms are those with $\ell+r=k$.  We now use the fact that $\sum_{m=0}^k (-1)^m \binom{k}{m} m^k = (-1)^k k!$ to see that
\begin{align}
\label{Eq:ThirdStep}
    \sum_{m=0}^k (-1)^m \binom{k}{m} \binom{m}{\ell}  \binom{k-m}{r} &= \frac{\mathbbm{1}_{\{\ell+r=k\}}}{\ell! r!} \sum_{m=0}^k (-1)^m \binom{k}{m} (m)_\ell (k-m)_r \nonumber \\
    &= \frac{\mathbbm{1}_{\{\ell+r=k\}}}{\ell! r!} (-1)^{r+k} k!.
\end{align}
From~\eqref{Eq:FirstStep},~\eqref{Eq:SecondStep} and~\eqref{Eq:ThirdStep} together with Stirling's inequality \citep[e.g.][p.~847]{dumbgen2021bounding}, we deduce that when $L +2 \geq 8M^2/\lambda$, we have
\begin{align*}
    d_\mathrm{TV}(Q,Q')     & \leq \sum_{k=L+1}^\infty \frac{1}{k!} \Bigl( \frac{2M}{\lambda} \Bigr)^k \biggl\{ \sum_{\ell,r=0}^\infty \ell! r! (\lambda/2)^{2k-\ell-r} \frac{\mathbbm{1}_{\{\ell+r=k\}}}{(\ell!)^2 (r!)^2} (k!)^2 \biggr\}^{1/2} \\
    & = \sum_{k=L+1}^\infty \frac{1}{k!} \Bigl( \frac{2M}{\lambda} \Bigr)^k \biggl\{ k! (\lambda/2)^k \sum_{\ell=0}^k \binom{k}{\ell} \biggr\}^{1/2} = \sum_{k=L+1}^\infty \frac{1}{(k!)^{1/2}} \Bigl( \frac{2M^2}{\lambda} \Bigr)^{k/2} \\
    & \leq \frac{2}{\{(L+1)!\}^{1/2}} \Bigl( \frac{2M^2}{\lambda} \Bigr)^{(L+1)/2} \leq \frac{2}{\{2 \pi (L+1)\}^{1/4}} \Bigl( \frac{2eM^2}{\lambda(L+1)} \Bigr)^{(L+1)/2} \\
    & \leq \frac{2^{1/2}}{\pi^{1/4}} \Bigl( \frac{2eM^2}{\lambda(L+1)} \Bigr)^{(L+1)/2},
\end{align*}
as required.
\end{proof}

\begin{proof}[Proof of Theorem~\ref{Prop:rs2lowerbound}]
Assume without loss of generality that $n_{\{1,2\}} \leq n_{\{1,3\}}$.  We will start by showing that we may work in a Poisson sampling model without changing the separation rates.  Extending our previous setting, let $(X_{S,i})_{S \in \mathbb{S},i \in \mathbb{N}}$ denote independent random variables, with $X_{S,i} \sim P_S$, and let $N_{\mathbb{S}}:=(N_S : S \in \mathbb{S})$ be an independent sequence of Poisson random variables, independent of $(X_{S,i})_{S \in \mathbb{S},i \in \mathbb{N}}$, with $\mathbb{E} N_S = n_S$ for all $S \in \mathbb{S}$.  Let $\Psi'_\mathbb{S}$ denote the set of sequences of tests of the form $(\psi'_{n_{\mathbb{S}}'} \in \Psi_{n_{\mathbb{S}}'}:n_\mathbb{S}' \in \mathbb{N}_0^{\mathbb{S}})$, and write
\[
\mathcal{R}^\mathrm{Poi}(n_\mathbb{S},\rho) := \inf_{\psi'_{\mathbb{S}} \in \Psi'_{\mathbb{S}}}\biggl\{ \sup_{P_\mathbb{S} \in \mathcal{P}_\mathbb{S}^0} \mathbb{E}_{P_\mathbb{S}}(\psi'_{N_{\mathbb{S}}}) + \sup_{P_\mathbb{S} \in \mathcal{P}_\mathbb{S}(\rho) }\mathbb{E}_{P_\mathbb{S}}(1-\psi'_{N_{\mathbb{S}}})\biggr\}.
\]
Here, the expectations are taken over the randomness both in the data and in the sample sizes.  Since $\mathcal{R}( n_{\mathbb{S}}', \rho) \leq \mathcal{R}(n_{\mathbb{S}}'',\delta)$ whenever $n_S' \geq n_S''$ for all $S \in \mathbb{S}$, we have that 
\begin{align*}
\mathcal{R}^\mathrm{Poi}(n_\mathbb{S},\rho) &= \inf_{\psi'_{\mathbb{S}} \in \Psi'_{\mathbb{S}}}\biggl\{\sup_{P_\mathbb{S} \in \mathcal{P}_\mathbb{S}^0} \sum_{n_\mathbb{S}' \in \mathbb{N}_0^\mathbb{S}} \mathbb{P}(N_\mathbb{S} \!=\! n_\mathbb{S}') \mathbb{E}_{P_\mathbb{S}} (\psi'_{n_\mathbb{S}'}) +\sup_{P_\mathbb{S} \in \mathcal{P}_\mathbb{S}(\rho) } \sum_{n_\mathbb{S}' \in \mathbb{N}_0^\mathbb{S}} \mathbb{P}(N_\mathbb{S} \!=\! n_\mathbb{S}') \mathbb{E}_{P_\mathbb{S}} (1-\psi'_{n_\mathbb{S}'}) \biggr\} \\
 &\leq \inf_{\psi'_{\mathbb{S}} \in \Psi'_{\mathbb{S}}} \sum_{n_\mathbb{S}' \in \mathbb{N}_0^\mathbb{S}} \mathbb{P}(N_\mathbb{S} = n_\mathbb{S}') \biggl\{  \sup_{P_\mathbb{S} \in \mathcal{P}_\mathbb{S}^0} \mathbb{E}_{P_\mathbb{S}} (\psi'_{n_\mathbb{S}'}) + \sup_{P_\mathbb{S} \in \mathcal{P}_\mathbb{S}(\rho) }\mathbb{E}_{P_\mathbb{S}} (1-\psi'_{n_\mathbb{S}'}) \biggr\} \\
 &= \sum_{n_\mathbb{S}' \in \mathbb{N}_0^\mathbb{S}} \mathbb{P}(N_\mathbb{S} = n_\mathbb{S}')  \mathcal{R}(n_\mathbb{S}',\rho) \\
 &\leq \mathcal{R}(\lceil n_\mathbb{S}/2\rceil,\rho)\prod_{S \in \mathbb{S}} \mathbb{P}(N_S \geq \lceil n_S/2\rceil) + \sum_{S \in \mathbb{S}} \mathbb{P}(N_S < \lceil n_S/2\rceil) \\
 &\leq \mathcal{R}(\lceil n_\mathbb{S}/2\rceil,\rho) + \sum_{S \in \mathbb{S}} e^{-n_S/12}.
 \end{align*}
Here, in the final inequality, we have used the fact that when $W \sim \mathrm{Poi}(\lambda)$, we have
\[
\mathbb{P}(W - \lambda \leq -x) \leq e^{-\frac{x^2}{2(\lambda + x)}}
\]
for all $x \geq 0$.

We will construct priors for consistent $P_\mathbb{S}$ over the null and alternative hypotheses that  satisfy $p_{\bullet 1 \bullet } = p_{\bullet \bullet 1} = 1/2$, $p_{\bullet 21} \geq 1/4$, and $p_{i\bullet \bullet} =1/r$ and $p_{i \bullet 1} =1/(2r)$ for each $i \in [r]$.  By~\eqref{Eq:RS3s2}, for such $P_\mathbb{S}$ we have
\begin{align*}
    R(P_\mathbb{S}) &= 2 \max_{j \in [2]} \biggl\{ p_{\bullet j 1} - \sum_{i=1}^r \min \Bigl( p_{ij\bullet}, \frac{1}{2r} \Bigr) \biggr\}_+ \\
    &=  \max_{j \in [2]} \biggl\{ 2p_{\bullet j 1} - \sum_{i=1}^r \biggl( p_{ij\bullet } + \frac{1}{2r} - \biggl|p_{i j \bullet }- \frac{1}{2r} \biggr| \biggr) \biggr\}_+ \\
    & = \biggl\{ \sum_{i=1}^r \biggl| p_{i1\bullet } - \frac{1}{2r} \biggr| -\frac{1}{2} + \max(2p_{\bullet 1 1} - p_{\bullet 1 \bullet }, 2p_{\bullet 2 1} - p_{\bullet 2 \bullet } ) \biggr\}_+ \\
    & = \biggl\{ \sum_{i=1}^r \biggl| p_{i1\bullet } - \frac{1}{2r} \biggr| -\frac{1}{2} + \max( 1/2 -2p_{\bullet 2 1} , 2p_{\bullet 2 1}  - 1/2 ) \biggr\}_+ \\
    & = \biggl( \sum_{i=1}^r \biggl| p_{i1\bullet } - \frac{1}{2r} \biggr| + 2p_{\bullet 21} -1 \biggr)_+.
\end{align*}
We now construct our priors using results from \citet{jiao2018minimax}; see also \citet{cai2011testing} and \citet{wu2016minimax}. Set $L:= \lceil 2e \log r \rceil$ and let $\nu_0,\nu_1$ be probability distributions on $[-1,1]$ satisfying:
\begin{itemize}
    \item $\nu_0$ and $\nu_1$ are symmetric about $0$;
    \item $\int_{-1}^1 t^\ell \, d\nu_0(t) = \int_{-1}^1 t^\ell \, d\nu_1(t)$ for $\ell=0,1,\ldots,L$;
    \item $\int_{-1}^1 |t| \, d\nu_1(t) - \int_{-1}^1 |t| \, d\nu_0(t) = 2 E_L$,
\end{itemize}
where $E_L \equiv E_L\bigl[ |\cdot|; [-1,1] \bigr]$ is the error in uniform norm of the best degree-$L$ polynomial approximation to the function $x \mapsto |x|$ on $[-1,1]$. The existence of such distributions $\nu_0$ and~$\nu_1$ follows from \citet[][Lemma~29]{jiao2018minimax}. We recall that $E_L = \beta_* L^{-1} \{1+o(1)\}$ as $L \rightarrow \infty$, where $\beta_* \approx 0.2802$ is the Bernstein constant \citep{bernstein1914meilleure}. Define $g:[-1,1] \rightarrow \mathbb{R}$ by
\[
    g(x) := \frac{1}{r} + \delta x, \quad \text{where } \delta:=\frac{1}{r} \wedge \biggl( \frac{\log r}{n_{\{1,2\}} r} \biggr)^{1/2}.
\]
Further, writing $a := 1/r - \delta \geq 0$ and $b := 1/r + \delta \leq 2/r$, define distributions $\mu_0$ and $\mu_1$ on $[a,b]$ by $\mu_j := \nu_j \circ g^{-1}$ for $j=0,1$.  These distributions satisfy
\begin{itemize}
    \item $\int_a^b t \, d\mu_0(t) = \int_a^b t \, d\mu_1(t) = 1/r$;
    \item $\int_a^b t^\ell \, d\mu_0(t) = \int_a^b t^\ell \,  d\mu_1(t)$ for $\ell=2,3,\ldots,L$;
    \item $\int_a^b |t-1/r| \, d\mu_1(t) - \int_a^b |t-1/r | \, d\mu_0(t) = 2 \delta E_L$. 
\end{itemize}
Since $\rho^*(n_\mathbb{S})$ is increasing in $r$, we may assume without loss of generality that $r$ is even.  We will write $\sigma_0$ and $\sigma_1$ for our priors under the null and alternative hypotheses respectively.  For $\sigma_j$ with $j \in \{0,1\}$ and for odd $i \in [r]$, generate $2p_{i1\bullet }$ independently from $\mu_j$. For even $i \in [r]$, set $p_{i 1 \bullet }:=1/r-p_{i-1, 1 \bullet }$ so that $p_{\bullet 1 \bullet} = 1/2$ with probability one. Given $(p_{i1\bullet })_{i=1}^r$, take $p_{i2 \bullet } := 1/r - p_{i1 \bullet }$ and $p_{i \bullet 1 } = p_{i \bullet 2}=1/(2r)$, so that $p_{i \bullet \bullet} = 1/r$ and $p_{\bullet \bullet 1} = 1/2$.  Write
\begin{align*}
    \chi := \mathbb{E}_{\sigma_1} \sum_{i=1}^r \Bigl| p_{i 1 \bullet} - \frac{1}{2r}\Bigr| - \mathbb{E}_{\sigma_0} \sum_{i=1}^r \Bigl|p_{i1 \bullet} - \frac{1}{2r}\Bigr| &=  r \delta E_L
\end{align*}
and set
\[
    \zeta := \frac{1}{2} \mathbb{E}_{\sigma_1} \sum_{i=1}^r \Bigl| p_{i 1 \bullet} - \frac{1}{2r}\Bigr| + \frac{1}{2} \mathbb{E}_{\sigma_0} \sum_{i=1}^r \Bigl|p_{i1 \bullet} - \frac{1}{2r}\Bigr| \leq 1/2.
\]
Our prior distributions are fully specified upon choosing $p_{\bullet 21} := 1/2 - (\zeta-\chi/4)/2 \geq 1/4$.  For $j \in \{0,1\}$, let
\[
\Omega_{0,j} := \biggl\{ (-1)^j \biggl( \sum_{i=1}^r \Bigl|p_{i1 \bullet} -\frac{1}{2r}\Bigr| - \mathbb{E}_{\sigma_j} \sum_{i=1}^r \Bigl|p_{i 1 \bullet} - \frac{1}{2r} \Bigr| \biggr) \leq \frac{\chi}{4} \biggr\}.
\]
Then, noting that the even terms in the sum are equal to the odd terms, by Hoeffding's inequality,
\begin{align*}
    \mathbb{P}_{\sigma_j}(\Omega_{0,j}^c) \leq  \exp \biggl( - \frac{\chi^2 }{16 r \delta^2} \biggr) = e^{-rE_L^2/16}.
\end{align*}
Moreover, on $\Omega_{0,0}$,
\begin{align*}
    \sum_{i=1}^r \biggl| p_{i1\bullet } &- \frac{1}{2r} \biggr| + 2p_{\bullet 21} - 1 \\
    &\leq \mathbb{E}_{\sigma_0} \biggl( \sum_{i=1}^r \biggl| p_{i1\bullet } - \frac{1}{2r} \biggr| + 2p_{\bullet 21} - 1 \biggr) + \chi/4 = \zeta - \chi/2 -(\zeta-\chi/4) + \chi/4 = 0,
\end{align*}
so that $\sigma_0\bigl((\mathcal{P}_\mathbb{S}^0)^c\bigr) = \mathbb{P}_{\sigma_0}\bigl\{R(P_{\mathbb{S}}) > 0\bigr\} \leq e^{-rE_L^2/16}$.  On the other hand,  on $\Omega_{0,1}$, 
\begin{align*}
    \sum_{i=1}^r \biggl| p_{i1\bullet } &- \frac{1}{2r} \biggr| + 2p_{\bullet 21} - 1 \\
     &\geq \mathbb{E}_{\mu_1} \biggl( \sum_{i=1}^r \biggl| p_{i1\bullet } - \frac{1}{2r} \biggr| + 2p_{\bullet 21} - 1 \biggr) - \chi/4 = \zeta + \chi/2 - (\zeta-\chi/4) - \chi/4 = \chi/2, 
\end{align*}
so that $\sigma_1\bigl(\mathcal{P}_\mathbb{S}(\chi/2)^c\bigr) = \mathbb{P}_{\sigma_1}\bigl\{R(P_{\mathbb{S}}) < \chi/2\bigr\} \leq e^{-rE_L^2/16}$. 


We finally bound the total variation distance between the marginal distributions of the data, using similar arguments to those in  \cite{wu2016minimax}.  We have
\begin{align*}
    \mathcal{R}^\mathrm{Poi}(n_\mathbb{S}, \chi/2) &\geq \inf_{\psi_\mathbb{S}' \in \Psi_\mathbb{S}'} \Bigl[ \mathbb{E}_{\sigma_0} \bigl\{ \mathbbm{1}_{\{P_\mathbb{S} \in \mathcal{P}_\mathbb{S}^0\}} \mathbb{E}_{P_\mathbb{S}}(\psi'_{N_\mathbb{S}}) \bigr\} + \mathbb{E}_{\sigma_1} \bigl\{ \mathbbm{1}_{\{P_\mathbb{S} \in \mathcal{P}_\mathbb{S}(\chi/2)\}} \mathbb{E}_{P_\mathbb{S}}(1-\psi'_{N_\mathbb{S}}) \bigr\} \Bigr] \\
    & \geq \inf_{\psi_\mathbb{S}' \in \Psi_\mathbb{S}'} \Bigl[ \mathbb{E}_{\sigma_0} \bigl\{ \mathbb{E}_{P_\mathbb{S}}(\psi'_{N_\mathbb{S}}) \bigr\} + \mathbb{E}_{\sigma_1} \bigl\{  \mathbb{E}_{P_\mathbb{S}}(1-\psi'_{N_\mathbb{S}}) \bigr\} \Bigr] - \sigma_0\bigl((\mathcal{P}_\mathbb{S}^0)^c\bigr) -\sigma_1\bigl(\mathcal{P}_\mathbb{S}(\chi/2)^c\bigr) \\
    & \geq 1- d_\mathrm{TV}\bigl( \mathbb{E}_{\sigma_0} P_\mathbb{S}^{ n_\mathbb{S}}, \mathbb{E}_{\sigma_1} P_\mathbb{S}^{n_\mathbb{S}} \bigr) - 2e^{-rE_L^2/16},
\end{align*}
where, for $j=0,1$, we write $\mathbb{E}_{\sigma_j} P_\mathbb{S}^{n_\mathbb{S}}$ for the marginal distribution of $(X_{S,i})_{S \in \mathbb{S},i \in \mathbb{N}}$ in our Poisson model when the prior distribution for $P_\mathbb{S}$ is $\sigma_j$.  The distributions $P_{\{2,3\}}$ and $P_{\{1,3\}}$ are deterministic and do not change between the two priors, so 
\[
d_\mathrm{TV}\bigl( \mathbb{E}_{\sigma_0} P_\mathbb{S}^{ n_\mathbb{S}}, \mathbb{E}_{\sigma_1} P_\mathbb{S}^{n_\mathbb{S}} \bigr) = d_\mathrm{TV}\bigl( \mathbb{E}_{\sigma_0} P_{\{1,2\}}^{ n_{\{1,2\}}}, \mathbb{E}_{\sigma_1} P_{\{1,2\}}^{n_{\{1,2\}}} \bigr),
\]
where, for $j=0,1$, $\mathbb{E}_{\sigma_j}P_{\{1,2\}}^{n_{\{1,2\}}}$ denotes the marginal distribution of $(X_{\{1,2\},i})_{i \in \mathbb{N}}$ in our Poisson model when the prior distribution for $(p_{i\ell\bullet})_{i \in [r],\ell \in [2]}$ is taken from the construction of $\sigma_j$.  Under our Poisson sampling scheme,  since $(p_{i 1 \bullet })_{i \, \mathrm{ odd}}$ is an independent sequence, it suffices to bound the total variation distance between the distributions of random vectors $(Y_1,Y_2,Y_3,Y_4)$ and $(Z_1,Z_2,Z_3,Z_4)$, where $V \sim n_{\{1,2\}} \mu_0/2$, $V' \sim n_{\{1,2\}} \mu_1/2$, and with $\lambda:=n_{\{1,2\}}/r$, we have
\[
 (Y_1,Y_2,Y_3,Y_4) | V=v \sim  \mathrm{Poi}(v)\otimes \mathrm{Poi}(\lambda-v) \otimes \mathrm{Poi}(\lambda-v) \otimes \mathrm{Poi}(v)
\]
for all $v$, and $(Z_1,Z_2,Z_3,Z_4) | V'=v \overset{d}{=} (Y_1,Y_2,Y_3,Y_4) | V=v$ for all $v$. We now have that
\[
    d_\mathrm{TV}\bigl( \mathbb{E}_{\sigma_0} P_{\{1,2\}}^{ n_{\{1,2\}}}, \mathbb{E}_{\sigma_1} P_{\{1,2\}}^{n_{\{1,2\}}} \bigr) \leq \frac{r}{2}d_\mathrm{TV}\bigl(\mathcal{L}(Y_1,Y_2,Y_3,Y_4),\mathcal{L}(Z_1,Z_2,Z_3,Z_4)\bigr).
\]
Recalling that $V$ and $V'$ have identical $\ell$th moments for $\ell \in [L]$, we have by Lemma~\ref{Lemma:PoissonLemma} above that
\begin{align*}
    d_\mathrm{TV}\bigl(\mathcal{L}(Y_1,Y_2,Y_3,Y_4)&,\mathcal{L}(Z_1,Z_2,Z_3,Z_4)\bigr) \\
    &= \frac{1}{2} \sum_{w,x,y,z=0}^\infty \frac{e^{-2\lambda}}{w!x!y!z!} \bigl| \mathbb{E} \bigl\{ V^{w+z} (\lambda-V)^{x+y} \bigr\} - \mathbb{E} \bigl\{ (V')^{w+z} (\lambda-V')^{x+y} \bigr\} \bigr| \\
    & = \frac{1}{2} \sum_{i,j=0}^\infty e^{-2\lambda} \frac{1}{i! j!} \bigl| \mathbb{E} \bigl\{ (2V)^{i} (2\lambda-2V)^{j} \bigr\} - \mathbb{E} \bigl\{ (2V')^{i} (2\lambda-2V')^{j} \bigr\} \bigr| \\
    &\leq \frac{2^{1/2}}{\pi^{1/4}} \biggl( \frac{e \log r}{L+1} \biggr)^{(L+1)/2}
\end{align*}  
since $L+2 \geq 4 \log r$.  We deduce that with $\rho = \chi/2 = r \delta E_L /2$,
\begin{align*}
    \mathcal{R}(\lceil n_\mathbb{S}/2\rceil,\rho) &\geq \mathcal{R}^\mathrm{Poi}(n_\mathbb{S},\rho) - \sum_{S \in \mathbb{S}} e^{-n_S/12} \\
    &\geq 1 - \frac{r}{2}\cdot \frac{2^{1/2}}{\pi^{1/4}} \biggl( \frac{e \log r}{L+1} \biggr)^{(L+1)/2} - 2e^{-rE_L^2/16} - \sum_{S \in \mathbb{S}} e^{-n_S/12} \\
    &\geq 1 - \frac{r^{1-e\log 2}}{2}\cdot \frac{2^{1/2}}{\pi^{1/4}} - 2e^{-rE_L^2/16} - \sum_{S \in \mathbb{S}} e^{-n_S/12}.
\end{align*}
It follows that there exists a universal constant $r_0>0$ such that when $\min\bigl(r,\min_{S \in \mathbb{S}} n_S\bigr) \geq r_0$ we have $\mathcal{R}(\lceil n_\mathbb{S}/2\rceil,\rho) \geq 1/2$, so
\[
\rho^*(\lceil n_{\mathbb{S}}/2 \rceil) \geq c' \biggl\{ \frac{1}{\log r} \wedge  \Bigl(\frac{r}{(n_{\{1,2\}}\wedge n_{\{1,3\}})\log r}\Bigr)^{1/2} \biggr\}
\]
for some universal constant $c' > 0$.  By reducing $c' > 0$ if necessary, we may therefore conclude that the same lower bound holds for $\rho^*(n_{\mathbb{S}})$.

We now prove that we always have a parametric lower bound, so that the result still holds when $2 \leq r < r_0$. Since $\rho^*$ is increasing in $r$ we assume without loss of generality that $r=2$ and that $n_{\{1,2\}} = \min_{S \in \mathbb{S}} n_S$. Here we use a two-point argument.  For any  $P_\mathbb{S} \in \mathcal{P}_{\mathbb{S}}^{\mathrm{cons}}$ with $p_{1 \bullet \bullet } = p_{\bullet 1 \bullet } = p_{\bullet \bullet 1} =1/2$, we have from~\eqref{Eq:RS3s2} that
\begin{align*}
    R(P_\mathbb{S})= 2 \max \biggl\{0, &\frac{1}{2} -  p_{11\bullet } - p_{\bullet 11} - p_{1 \bullet 1 }, p_{11\bullet } - p_{\bullet 11} + p_{1 \bullet 1 } - \frac{1}{2}, \\
    & p_{11\bullet } + p_{\bullet 11} - p_{1 \bullet 1 } - 1/2, -p_{11\bullet } + p_{\bullet 11} + p_{1 \bullet 1 } - \frac{1}{2} \biggr\}.
\end{align*}
In fact, when $p_{11\bullet } + p_{\bullet 11} + p_{1 \bullet 1 } \leq 1/2$ we have
\[
    R(P_\mathbb{S}) = 1- 2 (p_{11\bullet } + p_{\bullet 11} + p_{1 \bullet 1 }).
\]
Take $p_{\bullet 11} = p_{1 \bullet 1}=1/8$ so that $R(P_\mathbb{S}) = 1/2 - 2p_{11 \bullet }$. We can therefore take $P_\mathbb{S}^{(0)} \in \mathcal{P}_\mathbb{S}^0$ to have $p_{11 \bullet } = 1/4$ and $P_\mathbb{S}^{(1)} \in \mathcal{P}_\mathbb{S}\bigl((32n_{\{1,2\}})^{-1/2}\bigr)$ to have $p_{11 \bullet } = 1/4 - (32n_{\{1,2\}})^{-1/2}$. We now use Pinsker's inequality to calculate that
\begin{align*}
    d_\mathrm{TV}^2 \bigl( (P_\mathbb{S}^{(0)})^{n_\mathbb{S}}, (P_\mathbb{S}^{(1)})^{n_\mathbb{S}} \bigr) &= d_\mathrm{TV}^2 \bigl( (P_{n_{\{1,2\}}}^{(0)})^{n_{\{1,2\}}}, (P_{n_{\{1,2\}}}^{(1)})^{n_{\{1,2\}}} \bigr) \leq \frac{n_{\{1,2\}}}{2}\mathrm{KL}\bigl(P_{n_{\{1,2\}}}^{(0)},P_{n_{\{1,2\}}}^{(1)}\bigr) \\
    &= \frac{n_{\{1,2\}}}{4} \biggl\{ \log \biggl( \frac{1/4}{1/4 - (32n_{\{1,2\}})^{-1/2}} \biggr) + \log \biggl( \frac{1/4}{1/4 + (32n_{\{1,2\}})^{-1/2}} \biggr) \biggr\} \\
    &= \frac{n_{\{1,2\}}}{4} \log \biggl( \frac{1}{1- 1/(2n_{\{1,2\}})} \biggr)
    \leq \frac{1}{4}.
\end{align*}
and it follows that $\rho^*(n_\mathbb{S}) \geq (32\min_{S \in \mathbb{S}} n_S)^{-1/2}$. By considering the different possible orderings of $r$, $\min_{S \in \mathbb{S}} n_S$ and $r_0$, we see that the claimed lower bound holds.
\end{proof}

\begin{proof}[Proof of Proposition~\ref{Prop:SinglePattern}]
Suppose that $P_{\mathbb{S}}^{-J} \in \mathcal{P}_{\mathbb{S}^{-J}}^{\mathrm{cons}}$ and let $S_1,S_2 \in \mathbb{S}$ have $S_1 \cap S_2 \neq \emptyset$.  If neither or both of $S_1$ and $S_2$ are equal to $S_0$, then we have immediately that $P_{S_1}^{S_1 \cap S_2} = P_{S_2}^{S_1 \cap S_2}$.  On the other hand, if $S_1 = S_0$ but $S_2 \neq S_0$, say, then $P_{S_1}^{S_1\cap S_2} = P_{S_1 \cap J^c}^{S_1 \cap J^c \cap S_2} = P_{S_2 \cap J^c}^{S_1 \cap J^c \cap S_2} = P_{S_2}^{S_1\cap S_2}$.  This proves the first part of the proposition.

For the second part, if $f_{\mathbb{S}-J} = (f_S:S \in \mathbb{S}^{-J}) \in \mathcal{G}_{\mathbb{S}^{-J}}^+$, then we can define $f'_{\mathbb{S}} = (f'_S:S \in \mathbb{S})$ by $f'_S := f_S$ for $S \in \mathbb{S} \setminus \{S_0\}$ and $f'_{S_0}(x_J,x_{S_0 \cap J^c}) := f_{S_0 \cap J^c}(x_{S_0 \cap J^c})$.  Then $f'_\mathbb{S} \in \mathcal{G}_{\mathbb{S}}^+$, and
\begin{align*}
R(P_\mathbb{S},f'_{\mathbb{S}}) &= -\frac{1}{|\mathbb{S}|}\sum_{S \in \mathbb{S}} \int_{\mathcal{X}_S} f'_S(x_S) \, dP_S(x_S) \\
&= -\frac{1}{|\mathbb{S}^{-J}|}\sum_{S \in \mathbb{S} \setminus \{S_0\}} \int_{\mathcal{X}_S} f_S(x_S) \, dP_S(x_S) - \frac{1}{|\mathbb{S}^{-J}|}\int_{\mathcal{X}_{S_0}} f_{S_0\cap J^c}(x_{S_0 \cap J^c}) \, dP_{S_0}(x_{S_0}) \\
&= -\frac{1}{|\mathbb{S}^{-J}|}\sum_{S \in \mathbb{S} \setminus \{S_0\}} \int_{\mathcal{X}_S} \! f_S(x_S) \, dP_S(x_S) - \frac{1}{|\mathbb{S}^{-J}|}\int_{\mathcal{X}_{S_0\cap J^c}} \! \! f_{S_0\cap J^c}(x_{S_0 \cap J^c}) \, dP_{S_0 \cap J^c}(x_{S_0 \cap J^c}) \\
&= R(P_\mathbb{S}^{-J},f_{\mathbb{S}^{-J}}).
\end{align*}
It follows that $R(P_\mathbb{S}) \geq R(P_{\mathbb{S}}^{-J})$.  Conversely, suppose that $f_{\mathbb{S}} \in \mathcal{G}_{\mathbb{S}}^+$ is such that $R(P_\mathbb{S},f_{\mathbb{S}}) = R(P_\mathbb{S})$.  Now define $\tilde{f}_{\mathbb{S}} = (\tilde{f}_S:S \in \mathbb{S}^{-J})$ by $\tilde{f}_S := f_S$ for $S \in \mathbb{S} \setminus \{S_0\}$ and $\tilde{f}_{S_0 \cap J^c}(x_{S_0 \cap J^c}) := \inf_{x'_J \in \mathcal{X}_J} f_{S_0}(x'_J,x_{S_0 \cap J^c})$.  Then $\tilde{f}_\mathbb{S} \geq -1$.  Moreover, each $\tilde{f}_S$ is upper semi-continuous: this follows when $S \in \mathbb{S} \setminus \{S_0\}$ because $f_S$ is then upper semi-continuous; on the other hand, for any $x_J' \in \mathcal{X}_J$,
\[
\limsup_{x_{n,S_0 \cap J^c} \rightarrow x_{S_0 \cap J^c}} \tilde{f}_{S_0 \cap J^c}(x_{n,S_0 \cap J^c}) \leq \limsup_{x_{n,S_0 \cap J^c} \rightarrow x_{S_0 \cap J^c}} f_{S_0}(x_J',x_{n,S_0 \cap J^c}) \leq f_{S_0}(x_J',x_{S_0 \cap J^c}).
\]
We deduce that $\limsup_{x_{n,S_0 \cap J^c} \rightarrow x_{S_0 \cap J^c}} \tilde{f}_{S_0 \cap J^c}(x_{n,S_0 \cap J^c}) \leq \tilde{f}_{S_0 \cap J^c}(x_{S_0 \cap J^c})$, as required.  Finally, writing $\mathcal{X}_{-J} := \prod_{j \in [d] \setminus J} \mathcal{X}_j$, we have
\begin{align*}
\inf_{x_{-J} \in \mathcal{X}_{-J}} \sum_{S \in \mathbb{S}^{-J}} \tilde{f}_S(x_S) &= \inf_{x_{-J} \in \mathcal{X}_{-J}}\biggl\{\sum_{S \in \mathbb{S} \setminus \{S_0\}} \tilde{f}_S(x_S) + \tilde{f}_{S_0 \cap J^c}(x_{S_0 \cap J^c})\biggr\} \\
&= \inf_{x_{-J} \in \mathcal{X}_{-J}}  \biggl\{\sum_{S \in \mathbb{S} \setminus \{S_0\}} f_S(x_S) + \inf_{x_J \in \mathcal{X}_J}  f_{S_0}(x_J,x_{S_0 \cap J^c})\biggr\} \\
&= \inf_{x \in \mathcal{X}} \sum_{S \in \mathbb{S}} f_S(x_S) \geq 0.
\end{align*}
Thus $\tilde{f}_\mathbb{S} \in \mathcal{G}_{\mathbb{S}^{-J}}^+$, and $R(P_\mathbb{S}^{-J}) \geq R(P_{\mathbb{S}}^{-J},\tilde{f}_\mathbb{S}) \geq R(P_{\mathbb{S}},f_\mathbb{S}) = R(P_\mathbb{S})$.
\end{proof}

\begin{proof}[Proof of Proposition~\ref{Prop:EveryPattern}]
Any $f_\mathbb{S} \in \mathcal{G}_\mathbb{S}^+$ can be decomposed as $(f_{S|x_J}:x_J \in \mathcal{X}_J,S \in \mathbb{S})$, where $f_{S|x_J} \in \mathcal{G}_{S \cap J^c}$ is defined by $f_{S|x_J}(x_{S \cap J^c}) := f_S(x_J,x_{S \cap J^c})$.  We write $f_{\mathbb{S}|x_J}:= (f_{S|x_J}:S \in \mathbb{S})$.  Moreover, for each $x_J \in \mathcal{X}_J$,
\begin{align*}
\inf_{x_{S \cap J^c} \in \mathcal{X}_{S \cap J^c}} \sum_{S \in \mathbb{S}} f_{S|x_J}(x_{S \cap J^c}) &= \inf_{x_{S \cap J^c} \in \mathcal{X}_{S \cap J^c}} \sum_{S \in \mathbb{S}} f_S(x_J,x_{S \cap J^c}) \\
&\geq \inf_{x_J' \in \mathcal{X}_J,x_{S \cap J^c} \in \mathcal{X}_{S \cap J^c}} \sum_{S \in \mathbb{S}} f_S(x_J',x_{S \cap J^c}) \geq 0,
\end{align*}
so $f_{\mathbb{S}|x_J} \in \mathcal{G}_{\mathbb{S}^{-J}}^+$ for each $x_J \in \mathcal{X}_J$.  It follows that if $\epsilon > 0$, and if $f_\mathbb{S} \in \mathcal{G}_\mathbb{S}^+$ is such that $R(P_\mathbb{S},f_{\mathbb{S}}) \geq R(P_\mathbb{S}) - \epsilon$, then
\begin{align*}
R(P_\mathbb{S}) \leq R(P_\mathbb{S},f_{\mathbb{S}}) + \epsilon &= -\frac{1}{|\mathbb{S}|}\sum_{S \in \mathbb{S}} \int_{\mathcal{X}_S} f_S(x_S) \, dP_S(x_S) + \epsilon \\
&= -\frac{1}{|\mathbb{S}|}\sum_{S \in \mathbb{S}}\int_{\mathcal{X}_J} \int_{\mathcal{X}_{S \cap J^c}} f_{S|x_J}(x_{S \cap J^c}) \, dP_{S|x_J}(x_{S \cap J^c}) \, dP^J(x_J) + \epsilon \\
&= \int_{\mathcal{X}_J} R(P_{\mathbb{S}|x_J},f_{\mathbb{S}|x_J}) \, dP^J(x_J) + \epsilon.
\end{align*}
Since $\epsilon > 0$ was arbitrary, the desired inequality~\eqref{Eq:ConditioningInequality} follows.

Now consider the discrete case where $\mathcal{X}_j = [m_j]$ for some $m_1,\ldots,m_d \in \mathbb{N} \cup \{\infty\}$.  Given any $(f_{\mathbb{S}|x_J}:x_J \in \mathcal{X}_J)$ with  $f_{\mathbb{S}|x_J} \in \mathcal{G}_{\mathbb{S}^{-J}}^+$ for each $x_J \in \mathcal{X}_J$, we can define $f_\mathbb{S} = (f_S:S \in \mathbb{S})$ by $f_S(x_S) := f_{S|x_J}(x_{S \cap J^c})$.  Then $f_S \geq -1$ for all $S \in \mathbb{S}$, each $f_S$ is upper semi-continuous, and
\[
\min_{x \in \mathcal{X}} \sum_{S \in \mathbb{S}} f_S(x_S) = \min_{x_J \in \mathcal{X}_J} \min_{x_{S \cap J^c} \in \mathcal{X}_{S \cap J^c}} \sum_{S \in \mathbb{S}} f_{S|x_J}(x_{S \cap J^c}) \geq 0.
\]
Hence $f_\mathbb{S} \in \mathcal{G}_\mathbb{S}^+$.  Moreover, in this discrete case, maximising $R(P_{\mathbb{S}|x_J},\cdot)$ over $\mathcal{G}_{\mathbb{S}^{-J}}^+$ may be regarded as maximising a continuous function over a closed subset of $[-1,|\mathbb{S}|-1]^{\mathcal{X}_{\mathbb{S}^{-J}}}$ equipped with product topology, and this is a compact set by Tychanov's theorem \citep[e.g.][Theorem~4.42]{folland1999real}.   We may therefore assume that there exists $f_{\mathbb{S}|x_J} \in \mathcal{G}_{\mathbb{S}^{-J}}^+$ such that $R(P_{\mathbb{S}|x_J},f_{\mathbb{S}|x_J}) = R(P_{\mathbb{S}|x_J})$.  Then
\begin{align*}
    R(P_\mathbb{S}) \geq R(P_{\mathbb{S}},f_{\mathbb{S}}) &= -\frac{1}{|\mathbb{S}|} \sum_{S \in \mathbb{S}} \sum_{x_J \in \mathcal{X}_J} \sum_{x_{S \cap J^c \in \mathcal{X}_{S \cap J^c}}} f_S(x_J,x_{S \cap J^c}) p_S(x_J,x_{S \cap J^c}) \\
    &= \sum_{x_J \in \mathcal{X}_J} \biggl\{-\frac{1}{|\mathbb{S}|} \sum_{S \in \mathbb{S}} \sum_{x_{S \cap J^c \in \mathcal{X}_{S \cap J^c}}} f_{S|x_J}(x_{S \cap J^c}) p_{S|x_J}(x_{S \cap J^c})\biggr\}p^J(x_J) \\ 
    &= \sum_{x_J \in \mathcal{X}_J} R(P_{\mathbb{S}|x_J},f_{\mathbb{S}|x_J})p^J(x_J) = \sum_{x_J \in \mathcal{X}_J} R(P_{\mathbb{S}|x_J})p^J(x_J),
\end{align*}
and the desired conclusion follows.
\end{proof}
\begin{proof}[Proof of Proposition~\ref{Prop:CutSet}]
We first establish the lower bound on $R(P_\mathbb{S})$. 
Suppose that $\epsilon \in [0,1]$ is such that $P_\mathbb{S} \in (1-\epsilon) \mathcal{P}_\mathbb{S}^0 + \epsilon \mathcal{P}_\mathbb{S}$.  Then we can find $Q_\mathbb{S}^0 \in \mathcal{P}_\mathbb{S}^0$ and $Q_\mathbb{S} \in \mathcal{P}_\mathbb{S}$ such that $P_\mathbb{S} = (1-\epsilon)Q_\mathbb{S}^0 + \epsilon Q_\mathbb{S}$.  But then $P_{\mathbb{S}_1} := (P_S:S \in \mathbb{S}_1)$ satisfies $P_{\mathbb{S}_1} = (1-\epsilon)Q_{\mathbb{S}_1}^0 + \epsilon Q_{\mathbb{S}_1}$, so $P_{\mathbb{S}_1} \in (1-\epsilon) \mathcal{P}_{\mathbb{S}_1}^0 + \epsilon \mathcal{P}_{\mathbb{S}_1}$.  Hence, by Theorem~\ref{Thm:DualRepresentation} we have $R(P_\mathbb{S}) \geq R(P_{\mathbb{S}_1})$. The same argument applies to show that $R(P_\mathbb{S}) \geq R(P_{\mathbb{S}_2})$, and the lower bound therefore follows.

We now turn to the upper bound.  For $k \in \{1,2\}$, let $I_k := \cup_{S \in \mathbb{S}_k} S$.  From~\eqref{Eq:RPSEquality}, for each $k \in \{1,2\}$ we can find $q_k \in [0,\infty)^{\mathcal{X}_{I_k}}$ that maximises $1_{\mathcal{X}_{I_k}}^T q$ over all $q \in [0,\infty)^{\mathcal{X}_{I_k}}$ that satisfy $\mathbb{A}^k q \leq p_{\mathbb{S}_k}$, where $\mathbb{A}^k := (\mathbb{A}_{(S,y_S),x}^k)_{(S,y_S) \in \mathcal{X}_{\mathbb{S}_k},x \in \mathcal{X}_{I_k}} \in \{0,1\}^{\mathcal{X}_{\mathbb{S}_k} \times \mathcal{X}_{I_k}}$ is given by
\[
    \mathbb{A}_{(S,y_S),x}^k := \mathbbm{1}_{\{x_S = y_S\}}.
\]
Define a measure $Q$ on $\mathcal{X}$ with mass function $q$ given by
\[
    q(x) := \left\{ \begin{array}{ll} \min\bigl\{ q_1^J(x_J), q_2^J(x_J)\bigr\} \cdot  \frac{q_1(x_{I_1})}{q_1^J(x_J)} \cdot  \frac{q_2(x_{I_2})}{q_2^J(x_J)} & \mbox{if $\min\bigl\{ q_1^J(x_J), q_2^J(x_J)\bigr\} > 0$} \\ 0 & \mbox{otherwise.} \end{array} \right.
\]
Then whenever $\min\{ q_1^J(x_J), q_2^J(x_J)\} > 0$, we have 
\begin{align*}
q^J(x_J) &= \sum_{x_{J^c \cap I_1} \in \mathcal{X}_{J^c \cap I_1}} \sum_{x_{J^c \cap I_2} \in \mathcal{X}_{J^c \cap I_2}} \min\{ q_1^J(x_J), q_2^J(x_J)\} \frac{q_1(x_{I_1})}{q_1^J(x_J)} \cdot  \frac{q_2(x_{I_2})}{q_2^J(x_J)} \\
&=    \min\{ q_1^J(x_J), q_2^J(x_J)\} \sum_{x_{J^c \cap I_1} \in \mathcal{X}_{J^c \cap I_1}} \frac{q_1(x_{I_1})}{q_1^J(x_J)} \cdot \sum_{x_{J^c \cap I_2} \in \mathcal{X}_{J^c \cap I_2}}  \frac{q_2(x_{I_2})}{q_2^J(x_J)} \\
&= \min\{ q_1^J(x_J), q_2^J(x_J) \} = \min\bigl\{ (\mathbb{A}^1q_1)_{(J,x_J)},(\mathbb{A}^2q_2)_{(J,x_J)}\bigr\}  \leq (p_\mathbb{S})_{(J,x_J)} = p_J(x_J).
\end{align*}
On the other hand, if $\min\{ q_1^J(x_J), q_2^J(x_J)\} = 0$, then $q^J(x_J) = 0 \leq p_S(x_S)$.  Further, whenever $q_k^J(x_J) > 0$, we have for $k \in \{1,2\}$ and any $S \in \mathbb{S}_k \setminus \{J\}$ that
\begin{align*}
    q^S(x_S) &= \sum_{x_{J \cap S^c} \in \mathcal{X}_{J \cap S^c}} \sum_{x_{J^c \cap I_k} \in \mathcal{X}_{J^c \cap I_k}} \min\{ q_1^J(x_J), q_2^J(x_J)\} \frac{q_k(x_{I_k})}{q_k^J(x_J)} \\
    &\leq \sum_{x_{J \cap S^c} \in \mathcal{X}_{J \cap S^c}} \sum_{x_{J^c \cap I_k} \in \mathcal{X}_{J^c \cap I_k}} q_k(x_{I_k}) = q_k^S(x_S) = (\mathbb{A}^kq_k)_{(S,x_S)} \leq (p_\mathbb{S})_{(S,x_S)} = p_S(x_S).
\end{align*}
Finally, if $q_k^J(x_J) = 0$, then $q^S(x_S) = 0 \leq p_S(x_S)$.  It follows that $\mathbb{A} q \leq p_\mathbb{S}$, where $\mathbb{A} := (\mathbb{A}_{(S,y_S),x})_{(S,y_S) \in \mathcal{X}_{\mathbb{S}},x \in \mathcal{X}} \in \{0,1\}^{\mathcal{X}_{\mathbb{S}} \times \mathcal{X}}$ is given by~\eqref{Eq:A}. Thus, from~\eqref{Eq:RPSEquality},
\begin{align*}
    R(P_\mathbb{S}) &\leq 1 - \sum_{x \in \mathcal{X}} q(x) = 1 - \sum_{x_J \in \mathcal{X}_J} \min\bigl\{ q_1^J(x_J), q_2^J(x_J) \bigr\} \\
    & = \sum_{x_J \in \mathcal{X}_J} \max\bigl\{ p_J(x_J) - q_1^J(x_J), p_J(x_J) - q_2^J(x_J) \bigr\} \\
    & \leq \sum_{x_J \in \mathcal{X}_J} \bigl\{ p_J(x_J) - q_1^J(x_J) + p_J(x_J) - q_2^J(x_J) \bigr\} \\
    & = 1 - \sum_{x_{I_1} \in \mathcal{X}_{I_1}} q_1(x_{I_1}) + 1 -\sum_{x_{I_2} \in \mathcal{X}_{I_2}} q_2(x_{I_2}) = R(P_{\mathbb{S}_1}) + R(P_{\mathbb{S}_2}),
\end{align*}
as required.
\end{proof}

\begin{proof}[Proof of Proposition~\ref{Prop:FullDim}]
Suppose that there exist $f_\mathbb{S} \in \mathbb{R}^{\mathcal{X}_\mathbb{S}}$ and $c \in \mathbb{R}$ such that $f_\mathbb{S}^T p_\mathbb{S} = c$ for all $p_\mathbb{S} \in \mathcal{P}_\mathbb{S}^0$. We will show that we must also have $f_\mathbb{S}^T p_\mathbb{S}=c$ for all $p_\mathbb{S} \in \mathcal{P}_\mathbb{S}^\mathrm{cons}$. In fact, by replacing $f_\mathbb{S}$ by $f_\mathbb{S} - (c/|\mathbb{S}|) 1_{\mathcal{X}_\mathbb{S}}$, we may assume without loss of generality that $c=0$. 

In this proof we emphasise the dependence of $\mathbb{A}$ on $\mathbb{S}$ by writing $\mathbb{A}_\mathbb{S}$.  Since $(\mathbb{A}_\mathbb{S}^T f_\mathbb{S})^T p =0$ for all $p \in [0,1]^\mathcal{X}$ with $1_\mathcal{X}^Tp = 1$, we must have that $\mathbb{A}_\mathbb{S}^T f_\mathbb{S} = 0$. We will use induction on $|\mathbb{S}|$ to deduce  that $f_\mathbb{S}^T p_\mathbb{S} = 0$ for all $p_\mathbb{S} \in \mathcal{P}_\mathbb{S}^\mathrm{cons}$. When $|\mathbb{S}|=1$, we have that if $\mathbb{A}_\mathbb{S}^T f_\mathbb{S} = 0$, then $f_\mathbb{S} = 0$, so $f_\mathbb{S}^T p_\mathbb{S} = 0$ for all $p_\mathbb{S} \in \mathcal{P}_\mathbb{S}^\mathrm{cons}$.  As  our induction hypothesis, suppose that whenever $|\mathbb{S}| \leq m$ and $f_\mathbb{S} \in \mathbb{R}^{\mathcal{X}_\mathbb{S}}$ satisfies $\mathbb{A}_\mathbb{S}^T f_\mathbb{S} = 0 $, we must have $f_\mathbb{S}^T p_\mathbb{S} = 0$ for all $p_\mathbb{S} \in \mathcal{P}_\mathbb{S}^\mathrm{cons}$.

Let $\mathbb{S}$ be given with $|\mathbb{S}|=m+1$, suppose that $f_\mathbb{S} \in \mathbb{R}^{\mathcal{X}_\mathbb{S}}$ satisfies $\mathbb{A}_\mathbb{S}^T f_\mathbb{S} = 0$, and let $p_\mathbb{S} \in \mathcal{P}_\mathbb{S}^\mathrm{cons}$ be arbitrary. Without loss of generality, we may assume that $\mathcal{X}_j = [m_j]$ for $j \in [d]$ for some $m_1,\ldots,m_d \in \mathbb{N}$.  Fixing $S_0 \in \mathbb{S}$, we have
\[
    f_{S_0}(x_{S_0}) = - \sum_{S \in \mathbb{S} \setminus \{S_0\}} f_S(x_{S_0 \cap S}, 1_{S_0^c \cap S})
\]
for all $x_{S_0} \in \mathcal{X}_{S_0}$, since $(\mathbb{A}_\mathbb{S}^T f_\mathbb{S})_{(x_{S_0},1_{[d] \setminus S_0})}=0$. Using the notational convention that $\sum_{x_{S_1 \cap S_2} \in \mathcal{X}_{S_1 \cap S_2}} p_{S_1}^{S_1 \cap S_2}(x_{S_1 \cap S_2}) = 1$ whenever $S_1 \cap S_2 = \emptyset$, we may therefore write
\begin{align}
\label{Eq:Reduction}
    f_\mathbb{S}^T p_\mathbb{S} &= \sum_{x_{S_0} \in \mathcal{X}_{S_0}} f_{S_0}(x_{S_0}) p_{S_0}(x_{S_0}) + \sum_{S \in \mathbb{S} \setminus \{S_0\}} \sum_{x_S \in \mathcal{X}_S} f_S(x_S) p_S(x_S) \nonumber \\
    & = \sum_{S \in \mathbb{S} \setminus \{S_0\}} \biggl\{ \sum_{x_S \in \mathcal{X}_S} f_S(x_S) p_S(x_S) - \sum_{x_{S_0 \cap S} \in \mathcal{X}_{S_0 \cap S}} f_S(x_{S_0 \cap S}, 1_{S_0^c \cap S}) p_{S_0}^{S_0 \cap S}(x_{S_0 \cap S}) \biggr\} \nonumber\\
    & = \sum_{S \in \mathbb{S} \setminus \{S_0\}} \biggl\{ \sum_{x_S \in \mathcal{X}_S} f_S(x_S) p_S(x_S) - \sum_{x_{S_0 \cap S} \in \mathcal{X}_{S_0 \cap S}} f_S(x_{S_0 \cap S}, 1_{S_0^c \cap S}) p_{S}^{S_0 \cap S}(x_{S_0 \cap S}) \biggr\} \nonumber\\
    & = \sum_{S \in \mathbb{S} \setminus \{S_0\}} \sum_{x_S \in \mathcal{X}_S} p_S(x_S) \bigl\{ f_S(x_S) - f_S(x_{S_0 \cap S}, 1_{S_0^c \cap S}) \bigr\} = (\bar{f}_{\mathbb{S} \setminus \{S_0\}})^T p_{\mathbb{S} \setminus \{S_0\}},
\end{align}
where we define $\bar{f}_{\mathbb{S} \setminus \{S_0\}} \in \mathbb{R}^{\mathcal{X}_{\mathbb{S} \setminus \{S_0\}}}$ by $\bar{f}_S(x_S) := f_S(x_S) - f_S(x_{S_0 \cap S}, 1_{S_0^c \cap S})$ for $S \in \mathbb{S} \setminus \{S_0\}$ and $x_S \in \mathcal{X}_S$, and where $p_{\mathbb{S} \setminus \{S_0\}} := (p_S:S \in \mathbb{S} \setminus \{S_0\})$. For any $x \in \mathcal{X}$, we have
\begin{align*}
    &(\mathbb{A}_{\mathbb{S} \setminus \{S_0\}}^T \bar{f}_{\mathbb{S} \setminus \{S_0\}})_x =  \sum_{S \in \mathbb{S} \setminus \{S_0\}} \bar{f}_S(x_S) = \sum_{S \in \mathbb{S} \setminus \{S_0\}} f_S(x_S) - \sum_{S \in \mathbb{S} \setminus \{S_0\}} f_S(x_{S_0 \cap S}, 1_{S_0^c \cap S}) \\
    &= (\mathbb{A}_\mathbb{S}^T f_\mathbb{S})_x - f_{S_0}(x_{S_0}) - \bigl\{ (\mathbb{A}_\mathbb{S}^T f_\mathbb{S})_{(x_{S_0},1_{[d] \setminus \{S_0\}})} - f_{S_0}(x_{S_0}) \bigr\} = f_{S_0}(x_{S_0}) - f_{S_0}(x_{S_0}) = 0.
\end{align*}
Since $p_{\mathbb{S} \setminus \{S_0\}}$ satisfies the consistency constraints associated with $\mathbb{S}$, we see by $\eqref{Eq:Reduction}$ and our induction hypothesis that
\[
    f_\mathbb{S}^T p_\mathbb{S} = (\bar{f}_{\mathbb{S} \setminus \{S_0\}})^T p_{\mathbb{S} \setminus \{S_0\}}=0,
\]
as required.
\end{proof}

\begin{prop}
\label{Prop:4dexampleAnalytic}
Suppose that $\mathbb{S}=\bigl\{\{1,2\},\{2,3\},\{3,4\},\{1,4\}\bigr\}$, $\mathcal{X}_1=[r]$ for some $r \in \mathbb{N}$, and $\mathcal{X}_2=\mathcal{X}_3=\mathcal{X}_4=[2]$. Then
\begin{equation}
\label{Eq:ChainStatementGeneral}
    R(P_\mathbb{S}) = 2 \max_{k, \ell \in [2]} \biggl\{ p_{\bullet \bullet k \ell } - p_{\bullet 2 k \bullet } - \sum_{i=1}^r \min(p_{i1\bullet \bullet }, p_{i \bullet \bullet \ell }) \biggr\}_+.
\end{equation}
\end{prop}
\begin{proof}[Proof of Proposition~\ref{Prop:4dexampleAnalytic}]
We first prove that $R(P_{\mathbb{S}})$ is bounded below by the quantity on the right-hand side of~\eqref{Eq:ChainStatementGeneral}, before proving the corresponding upper bound.  First, we always have $R(P_\mathbb{S}) \geq 0$. Now, define $f_\mathbb{S} \in \mathcal{G}_\mathbb{S}$ by setting, for $i \in [r]$,
\[
    (f_{i1 \bullet \bullet}, f_{i 2 \bullet \bullet }, f_{i \bullet \bullet 1}, f_{i \bullet \bullet 2 } ) := \left\{ \begin{array}{ll} (3,-1,-1,3) & \mbox{if } p_{i 1 \bullet \bullet } \leq p_{i \bullet \bullet 1} \\ (-1,3,3,-1) & \mbox{otherwise} \end{array} \right.,
\]
$f_{\bullet \bullet 1 2}= f_{\bullet \bullet 2 1} = f_{\bullet 1 2 \bullet} = f_{\bullet 2 1 \bullet } = 3$ and $f_{\bullet \bullet 1 1}= f_{\bullet \bullet 2 2} = f_{\bullet 1 1 \bullet} = f_{\bullet 2 2 \bullet } = -1$. It is straightforward to check that $f_\mathbb{S} \in \mathcal{G}_\mathbb{S}^+$.  Now
\begin{align*}
    &R(P_\mathbb{S}, f_\mathbb{S}) = - \frac{1}{4} \sum_{i=1}^r \biggl( \sum_{j=1}^2 p_{ij \bullet \bullet} f_{ij \bullet \bullet} + \sum_{\ell=1}^2 p_{i \bullet \bullet \ell} f_{i \bullet \bullet \ell} \biggr) - \frac{1}{4} \sum_{j,k=1}^2 p_{\bullet jk \bullet} f_{\bullet jk \bullet} - \frac{1}{4} \sum_{k,\ell=1}^2 p_{\bullet \bullet k \ell} f_{\bullet \bullet k \ell} \\
    &= -\frac{1}{4} \sum_{i=1}^r \bigl\{ 3 \min(p_{i 1 \bullet \bullet}, p_{i \bullet \bullet 1}) - \max(p_{i 2 \bullet \bullet }, p_{i \bullet \bullet 2}) - \max(p_{i 1 \bullet \bullet }, p_{i \bullet \bullet 1}) + 3 \min(p_{i 2 \bullet \bullet }, p_{i \bullet \bullet 2}) \} \\
    & \hspace{30pt} - \frac{1}{4}\{ 3(p_{\bullet 1 2 \bullet } + p_{\bullet 2 1 \bullet }) - (p_{\bullet 11 \bullet } + p_{\bullet 22 \bullet }) \} - \frac{1}{4}\{ 3(p_{\bullet \bullet 1 2} + p_{\bullet \bullet 2 1}) - (p_{\bullet \bullet 11} + p_{\bullet \bullet 22}) \} \\
    & = - \frac{1}{4} \sum_{i=1}^r \bigl\{ 4 \min(p_{i 1 \bullet \bullet}, p_{i \bullet \bullet 1}) - 4 \max(p_{i 1 \bullet \bullet}, p_{i \bullet \bullet 1}) + 2 p_{i \bullet \bullet \bullet } \bigr\} \\
    & \hspace{30pt}- \frac{1}{4} \bigl\{ 3(2 p_{\bullet 2 1 \bullet } - p_{\bullet 2 \bullet \bullet } - p_{\bullet \bullet 1 \bullet } + p_{\bullet \bullet \bullet \bullet }) - (p_{\bullet \bullet 1 \bullet } - 2p_{\bullet 2 1 \bullet } + p_{\bullet 2 \bullet \bullet } ) \bigr\} \\
    &  \hspace{30pt} - \frac{1}{4}  \bigl\{  3 (p_{\bullet \bullet 1 \bullet } - 2p_{\bullet \bullet 11} + p_{\bullet \bullet \bullet 1 }) - (2 p_{\bullet \bullet 11} - p_{\bullet \bullet 1 \bullet } - p_{\bullet \bullet \bullet 1} + p_{\bullet \bullet \bullet \bullet }) \bigr\} \\
    & = - \frac{1}{4} \sum_{i=1}^r \{ 8 \min(p_{i 1 \bullet \bullet}, p_{i \bullet \bullet 1}) - 4p_{i1 \bullet \bullet } - 4p_{i \bullet \bullet 1} + 2 p_{i \bullet \bullet \bullet } \}  \\
    & \hspace{30pt} - \frac{1}{4} \bigl( 8 p_{\bullet 2 1 \bullet } - 4 p_{\bullet 2 \bullet \bullet } -4 p_{\bullet \bullet 1 \bullet } + 3p_{\bullet \bullet \bullet \bullet } \bigr) - \frac{1}{4} \bigl( -8 p_{\bullet \bullet 1 1} +4 p_{\bullet \bullet 1 \bullet } + 4 p_{\bullet \bullet \bullet 1} - p_{\bullet \bullet \bullet \bullet } \bigr) \\
    & = 2 \biggl\{ p_{\bullet \bullet 1 1} - p_{\bullet 21 \bullet } - \sum_{i=1}^r \min(p_{i1\bullet \bullet }, p_{i \bullet \bullet 1}) \biggr\}.
\end{align*}
Since $R(P_\mathbb{S}) \geq R(P_{\mathbb{S}},f_{\mathbb{S}})$, this completes the lower bound in the case that $(k,\ell)=(1,1)$ is the maximiser in~\eqref{Eq:ChainStatementGeneral}.  The other three cases follow by almost identical arguments by choosing different $f_\mathbb{S} \in \mathcal{G}_{\mathbb{S}}^+$ appropriately. We now turn to the upper bound, which we will prove by using the dual formulation
\[
    1 - R(P_\mathbb{S}) = \max \biggl\{ \sum_{i=1}^r \sum_{j,k,\ell =1}^2 q_{ijk \ell} : q \in [0,\infty)^\mathcal{X}, \mathbb{A}q \leq p_\mathbb{S} \biggr\}.
\]
Write $A:=\{ i \in [r] : p_{i1\bullet \bullet } \leq p_{i \bullet \bullet 1} \}$ and suppose that
\begin{equation}
\label{Eq:Incompatible}
    p_{\bullet \bullet 1 1} - p_{\bullet 2 1 \bullet } - p_{A 1 \bullet \bullet } - p_{A^c \bullet \bullet 1} \geq 0,
\end{equation}
where we note that an alternative expression for the left-hand side of~\eqref{Eq:Incompatible} is given by $p_{\bullet \bullet 1 1} - p_{\bullet 2 1 \bullet } - \sum_{i=1}^r \min(p_{i1\bullet \bullet }, p_{i \bullet \bullet 1})$.  For $i \in [r]$, consider the choices
\begin{alignat*}{5}
q_{i111} &&= \min(p_{i1 \bullet \bullet }, p_{i \bullet \bullet 1}), \ \ q_{i112} &&= \frac{(p_{i 1 \bullet \bullet } - p_{i \bullet \bullet 1})_+}{p_{A^c 1 \bullet \bullet} - p_{A^c \bullet \bullet 1}}p_{\bullet \bullet 1 2}, \ \  q_{i121}&&=0, \ \ q_{i122} &&= \frac{(p_{i 1 \bullet \bullet } - p_{i \bullet \bullet 1})_+}{p_{A^c 1 \bullet \bullet} - p_{A^c \bullet \bullet 1}}p_{\bullet 1 2 \bullet }, \\
q_{i222} &&= \min(p_{i\bullet \bullet 2}, p_{i2 \bullet \bullet }), \ \ q_{i211}&&=\frac{(p_{i \bullet \bullet 1} - p_{i 1 \bullet \bullet })_+}{p_{A \bullet \bullet 1} - p_{A 1 \bullet \bullet}}p_{\bullet 21 \bullet }, \ \ q_{i212}&&=0, \ \  q_{i221}&&=\frac{(p_{i \bullet \bullet 1} - p_{i1 \bullet \bullet})_+}{p_{A \bullet \bullet 1} - p_{A 1 \bullet \bullet}}p_{\bullet \bullet 2 1 },
\end{alignat*}
where we interpret $q_{i211} = q_{1221} = 0$ if $p_{A \bullet \bullet 1} = p_{A1\bullet \bullet}$.  It is clear that $q \in [0,\infty)^\mathcal{X}$, and we now check that $\mathbb{A}q \leq p_\mathbb{S}$. First,
\begin{align*}
    \sum_{k,\ell=1}^2 q_{i1 k \ell} &= \min(p_{i1\bullet \bullet }, p_{i \bullet \bullet 1}) + \frac{(p_{i 1 \bullet \bullet } - p_{i \bullet \bullet 1})_+}{p_{A^c 1 \bullet \bullet} - p_{A^c \bullet \bullet 1}}(p_{\bullet \bullet 1 2} + p_{\bullet 1 2 \bullet }) \\
    & = \min(p_{i1\bullet \bullet }, p_{i \bullet \bullet 1}) + \frac{(p_{i 1 \bullet \bullet } - p_{i \bullet \bullet 1})_+}{p_{A^c 1 \bullet \bullet} - p_{A^c \bullet \bullet 1}}(p_{\bullet 2 1 \bullet } - p_{\bullet \bullet 11} + p_{\bullet 1 \bullet \bullet }) \\
    & \leq \min(p_{i1\bullet \bullet }, p_{i \bullet \bullet 1}) +  (p_{i 1 \bullet \bullet } - p_{i \bullet \bullet 1})_+ = p_{i1\bullet \bullet },
\end{align*}
for each $i \in [r]$, where the inequality follows from~\eqref{Eq:Incompatible}. It is very similar to check that $\sum_{k,\ell=1}^2 q_{i2k\ell} \leq p_{i2 \bullet \bullet }$, that $\sum_{j,k=1}^2 q_{ijk1} \leq p_{i \bullet \bullet 1}$, and that $\sum_{j,k=1}^2 q_{ijk2} \leq p_{i \bullet \bullet 2}$ for each $i \in [r]$. Now
\begin{align*}
    \sum_{i=1}^r \sum_{\ell=1}^2 q_{i11\ell} &= \sum_{i=1}^r \biggl\{ \min(p_{i1\bullet \bullet }, p_{i \bullet \bullet 1}) + \frac{(p_{i 1 \bullet \bullet } - p_{i \bullet \bullet 1})_+}{p_{A^c 1 \bullet \bullet} - p_{A^c \bullet \bullet 1}}p_{\bullet \bullet 1 2} \biggr\} \\
    & = p_{A1\bullet \bullet} + p_{A^c \bullet \bullet 1} + p_{\bullet \bullet 12} \leq p_{\bullet \bullet 11} - p_{\bullet 21 \bullet } + p_{\bullet \bullet 1 2} = p_{\bullet 1 1 \bullet },
\end{align*}
where the inequality again follows from~\eqref{Eq:Incompatible}. It is very similar to check that $\sum_{i=1}^r \sum_{\ell=1}^2 q_{i22\ell} \leq p_{\bullet 22 \bullet}$, that $\sum_{i=1}^r \sum_{j=1}^2 q_{ij11} \leq p_{\bullet \bullet 11}$, and that $\sum_{i=1}^r \sum_{j=1}^2 q_{ij22} \leq p_{\bullet \bullet 22}$. Finally, it is straightforward to see using similar arguments that $\sum_{i=1}^r \sum_{\ell=1}^2 q_{i12 \ell} = p_{\bullet 1 2 \bullet }$, that $\sum_{i=1}^r \sum_{\ell=1}^2 q_{i 21 \ell}= p_{\bullet 21 \bullet}$, that $\sum_{i=1}^r \sum_{j=1}^2 q_{ij 21}= p_{\bullet \bullet 21}$, and that $\sum_{i=1}^r \sum_{j=1}^2 q_{ij 12}= p_{\bullet \bullet 12}$. Now that we have seen that $q$ satisfies the necessary constraints, we calculate that
\begin{align*}
    R(P_\mathbb{S}) &\leq 1 -\sum_{i=1}^r \sum_{j,k,\ell=1}^2 q_{ijk \ell} \\
    &= 1 -( p_{A1 \bullet \bullet } + p_{A^c \bullet \bullet 1} + p_{\bullet \bullet 12} + p_{\bullet 12 \bullet } + p_{\bullet 21 \bullet } + p_{\bullet \bullet 21} + p_{A^c 2 \bullet \bullet } + p_{A \bullet \bullet 2} )\\
    & = 2 \biggl\{ p_{\bullet \bullet 11} - p_{\bullet 21 \bullet } - \sum_{i=1}^r \min(p_{i1\bullet \bullet }, p_{i \bullet \bullet 1}) \biggr\}.
\end{align*}
This deals with the case where $(k,\ell)=(1,1)$ gives the maximiser in~\eqref{Eq:ChainStatementGeneral} and where the right-hand side of~\eqref{Eq:ChainStatementGeneral} is positive, as in this case~\eqref{Eq:Incompatible} must hold. The other cases  follow by very similar arguments, and this completes the proof.
\end{proof}

\begin{proof}[Proof of Theorem~\ref{Thm:ContinuousUpperBound}]
Given $S \in \mathbb{S}$ and $k = (k_1,\ldots,k_d) \in \mathcal{K}_h$, we can define a discretised version $Q_S$ of $P_S$ with mass function
\[
q_S(k_S) := P_S\biggl(\prod_{j \in {S \in [d_0]}} I_{h_j,k_j} \times \prod_{j \in S \cap ([d] \setminus [d_0])} \{k_j\}\biggr).
\]
Then $R_h(\hat{P}_\mathbb{S}) \stackrel{d}{=} R(\hat{Q}_\mathbb{S})$, where $(Y_{S,i}:S \in \mathbb{S},i \in [n_S])$ are independent with $Y_{S,i} \sim Q_S$ for $i \in [n_S]$, and $\hat{Q}_\mathbb{S}$ denotes their empirical distribution.  Moreover, if $R(P_\mathbb{S}) = 0$, then $P_\mathbb{S} \in \mathcal{P}_\mathbb{S}^0$ so there exists a distribution $P$ on $\mathcal{X}$ whose marginal distribution on $\mathcal{X}_S$ is $P_S$, for each $S \in \mathbb{S}$.  The discretised version $Q$ of $P$ with mass function
\begin{equation}
\label{Eq:PtoQ}
q(k) := P\biggl(\prod_{j=1}^{d_0} I_{h_j,k_j} \times \prod_{j=d_0+1}^d \{k_j\}\biggr)
\end{equation}
on $\mathcal{K}_h$ then satisfies the condition that its marginal on $(\mathcal{K}_h)_S$ is $q_S$, for each $S \in \mathbb{S}$.  It follows that $Q$ is compatible, i.e.~$R(Q_\mathbb{S}) = 0$.  The Type I error probability control follows from this and the first parts of Theorems~\ref{Prop:DiscreteTest1} and~\ref{Prop:DiscreteTest}.

For the second claim, given $\epsilon > 0$, find $f_\mathbb{S} \in \mathcal{G}_\mathbb{S}^+$ with $R(P_\mathbb{S},f_\mathbb{S}) \geq R(P_\mathbb{S}) - \epsilon$.  As in the proof of Proposition~\ref{Prop:L1Projection}, we may assume without loss of generality that $f_S \leq |\mathbb{S}|-1$ for all $S \in \mathbb{S}$.  Now define $f_{\mathbb{S},h} = (f_{S,h}:S \in \mathbb{S})$ by
\[
f_{S,h}(x_{S\cap [d_0]},x_{S \cap ([d] \setminus [d_0])}) := \frac{\int_{\prod_{j \in S \cap [d_0]} I_{h_j,k_j}} f_S(x_{S\cap [d_0]}',x_{S \cap ([d] \setminus [d_0])}) \, dx_{S\cap [d_0]}'}{\int_{\prod_{j \in S \cap [d_0]} I_{h_j,k_j}}  \, dx_{S\cap [d_0]}'}
\]
for $(x_{S\cap [d_0]},x_{S \cap ([d] \setminus [d_0])}) \in \mathcal{X}_S$ with $x_{S\cap [d_0]} \in \prod_{j \in S \cap [d_0]} I_{h_j,k_j}$. Each $f_{S,h}$ is then clearly piecewise constants on the appropriate sets, and is bounded below by $-1$. To check the other constraints of $\mathcal{G}_{\mathbb{S},h}^+$, let $x\in\mathcal{X}$ be given and let $U$ be uniformly distributed on the part of the partition of $[0,1)^{S \cap [d_0]}$ that contains $x_{[d_0]}$. We have that
\begin{align*}
    \sum_{S \in \mathbb{S}} f_{S,h}(x_S) = \sum_{S \in \mathbb{S}} \mathbb{E} \bigl\{ f_S(U_{S \cap [d_0]}, x_{S \cap ([d] \setminus [d_0])}) \bigr\} = \mathbb{E} \biggl\{ \sum_{S \in \mathbb{S}} f_S(U_{S \cap [d_0]}, x_{S \cap ([d] \setminus [d_0])}) \biggr\} \geq 0,
\end{align*}
and thus indeed $f_{\mathbb{S},h} \in \mathcal{G}_{\mathbb{S},h}^+$. Now,
\begin{align*}
R(P_\mathbb{S},f_{\mathbb{S},h}) &= -\frac{1}{|\mathbb{S}|}\sum_{S \in \mathbb{S}} \sum_{k \in (\mathcal{K}_h)_S} \frac{\int_{\prod_{j \in S \cap [d_0]} I_{h_j,k_j}} f_S(x_{S\cap [d_0]}',k_{S \cap ([d] \setminus [d_0])}) \, dx_{S\cap [d_0]}'}{\int_{\prod_{j \in S \cap [d_0]} I_{h_j,k_j}}  \, dx_{S\cap [d_0]}'} \\
&\hspace{6cm}\times P_S\biggl(\prod_{j \in S \cap [d_0]} I_{h_j,k_j} \times \prod_{j \in S \cap ([d] \setminus [d_0])} \{k_j\}\biggr) \\
& \geq -\frac{1}{|\mathbb{S}|}\sum_{S \in \mathbb{S}} \sum_{k \in (\mathcal{K}_h)_S} \int_{\prod_{j \in S \cap [d_0]} I_{h_j,k_j}} f_S(x_{S\cap [d_0]}',k_{S \cap ([d] \setminus [d_0])})  \, dP_S(x_{S\cap [d_0]}', k_{S \cap ([d] \setminus [d_0])}) \\
&\hspace{0.5cm} - \frac{L(|\mathbb{S}|-1)}{|\mathbb{S}|} \biggl( \sum_{j=1}^{d_0} h_j^{r_j} \biggr) \sum_{S \in \mathbb{S}} \sum_{k \in (\mathcal{K}_h)_S} p_S^{S \cap ([d]\setminus [d_0])}(k_{S \cap ([d] \setminus [d_0])}) \int_{\prod_{j \in S \cap [d_0]} I_{h_j,k_j}} dx_{S\cap [d_0]}' \\
& = R(P_\mathbb{S}, f_\mathbb{S}) - L(|\mathbb{S}|-1) \sum_{j=1}^{d_0} h_j^{r_j} \geq R(P_\mathbb{S}) - \epsilon - L(|\mathbb{S}|-1) \sum_{j=1}^{d_0} h_j^{r_j}. 
\end{align*}
Since $\epsilon > 0$ was arbitrary, we deduce that
\begin{equation}
\label{Eq:RhLowerBound}
R_h(P_{\mathbb{S}}) \geq R(P_\mathbb{S}) - L(|\mathbb{S}|-1) \sum_{j=1}^{d_0} h_j^{r_j}. 
\end{equation}
The completion of the argument is now very similar to the first part of the theorem: we define the discretised version $Q_S$ of $P_S$ via~\eqref{Eq:PtoQ}.  Note again that $R_h(\hat{P}_\mathbb{S}) \stackrel{d}{=} R(\hat{Q}_\mathbb{S})$, where $(Y_{S,i}:S \in \mathbb{S},i \in [n_S])$ are independent with $Y_{S,i} \sim Q_S$ for $i \in [n_S]$, and $\hat{Q}_\mathbb{S}$ denotes their empirical distribution.  Since $R(Q_\mathbb{S}) = R_h(P_\mathbb{S})$, the result follows from~\eqref{Eq:RhLowerBound} together with the second parts of Theorems~\ref{Prop:DiscreteTest1} and~\ref{Prop:DiscreteTest}.
\end{proof}

\noindent \textbf{Acknowledgements}: The first author was supported by Engineering and Physical Sciences Research Council (EPSRC) New Investigator Award EP/W016117/1.  The second author was supported by EPSRC grants EP/P031447/1 and EP/N031938/1, as well as European Research Council
Advanced Grant 101019498.

{
\bibliographystyle{custom}
\bibliography{bib}
}

\section{Glossary of topological definitions}
\label{Sec:Glossary}

A topological space $\mathcal{X}$ is said to be \emph{completely regular} if for every closed set $B \subseteq \mathcal{X}$ and and every $x_0 \in \mathcal{X} \setminus B$, there exists a bounded continuous function $f : \mathcal{X} \rightarrow \mathbb{R}$ such that $f(x_0) = 1$ and $f(x)= 0$ for all $x \in B$.  We say $\mathcal{X}$ is \emph{Hausdorff} if, given any distinct $x, y \in \mathcal{X}$, there exist open sets $U \subseteq \mathcal{X}$ containing $x$ and $V \subseteq \mathcal{X}$ such that $U \cap V = \emptyset$.  We say a subset of $\mathcal{X}$ is \emph{$\sigma$-compact} if it is countable union of compact sets.  Given a Borel subset $E$ of $\mathcal{X}$, we say a Borel measure $\mu$ on $\mathcal{X}$ is \emph{outer regular} on $E$ if
\[
\mu(E) = \inf\{\mu(U): U \supseteq E, U \text{ open}\}
\]
and \emph{inner regular} on $E$ if
\[
\mu(E) = \sup\{\mu(K): K \subseteq E, K \text{ compact}\}.
\]
We say $\mu$ is a \emph{Radon} measure if it is outer regular on all Borel sets, inner regular on all open sets, and finite on all compact sets.

If $\mathcal{T}$ is a topology on $\mathcal{X}$, a \emph{neighbourhood base} for $\mathcal{T}$ at $x \in \mathcal{X}$ is a family $\mathcal{N} \subseteq \mathcal{T}$ such that $x \in V$ for all $V \in \mathcal{N}$ and, whenever $U \in \mathcal{T}$ and $x \in U$, there exists $V \in \mathcal{N}$ such that $V \subseteq U$.  A \emph{base} for $\mathcal{T}$ is a family $\mathcal{B} \subseteq \mathcal{T}$ that contains a neighbourhood base for $\mathcal{T}$ at each $x \in \mathcal{X}$.  We say $\mathcal{X}$ is \emph{second countable} if it has a countable base.

\end{document}